\documentclass[12pt]{amsart}
\usepackage{amsfonts,amsthm,latexsym,amsmath,amssymb,amscd,amsmath,epsf,a4wide, verbatim } 
\usepackage{tikz}
\usepackage[hidelinks]{hyperref}
\usepackage{subcaption}

\theoremstyle{plain}
\newtheorem{theorem}{\bf Theorem}[section]
\newtheorem{lemma}[theorem]{\bf Lemma}
\newtheorem{proposition}[theorem]{\bf Proposition}
\newtheorem{corollary}[theorem]{\bf Corollary}

\newtheorem{question}[theorem]{\bf Question}

\theoremstyle{definition}
\newtheorem{definition}[theorem]{Definition}
\newtheorem{example}[theorem]{Example}

\newcommand{\ree}[1]{(\ref{#1})}

\theoremstyle{remark}
\newtheorem{remark}[theorem]{Remark}

\numberwithin{equation}{section}

\newcommand{\ZZ}{\mathbb{Z}}
\newcommand{\M}{\mathcal{M}}
\newcommand{\E}{\mathcal{E}}
\newcommand{\sym}{\mathfrak{S}}
\newcommand{\T}{\mathcal{T}}
\newcommand{\xx}{\mathbf{x}}

\newcommand{\cover}{\lessdot}

\newcommand{\ISF}{\mathcal{ISF}}
\newcommand{\SF}{\mathcal{SF}}
\newcommand{\NC}{\mathcal{NC}}
\newcommand{\NCDyck}{\mathcal{NCD}yck}

\renewcommand{\c}{\mathbf{c}}
\newcommand{\m}{\mathbf{m}}
\renewcommand{\d}{\mathbf{d}}
\newcommand{\D}{\mathbf{D}}
\newcommand{\C}{\mathbf{C}}

\renewcommand{\r}{\mathbf{r}}
\newcommand{\F}{\mathfrak{F}}
\newcommand{\FF}{\mathbf{F}}

\def\newop#1{\expandafter\def\csname #1\endcsname{\mathop{\rm #1}\nolimits}}

\newop{sort}

\author[R. S. Gonz\'alez D'Le\'on]{Rafael S. Gonz\'alez D'Le\'on}
\address{Escuela de Ciencias Exactas e Ingenier\'ia, Universidad Sergio Arboleda, Bogot\'a, 
Colombia}
\email{rafael.gonzalezl@usa.edu.co}

\author[J. Hallam]{Joshua Hallam}
\address{Department of Mathematics and Statistics, Wake Forest University,  Winston-Salem, NC 
27109, USA}
\email{hallamjw@wfu.edu}

\begin{document}

\title{The Whitney Duals of a Graded Poset}
\maketitle

\begin{abstract}
We introduce the notion of a \emph{Whitney dual} of a graded poset.  Two posets are Whitney duals to each other if (the absolute value of) their Whitney numbers of the first and second kind are interchanged between the two posets. We define new types of edge and chain-edge labelings which we call \emph{Whitney labelings}. We prove that every graded poset with a Whitney labeling has a Whitney dual. Moreover, we show how to explicitly construct a Whitney dual using a technique involving quotient posets. 

As applications of our main theorem, we show that geometric lattices, the lattice of noncrossing partitions, the poset of weighted partitions studied by Gonz\'alez D'Le\'on-Wachs, and most of the R$^*$S-labelable posets studied by Simion-Stanley all have Whitney duals. Our technique gives a combinatorial description of a Whitney dual of the noncrossing partition lattice in terms of a family of noncrossing Dyck paths. Our method also provides an explanation of the Whitney duality between the poset of weighted partitions and a poset of rooted forests studied by Reiner and Sagan. An integral part of this explanation is a new chain-edge labeling for the poset of weighted partitions which we show is a Whitney labeling.

Finally, we show that a graded poset with a Whitney labeling admits a local action of the $0$-Hecke algebra of type $A$ on its set of maximal chains. 
The characteristic of the associated representation is Ehrenborg's flag quasisymmetric function.  The existence of this action implies, using a result of McNamara, that 
 when the maximal intervals of the constructed Whitney duals are bowtie-free, they are also snellable.
In the case where 
these maximal intervals are lattices, they are supersolvable.

\end{abstract}

\tableofcontents

\section{Introduction}\label{section:introduction}
All \emph{partially ordered sets (or posets)} considered here will be finite, graded, and 
contain a minimum element (denoted by $\hat{0}$). We assume familiarity with poset and poset 
topology terminology and notation. For background on posets the reader should visit 
\cite[Chapter 3]{Stanley2012} and \cite{Wachs2007}.

Throughout the paper, $P$ will denote a finite graded poset with a  $\hat{0}$ and $\rho$ will 
denote its rank function. The \emph{M\"obius function} of a poset $P$ is defined 
recursively for pairs $x<y$ in $P$ by
\begin{equation}\label{equation:mobiusdefinition}
\mu(x,y)=\begin{cases} 1 &\mbox{if } x=y, \\ 
-\displaystyle \sum_{x\le z < y} \mu(x,z) & \mbox{if } x \neq y. \end{cases}
\end{equation}

We illustrate, with two examples, how to calculate the values $\mu(\hat{0},x)$ of the M\"obius 
function. These two examples will be of a particular relevance throughout this article.
\begin{example}\label{example:partitionlattice}
Let $\Pi_n$ denote the poset whose underlying set is formed by the partitions of the set 
$[n]:= \{1,2,\dots, n\}$  with order relation given for $\pi,\pi^{\prime} \in \Pi_n$ by $\pi \le 
\pi^{\prime}$ if every block of $\pi$ is contained in some block of $\pi^{\prime}$. Equivalently,  
the cover relation $\pi \lessdot \pi^{\prime}$ is defined whenever $\pi^{\prime}$ is obtained from 
$\pi$ when exactly two blocks of $\pi$ are merged to form a single block in $\pi^{\prime}$ while 
the remaining blocks of $\pi$ and $\pi^{\prime}$ are the same. We say that the 
partitions are ordered by \emph{refinement} and we call $\Pi_n$ the \emph{partition lattice} (since this poset has additional structure, that of a lattice). In 
Figure \ref{figure:examplemobius} we illustrate the values $\mu(\hat{0},x)$ of 
the M\"obius function for every $x \in \Pi_3$.
\end{example}

\begin{example}\label{example:ISF}
Let $T$ be a tree with vertices labeled by distinct integers.  We call the smallest vertex of 
$T$ the \emph{root}.  We say $T$ is an \emph{increasing tree} if  the sequence of vertex labels 
read along any path starting at the root  of $T$ is increasing.  An \emph{increasing spanning 
forest} is a 
collection of 
increasing trees whose vertex labels form a partition of $[n]$.  The word ``spanning"  here 
indicates that these forests are spanning forests of the complete graph.  For more information 
about increasing 
spanning forests 
see~\cite{HallamMartinSagan2016}. We use  $\ISF_n$ to denote the set 
of increasing spanning forests on $[n]$. A partial 
order on
$\ISF_n$ is defined by  $F_1 \cover F_2$  if exactly two trees in 
$F_1$ are replaced by the tree in $F_2$ that is obtained after joining their roots with an edge.  
Note that the root  of the resulting tree is the smaller label among the roots of the 
two joined trees. See  Figure \ref{figure:examplemobius} for the Hasse diagram of $\ISF_3$ 
together with the M\"obius values $\mu(\hat{0},x)$ for every $x \in \ISF_3$.
\end{example}

\begin{figure}
 \begin{tikzpicture}[scale=1]

\begin{scope}[xshift=0cm]
\tikzstyle{every node}=[inner sep=3pt, scale=1.1, minimum width=4pt]
 \node[pin={[pin distance=10pt, pin edge={red,-, dashed}, red, inner sep=0pt]200:$+1$}] (n102030) at 
(0,0)  {$1/ 2/ 3$};
\node at (0,-1) {$\Pi_3$}; 

  \node[pin={[pin distance=10pt, pin edge={red,-, dashed}, red, inner sep=0pt]200:$-1$}] (n12i30) at (-2,2) {$12/ 3$};
  \node[pin={[pin distance=10pt, pin edge={red,-, dashed}, red, inner sep=0pt]200:$-1$}] (n13i20) at (0,2) {$13/ 2$};
  \node[pin={[pin distance=10pt, pin edge={red,-, dashed}, red, inner sep=0pt]200:$-1$}] (n1023i) at (2,2)  {$1/ 23$};

 \node[pin={[pin distance=10pt, pin edge={red,-, dashed}, red, inner sep=0pt]200:$+2$}] (n123ii) at 
(0,4){$123$};

  \tikzstyle{every path}=[line width=1pt]
 \draw (n123ii) -- (n1023i) ;
 \draw (n123ii)-- (n13i20);
  \draw (n123ii) -- (n12i30);  

  \draw (n13i20) -- (n102030);
  \draw (n12i30) -- (n102030);
  \draw (n1023i) -- (n102030);
\end{scope}
\begin{scope}[xshift=6cm,yshift=0.1cm,scale=0.40]

\node at (0,-3) {$\mathcal{ISF}_3$};
\node[pin={[pin distance=10pt, pin edge={red,-, dashed}, red, inner 
sep=0pt]210:$+1$},fill,gray!10,circle, inner sep=0pt,minimum size=1.3cm] (p1) at (0,0) {â¢};
\node[pin={[pin distance=10pt, pin edge={red,-, dashed}, red, inner sep=0pt]210:$-1$},fill,gray!10,circle, inner sep=0pt,minimum size=1.3cm] (p2) at (0,5) {â¢};
\node[pin={[pin distance=10pt, pin edge={red,-, dashed}, red, inner sep=0pt]210:$-1$},fill,gray!10,circle, inner sep=0pt,minimum size=1.3cm] (p3) at (-5,5) {â¢};
\node[pin={[pin distance=10pt, pin edge={red,-, dashed}, red, inner sep=0pt]210:$-1$},fill,gray!10,circle, inner sep=0pt,minimum size=1.3cm] (p4) at (5,5) {â¢};

\node[pin={[pin distance=10pt, pin edge={red,-, dashed}, red, inner 
sep=0pt]210:$+1$},fill,gray!10,circle, inner sep=0pt,minimum size=1.3cm] (p5) at (-2.5,9) {â¢};
\node[pin={[pin distance=10pt, pin edge={red,-, dashed}, red, inner sep=0pt]210:$0$},fill,gray!10,circle, inner sep=0pt,minimum size=1.3cm] (p6) at (5,9) {â¢};

\tikzstyle{every node}=[black, draw, circle, inner sep=0.4pt, scale=1, minimum width=4pt,scale=0.8]
\tikzstyle{every path}=[blue, line width=0.2pt]

\node (a1) at (0,1){$1$};
\node (a2) at (1,0){$2$};
\node (a3) at (-1,0){$3$};

\node (b1) at (0,6){$1$};
\node (b2) at (1,5){$2$};
\node (b3) at (-1,5){$3$};
 \draw[thick] (b1) -- (b3) ;

\node (c1) at (-5,6){$1$};
\node (c2) at (-4,5){$2$};
\node (c3) at (-6,5){$3$};
 \draw[thick] (c1) -- (c2) ;

\node (d1) at (5,6){$1$};
\node (d2) at (6,5){$2$};
\node (d3) at (4,5){$3$};
 \draw[thick] (d2) -- (d3) ;

\node (e1) at (-2.5,10){$1$};
\node (e2) at (-1.5,9){$2$};
\node (e3) at (-3.5,9){$3$};
\draw[thick] (e2)--(e1)--(e3);

\node (f1) at (5,10){$1$};
\node (f2) at (6,9){$2$};
\node (f3) at (4,9){$3$};
\draw[thick] (f1)--(f2)--(f3);

\tikzstyle{every path}=[line width=1.5pt]
 \draw (p1) -- (p2) ;
 \draw (p1) -- (p3) ;
 \draw (p1) -- (p4) ;
 \draw (p2) -- (p5) ;
 \draw (p3) -- (p5) ;
 \draw (p4) -- (p6) ;

\end{scope}

\end{tikzpicture}
 \caption{M\"obius function values $\mu(\hat{0},x)$ on $\Pi_3$ and $\ISF_3$}
 \label{figure:examplemobius}
\end{figure}
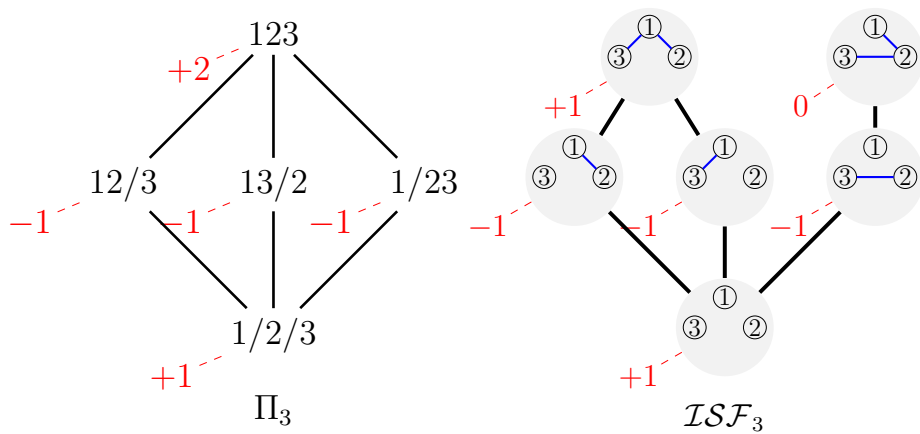

Two important invariants that we can associate to a graded poset $P$ are its Whitney numbers of the 
first and second kind. The \emph{ $k^{th}$ Whitney number of the first kind}, $w_k(P)$, is defined 
by
\begin{align}\label{definition:whitneyfirstkind}
w_k(P) = \sum_{\rho(x) =k} \mu(\hat{0},x),
\end{align} 
and the \emph{$k^{th}$ Whitney number of the second kind}, $W_k(P)$, is defined by
\begin{align}\label{definition:whitneysecondkind}
W_k(P)  = |\{x\in P \mid \rho(x) = k\}|.
\end{align}

The Whitney numbers of a graded poset play an important role in many areas of combinatorics.  For 
example, 
they appear as coefficients of the chromatic polynomial of a finite graph~\cite{rota1964}.  
Stanley~\cite{Stanley1973} showed they 
can be used to count the number of acyclic orientations of a graph.   When the poset is the 
intersection lattice of a real hyperplane arrangement, Zaslavsky~\cite{Zaslavsky1975} showed  that 
its Whitney numbers can be used to 
count the number of bounded and unbounded regions. For complex hyperplane 
arrangements the Whitney numbers can be used to compute the dimensions of the 
Orlik-Solomon algebra that is isomorphic to the Whitney cohomology of the intersection 
lattice~\cite{OrlikSolomon1980}.  Very recently, a long 
standing conjecture of Heron~\cite{Heron1972}, Rota~\cite{rota1971} and Welsh~\cite{Welsh1976} 
concerning the 
log-concavity of the Whitney numbers of the first kind of geometric lattices was settled by 
Adiprasito, Huh, and Katz~\cite{AdiprasitoHuhKatz2015} using ideas coming from Hodge theory.

\begin{table}
\begin{center}
\begin{tabular}{|c|c|c|c|c|}\hline
\bf{$k$} & \textbf{$w_k(\Pi_n)$} & \textbf{$W_k(\Pi_n)$} & \textbf{$w_k(\ISF_n)$} & 
\textbf{$W_k(\ISF_n)$}\\\hline
\textbf{$0$}  & $1$ & $1$& $1$ & $1$\\\hline
\textbf{$1$}  & $-3$ & $3$& $-3$ & $3$\\\hline
\textbf{$2$}  & $2$ & $1$ & $1$ & $2$\\\hline
\end{tabular}
\end{center}
\caption{Whitney numbers of the first and second kind for $\Pi_3$ for $\ISF_3$.}
\label{table:whitneynumberspi3if3}
\end{table}

Now consider the previous two examples in Figure~\ref{figure:examplemobius}. When we compute the 
Whitney 
numbers of $\Pi_3$ and $\ISF_3$ and we list them side by side (see Table 
\ref{table:whitneynumberspi3if3}), we notice a curious coincidence: (up to sign) their Whitney 
numbers of the first and second kind are interchanged. This surprising phenomenon was 
initially noticed by the authors of \cite{DleonWachs2016} for a different, but closely related, 
pair of posets: the poset $\Pi^w_n$ of
weighted partitions and the poset $\SF_n$ of rooted spanning forests on $[n]$ 
    studied by Reiner in \cite{Reiner1978} and Sagan in \cite{Sagan1983}. It turns out that this 
  phenomenon occurs for many other pairs of posets, in particular, the authors announced in
  \cite{DleonHallam2017} that for every geometric lattice there exist another poset   
  with their Whitney numbers interchanged. This seemingly common phenomenon motivates the following 
  definition.

  \begin{definition}\label{definition:whitneyduals}
  Let $P$ and $Q$ be graded posets. We say that $P$ and $Q$ are \emph{Whitney Duals} if for all 
  $k\ge0$ we have that 
  \begin{align}
  |w_k(P)|=W_k(Q) \text{ and } |w_k(Q)|=W_k(P).
  \end{align}
  \end{definition}

  According to this definition, $\Pi_3$ and $\ISF_3$ are Whitney duals. In fact, this is true in general for 
  $\Pi_n$ and $\ISF_n$ for $n\ge 1$ (see \cite{DleonHallam2017}).

Our investigation on Whitney duals is driven by the following two questions that we address to 
different extents in the present work:

\begin{question}\label{question:whitneydual1}
 When does a graded poset $P$ has a Whitney dual?
\end{question}

\begin{question}\label{question:whitneydual2}
Is there a method to construct a Whitney dual for a graded poset $P$? Perhaps under certain 
assumptions about $P$.
\end{question}

In this article we advance towards an answer to the first question as we provide an answer to the 
second question.  To do this, we use poset topology tools including edge labeling,  chain-edge labelings, and 
quotient posets.

To state our main theorem, we now briefly review a few basic ideas concerning edge and chain-edge labelings. Recall that the \emph{Hasse diagram of $P$} is the directed graph on $P$ with directed 
edges the 
covering relations $x\lessdot y$ in $P$, i.e., $x < y$ where there is no
$z\in P$ satisfying $x<z<y$.
Also recall  that an \emph{edge labeling or E-labeling} of $P$ is a map $\lambda: \E(P)\rightarrow 
\Lambda$ where $\E(P)$ is the set of edges of the Hasse 
diagram of $P$ and $\Lambda$ is some other poset of labels. An 
edge labeling is said to be an \emph{ER-labeling} if in every interval $[x,y]$ of $P$ there is a 
unique \emph{saturated or unrefinable chain} $$\c : 
(x 
= x_0 \lessdot x_1 \lessdot \cdots \lessdot x_{\ell-1} \lessdot x_{\ell}= y)$$  that is 
\emph{increasing}, i.e., such that
$$\lambda(x_0 \lessdot x_1) < \lambda(x_1 \lessdot x_2) < \cdots < \lambda(x_{\ell-1}\lessdot 
x_{\ell}).$$  
The concept of an  edge labeling was generalized 
by Bj\"orner and Wachs to a labeling of pairs 
formed by maximal chains and edges of the Hasse diagram known as \emph{chain-edge labelings or C-labelings}. The generalization of an ER-labeling is then known as a  CR-labeling.
Chain-edge 
labelings are more involved and technical in their definition than edge labelings. Thus,  as 
a presentation strategy across this article we  choose to first discuss  the 
definitions and constructions in the context of edge labelings and then explain how these 
constructions generalize to the context of chain-edge labelings.
The reader should visit \cite{ Bjorner1980, BjornerWachs1983, Stanley2012, Wachs2007} for information  on edge and edge-chain labelings. 

In Section \ref{section:whitneylabelings} we define new types of edge and chain-edge labelings that 
we call \emph{Whitney labelings} or \emph{W-labelings}, the name coming from the fact 
that they provide sufficient conditions for the construction of Whitney duals of graded posets.  A crucial property of Whitney labelings is the rank two switching property.   We say that an edge labeling $\lambda$ has the 
\emph{rank two switching property} if for every maximal chain 
$$\c : ( \hat{0}= x_0 \lessdot x_1 \lessdot \cdots \lessdot x_{k-1} \lessdot x_{k})$$ that 
has an increasing step $\lambda(x_{i-1} \lessdot x_i) < \lambda(x_i \lessdot x_{i+1})$ at rank $i$ 
there is a unique maximal chain 
$$\c' : ( \hat{0}= x_0 \lessdot x_1 \lessdot \cdots \lessdot x_{i-1} \lessdot x_i' \lessdot 
x_{i+1}\lessdot \cdots \lessdot x_{k-1} \lessdot x_{k}),$$
whose labels are the same as the ones from $\c$ except for $\lambda(x_{i-1} \lessdot 
x_i')=\lambda(x_{i} \lessdot x_{i+1})$ and $\lambda(x_{i}' \lessdot x_{i+1})=\lambda(x_{i-1} 
\lessdot x_i).$
For chain-edge labelings we require an  additional condition for the rank two switching property, namely that the 
choice of the $x_i'$ is consistent among maximal chains that coincide in the bottom $d$ elements 
when $d>i$.

A Whitney labeling of $P$ is an ER or CR-labeling with the rank two switching 
property and with the property that in every interval $[x,y]$ (a rooted interval in 
the case of CR-labelings) of $P$ each ascent-free maximal chain is determined uniquely by its 
sequence of labels from bottom to top. We will call Whitney labelings respectively EW or CW-labelings depending if the underlying labeling is an ER or a CR-labeling.

The main result of this paper is the following:
\begin{theorem}\label{theorem:WimpliesWhitney}
 Let $P$ be a poset with a Whitney labeling $\lambda$.  Then $P$ has a Whitney dual.
 Moreover, using $\lambda$ we can explicitly construct a Whitney dual $Q_\lambda(P)$ of $P$.
\end{theorem}

We actually prove a more general result. We define a more general kind of labeling that we call a 
\emph{generalized Whitney labeling}. Its definition relies on a set of more 
general but, at the same time, more technical conditions that also imply the existence of a Whitney dual. However, all the examples that we know of so far satisfy the nicer definition 
that we gave above.

Theorem \ref{theorem:WimpliesWhitney} guarantees for every Whitney labeling $\lambda$ of a graded 
poset $P$ a Whitney dual $Q_{\lambda}(P)$. In general, there is no reason to expect that this 
construction is independent of $\lambda$ and in fact in Section \ref{weightedPartSec} we show 
that for the poset of weighted partitions $\Pi_n^w$ there are at least two non-isomorphic Whitney 
duals. The two Whitney duals that we find are of the form $Q_{\lambda}(P)$ for two different Whitney labelings.  One of the labelings is an ER-labeling that was already introduced in 
\cite{DleonWachs2016}. We 
prove here that this labeling is also an EW-labeling. The second labeling is a new CR-labeling that 
we also show is a CW-labeling. The ascent-free chains of the new CW-labeling are indexed by 
rooted forests of $[n]$ and the construction of the poset $Q_{\lambda}(P)$  provides an explanation 
of the Whitney duality between the weighted partition poset, $\Pi_n^w$, and the poset of rooted spanning forest, $\SF_n$. This duality can be seen 
from the work of Gonz\'alez D'Le\'on - Wachs in \cite{DleonWachs2016} and the work 
of Sagan in \cite{Sagan1983} after comparing the Whitney numbers of $\Pi_n^w$ and $\SF_n$.  Our construction gives an explanation for this duality, which was our initial motivation for the 
current project. 

\begin{theorem}\label{theorem:weightedpartitionsSF}
For all $n\ge 1$, there is a CW-labeling $\lambda_C$ of $\Pi_n^w$ such that 
$Q_{\lambda_C}(\Pi_n^w)$ and $\SF_n$ are isomorphic posets. In particular  $\Pi_n^w$ and  $\SF_n$ 
are Whitney duals.
\end{theorem}

The rest of the paper is organized as follows.  In Section~\ref{section:basicexamples} we 
consider basic notions and examples of Whitney duals.  In Section~\ref{section:whitneylabelings} we 
introduce the notion of Whitney labeling and we give a construction of a Whitney dual using a Whitney labeling and quotient posets.

In Section \ref{section:whitneydual} we give a characterization of the M\"obius values of 
$Q_\lambda(P)$. We also provide a simpler description of the posets $Q_\lambda(P)$ when $\lambda$ is a 
Whitney labeling (in the strict sense of the 
definition, i.e., not a generalized one).

In Section \ref{section:examplesWposets} we provide a series of examples of posets for which Theorem \ref{theorem:WimpliesWhitney} applies. These include all geometric lattices, the poset of weighted partitions $\Pi_n^w$ and the 
R$^*$S-labelable posets studied by Simion and Stanley \cite{SimionStanley1999}.   This last family 
includes the lattice $\NC_n$ of 
noncrossing  partitions of $[n]$, the posets of shuffles studied by Greene \cite{Greene1988} and also the noncrossing partition 
lattices of types B and D studied by Reiner \cite{Reiner1997} (see also \cite{Hersh1999}). With these 
examples we illustrate both the existence of EW and CW-labelings. For the particular 
cases of $\NC_n$ and $\Pi_n^w$ we also give explicit combinatorial descriptions of their Whitney 
duals $Q_\lambda(P)$. 

In Section \ref{section:heckeaction} we show that a Whitney labeling induces a local action
of the $0$-Hecke algebra  $H_n(0)$ of type A on the set of maximal chains of $P$ and hence a 
representation $\chi_P$ of $H_n(0)$ on the space spanned by the maximal chains of $P$. This action 
can be transported to a local action on the set of maximal chains of its Whitney dual 
$Q_\lambda(P)$ and hence also induces  a representation $\chi_{Q_\lambda(P)}$. In particular, we prove the 
following theorem. 

 \begin{theorem}\label{theorem:goodactions}
Let $P$ be a poset with a (generalized) CW-labeling $\lambda$.  Then
$$
ch(\chi_{Q_{\lambda}(P)})=ch(\chi_P) = F_p(x) = \omega F_{Q_{\lambda}}(x),
$$
where $F_P(x)$ is Ehrenborg's flag quasisymmetric function of the graded poset $P$, $ch(V)$ indicates the quasisymmetric characteristic of the $H_n(0)$-representation $V$ and $\omega$ is 
the classical involution in the ring of quasisymmetric functions that maps Gessel's fundamental 
quasisymmetric function $L_{S,n}$ indexed by a set $S$ to the one indexed by its complement $[n-1]\setminus S$.
\end{theorem}

Theorem \ref{theorem:goodactions} implies that $\chi_{Q_{\lambda}(P)}$ is induced by a ``good 
action'' in the sense of \cite{SimionStanley1999}. The results in \cite{McNamara2003} then 
imply that the  bowtie-free  maximal intervals of the posets $Q_{\lambda}(P)$ are \emph{snellable} (see \cite{McNamara2003}) and the maximal intervals of $Q_\lambda(P)$ that are lattices are \emph{supersolvable}.

\section{Basic Examples  of posets with and without Whitney Duals}\label{section:basicexamples}

We start by a discussion on some basic examples of posets which have and do not have Whitney 
duals.

A graded poset $P$ is \emph{Eulerian} if $\mu(x,y)  = 
(-1)^{\rho(y)-\rho(x)}$ for all $x\leq y$ in $P$ (see Figure~\ref{fig:Eulerian} for an example).  
Other examples of Eulerian posets include the face lattice of a convex polytope and the (strong) 
Bruhat order.  For more information on Eulerian posets see the survey 
article \cite{Stanley1994}.  From the definition of Eulerian poset, it is immediate that we have
$|w_k(P)| 
= W_k(P)$ for all $k$.  Therefore, every Eulerian poset has a Whitney dual, namely itself.  It is 
natural to ask if all posets which are their own Whitney dual are Eulerian.   The poset in 
Figure~\ref{fig:selfdulnoneulerian} shows that this is not the case \footnote{We thank Cyrus Hettle 
 for pointing out this example to the authors.}.    This leads to the currently open question of whether there is a natural characterization of self-Whitney dual posets.

\begin{figure}
\begin{subfigure}[b]{0.5\textwidth}
                \centering
 \begin{tikzpicture}[line join=bevel,scale=.75]

\tikzstyle{every node}=[inner sep=3pt, scale=1, minimum width=4pt]
 \node (n102030) at (0,0)  {$a$};

  \node (n12i30) at (-2,2) {$b$};
  \node (n13i20) at (0,2) {$c$};
  \node (n1023i) at (2,2)  {$d$};

 \node (n123ii) at (-2,4){$e$};
 \node (nf) at (2,4){$f$};

 \draw (n123ii)-- (n13i20);
  \draw (n123ii) -- (n12i30);  

  \draw [] (n13i20) -- (n102030);
  \draw [] (n12i30)-- (n102030);
  \draw [] (n1023i)--(n102030);
 \draw [] (n1023i) --(nf)-- (n13i20) ;
\tikzstyle{every node}=[red, inner sep=0.4pt, scale=1, minimum width=4pt]

\node[] at (0.5,0){$+1$};
\node at (-1.4,2){$-1$};
\node at (0.5,2){$-1$};
\node at (2.5,2){$-1$};
\node at (-1.5,4){$+1$};
\node at (2.4,4){$+1$};
 
\end{tikzpicture}

\caption{Eulerian and self Whitney dual} 
                  \label{fig:Eulerian}
               \end{subfigure}~
\begin{subfigure}[b]{0.5\textwidth}
                \centering
 \begin{tikzpicture}[line join=bevel,scale=.75]

\tikzstyle{every node}=[inner sep=3pt, scale=1, minimum width=4pt]
 \node (n102030) at (0,0)  {$a$};

  \node (n12i30) at (-2,2) {$b$};
  \node (n13i20) at (0,2) {$c$};
  \node (n1023i) at (2,2)  {$d$};

 \node (n123ii) at (0,4){$e$};
 \node (nf) at (2,4){$f$};
 
 \draw (n123ii) -- (n1023i) ;
 \draw (n123ii)-- (n13i20);
  \draw (n123ii) -- (n12i30);  

  \draw [] (n13i20) -- (n102030);
  \draw [] (n12i30)-- (n102030);
  \draw [] (n1023i)--(n102030);
 \draw [] (n1023i) --(nf);
\tikzstyle{every node}=[red, inner sep=0.4pt, scale=1, minimum width=4pt]

\node[] at (0.5,0){$+1$};
\node at (-1.4,2){$-1$};
\node at (0.5,2){$-1$};
\node at (2.5,2){$-1$};
\node at (0.6,4){$+2$};
\node at (2.4,4){$0$};
 
\end{tikzpicture}

\caption{Self  Whitney dual, but not Eulerian} 
                  \label{fig:selfdulnoneulerian}
               \end{subfigure}~

\begin{subfigure}[b]{0.5\textwidth}
                \centering
                  
 \begin{tikzpicture}[line join=bevel,scale=.75]

\tikzstyle{every node}=[inner sep=3pt, scale=.9, minimum width=4pt]
 \node (n102030) at (0,0)  {$a$};

  \node (n13i20) at (0,2) {$b$};

 \node (n123ii) at (0,4){$c$};
 
 \draw (n123ii)-- (n13i20);

  \draw [] (n13i20) -- (n102030);
 
\tikzstyle{every node}=[red, inner sep=0.4pt, scale=1, minimum width=4pt]

\node[] at (0.5,0){$+1$};
\node at (0.5,2){$-1$};
\node at (0.6,4){$0$};
 
\end{tikzpicture}
\caption{No Whitney duals}
                  \label{fig:posetnowhitneydual}

\end{subfigure}~
\begin{subfigure}[b]{0.5\textwidth}
                \centering
                  
 \begin{tikzpicture}[line join=bevel,scale=.75]

\tikzstyle{every node}=[inner sep=3pt, scale=.9, minimum width=4pt]
 \node (a) at (0,0)  {$a$};

  \node (b) at (-2,2) {$b$};

 \node (c) at (0,2){$c$};
 
 \node (d) at (2,2){$d$};
 
  \node (e) at (-2,4) {$e$};

 \node (f) at (2,4){$f$};

 \node (g) at (0,6){$g$};

\draw (a) --(b)--(e) (a)--(c)--(e) (a)--(d)--(e);
\draw  (a)--(c)--(f) (a)--(d)--(f);
\draw  (e)--(g) (g)--(f);
 
\tikzstyle{every node}=[red, inner sep=0.4pt, scale=1, minimum width=4pt]

\node[] at (0.5,0){$+1$};
\node at (-1.4,2){$-1$};
\node at (0.5,2){$-1$};
\node at (2.5,2){$-1$};
\node at (-1.5,4){$+2$};
\node at (2.4,4){$+1$};
\node at (.5,6) {$-1$};

\end{tikzpicture}
\caption{No Whitney duals}
                  \label{fig:posetnowhitneydual2}

\end{subfigure}~

         \caption{Examples of self Whitney-duality and of posets without Whitney dual}
  \label{fig:exampleofselfdualityandnotwhitney}
\end{figure}
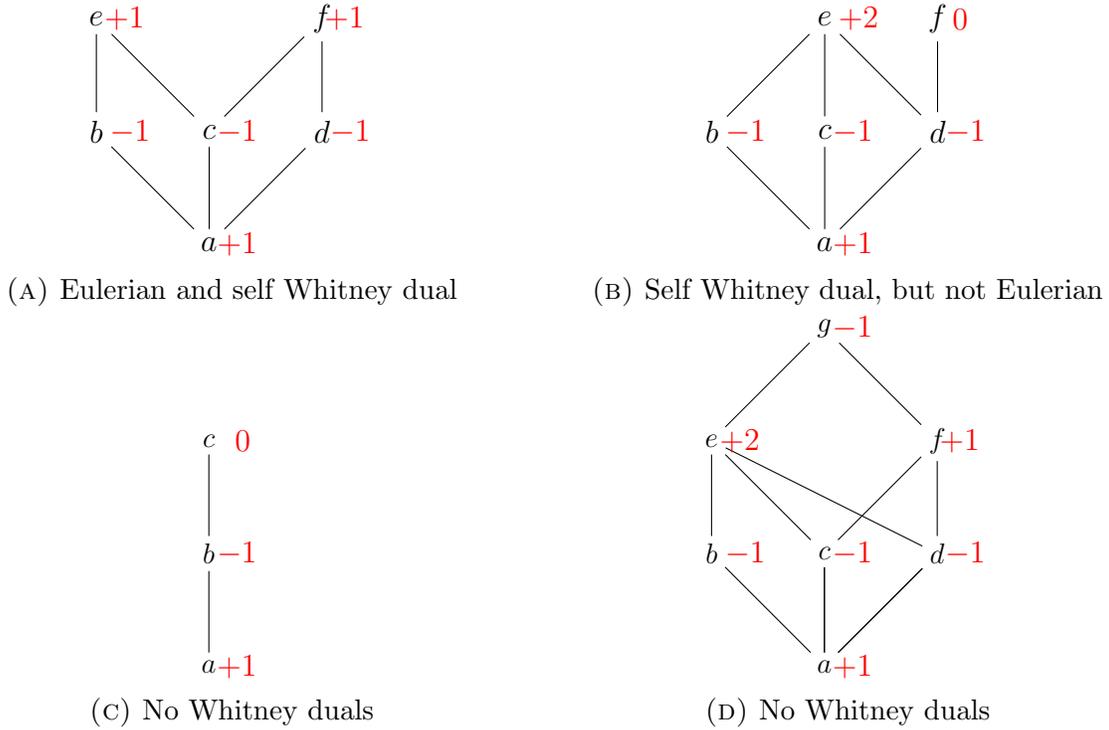

Not every ranked poset has a Whitney dual.  For example, consider the three element chain $C$ in 
Figure \ref{fig:posetnowhitneydual}.  We have that $w_2(C) =0$ and $W_2(C) =1$.  
If $Q$ was a Whitney dual of $C$, then $|w_2(Q)|=1$ and $W_2(Q)=0$, which is clearly impossible.  
This illustrates the fact that a poset $P$ with $|w_k(P)|=0$ for some $k$ smaller than the rank of 
the poset  cannot 
have a Whitney dual.

We also remark that there are more complicated reasons  that can prevent a poset from having a 
Whitney dual.  As an example, consider the poset in Figure~\ref{fig:posetnowhitneydual2}.   Suppose 
that this poset had a Whitney dual, $Q$.  Then the absolute value of the sequence of Whitney 
numbers of the first kind for $Q$ would be $1,3,2,1$.  Moreover, $Q$ would necessarily have a unique maximal element ($\hat{1}$).  This would imply that the sum of the Whitney numbers of the first kind would be 0.   
Now 
it must be the case that $w_0(Q) = 1$ and $w_2(Q) =-3$.  As a result we should be able to
assign a sign to $2$ and $1$ so that $1-3\pm 2 \pm 1=0$, which is impossible.

One might wonder if Whitney duals are unique.   By considering the posets in 
Figures~\ref{fig:Eulerian} and~\ref{fig:selfdulnoneulerian}, one can see this is not the case.  
These posets are both self Whitney dual and Whitney dual to each other.  In fact, a poset that is not self-Whitney dual can have multiple non-isomorphic Whitney duals.  We show in  
Section~\ref{weightedPartSec} that  $\Pi_n^w$ is an example of such a poset.

Now that we have seen some basic examples of posets with and without Whitney duals, we turn our attention to an approach on constructing Whitney duals.

\section{Whitney labelings and quotient posets}\label{section:whitneylabelings}
The idea of edge labelings or E-labelings is pervasive in the poset literature.
The concept of an ER-labeling (originally called an R-labeling) was introduced by Stanley 
(see \cite{Stanley1972}) to study the M\"obius function of rank selected subposets of a graded 
poset. Bj\"orner \cite{Bjorner1980} extended this notion by adding a lexicographic condition. Such labelings are called EL-labelings. The extra condition on an EL-labeling implies the shellability of the 
order (simplicial) complex of the poset and hence has stronger topological 
consequences that allow for the determination of its homotopy type.  Bj\"orner and Wachs 
\cite{BjornerWachs1982} extended  the theory of lexicographic shellability to certain type 
of 
labelings called chain-edge labelings or C-labelings (see also 
\cite{BjornerWachs1983,BjornerWachs1996,BjornerWachs1997}). This new concept of a CR-labeling (or CL-labeling when a lexicographic condition is considered) provided a more flexible 
description that helped in  the determination of the M\"obius numbers of posets that were not known 
to be included in the family of ER-labelable posets.  For example, C-labelings can be used to understand the M\"obius numbers    of the Bruhat order of a 
Coxeter group. An ER-labeling is a special case of a 
CR-labeling and hence chain-edge labelings are in principle more general than edge labelings. In 
fact it is not known whether a poset that has a CR-labeling also has an ER-labeling.
In this section we show how edge and chain-edge labelings can be used to construct Whitney duals. 
We will first describe all constructions using only the concept of an E-labeling 
since the presentation and proofs will be more clear and pleasant to the reader. The use of 
C-labelings require certain technical details, like the 
concept of a rooted interval, that might obscure the relevant ideas involved. However, we hope that the reader will find that the bottom-to-top nature of our constructions are compatible with C-labelings.

In the last subsection we discuss how all the theory developed for E-labelings continue to hold 
for C-labelings. We describe  the important modifications to the definitions and properties 
involved when transferring from the context of E-labelings to C-labelings.

\subsection{Edge labelings}

We now discuss edge labelings and their relation with Whitney numbers.  First, let 
us recall some basic facts about edge labelings.  For complete treatments 
on the topic, see~\cite{ Bjorner1980, BjornerWachs1983, Stanley2012}.  Let $P$ be a poset, and 
let 
$\E(P)$ be the set of edges of the Hasse diagram of $P$.  Moreover, let $\Lambda$ be 
an arbitrary fixed poset that will be considered as the \emph{poset of labels}.  An \emph{edge 
labeling} of $P$ is a map  $\lambda: \E(P)\rightarrow \Lambda$.  

Let $P$ be a poset with edge 
labeling $\lambda$. To every \emph{saturated chain} (also known as an \emph{unrefinable 
chain}) $$\c : 
(x 
= x_0 \lessdot x_1 \lessdot \cdots \lessdot x_{\ell-1} \lessdot x_{\ell}= y)$$ we can associate a  
corresponding \emph{word of labels}
$$\lambda(\c) = \lambda(x_0\lessdot x_1) \lambda(x_1\lessdot x_2) \cdots \lambda(x_{\ell-1}\lessdot 
x_{\ell}).$$
We say that  $\c$ is  \emph{increasing} if its word of labels $\lambda(\c)$ is
\emph{strictly}  increasing, that is, $\c$ is  increasing if 
\begin{align*}
 \lambda(x_0\lessdot x_1) < \lambda(x_1\lessdot x_2)<  \cdots < \lambda(x_{\ell-1}\lessdot 
x_{\ell}). 
\end{align*}
  We say that  $\c$ is  \emph{ascent-free} if its word of labels 
$\lambda(\c)$ has no ascents, i.e. $\lambda(x_{i-1}\lessdot x_i) \not<  \lambda(x_i\lessdot 
x_{i+1}) $, 
for 
all 
$i=1,\dots,\ell-1$. 
Clearly, there are chains that are neither increasing nor ascent-free.

\begin{definition}\label{definition:ER}
 An edge labeling is an \emph{ER-labeling} if in each closed
interval $[x,y]$ of $P$, there is a unique  increasing maximal chain. By analogy, we say that an 
edge labeling is an 
ER$^*$-labeling if 
in each closed interval $[x,y]$ of $P$, there is a unique ascent-free maximal chain. 
\end{definition}

We note that in~\cite{Stanley2012}, ER-labelings are referred to as R-labelings and  have a 
slightly different definition.    First, it is assumed that the labels are 
totally ordered. Also, 
increasing refers to weakly increasing and ascent-free is replaced with strictly decreasing. 
An almost identical proof as the one of Theorem 3.14.2 in~\cite{Stanley2012} gives 
the following.

\begin{theorem}[c.f. Theorem 3.14.2 in \cite{Stanley2012}]\label{theorem:muER}
Let $P$ be a graded poset with an ER-labeling (ER$^*$-labeling).  Then
$$
 \mu(x,y)= (-1)^{\rho(y)-\rho(x)}|\{\c \,\mid\,\c \mbox{ is an ascent-free (increasing) maximal 
chain 
in $[x,y]$}\}|.
$$
\end{theorem}

Let us now consider examples of ER and ER$^*$ labelings.   In both examples the labels will 
come from the set $[n]\times [n]$ and we assume that this set is ordered 
lexicographically  using  the natural order on $[n]$ as integers.

\begin{example}\label{example:labelingposetofpartitions}
Let $\lambda:\E(\Pi_n)\rightarrow [n]\times [n]$ be the edge 
labeling defined on the partition lattice $\Pi_n$ by $\lambda(\pi \cover \sigma) = (i,j)$ where 
$i<j$ and $i$ and $j$ are the 
minimum elements of the two blocks of 
$\pi$ that were merged to obtain $\sigma$.  This edge labeling is an ER-labeling and is a special 
case of 
Bj\"orner's minimum labeling for 
geometric lattices described in \cite{Bjorner1982}. In  
Figure~\ref{figure:exampleER} the labeling $\lambda$ of $\Pi_3$ is depicted.
\end{example}

\begin{example}\label{example:labelingISF}
Let $\lambda^*:\E(\ISF_n)\rightarrow [n]\times [n]$ be the edge labeling defined on $\ISF_n$ by setting
$\lambda^*(F_1\cover F_2)$ to be the unique edge in $F_2$ that is not in $F_1$. It was proved in  
\cite[Proposition 3]{DleonHallam2017} that this edge labeling is an ER$^*$-labeling. In 
Figure~\ref{figure:exampleER}, the labeling $\lambda^*$ of $\ISF_3$ is 
depicted.
\end{example}

\begin{figure}
 \begin{tikzpicture}[scale=1]

\begin{scope}[xshift=0cm]
\tikzstyle{every node}=[inner sep=3pt, scale=1.1, minimum width=4pt]
 \node (n102030) at (0,0)  {$1/ 2/ 3$};
\node at (0,-1) {$\Pi_3$}; 

  \node (n12i30) at (-2,2) {$12/ 3$};
  \node (n13i20) at (0,2) {$13/ 2$};
  \node (n1023i) at (2,2)  {$1/ 23$};

 \node (n123ii) at (0,4){$123$};

  \tikzstyle{every path}=[line width=1pt]
 \draw (n123ii) -- (n1023i) ;
 \draw (n123ii)-- (n13i20);
  \draw (n123ii) -- (n12i30);  

  \draw (n13i20) -- (n102030);
  \draw (n12i30) -- (n102030);
  \draw (n1023i) -- (n102030);

\tikzstyle{every node}=[blue, inner sep=0.4pt, scale=.85, minimum width=4pt]
\node at (-1.5,1) {$(1,2)$};
\node at (-0.35,1) {$(1,3)$};
\node at (1.5,1) {$(2,3)$};

\node at (-1.5,3) {$(1,3)$};
\node at (-0.35,3) {$(1,2)$};
\node at (1.5,3) {$(1,2)$};

\end{scope}
\begin{scope}[xshift=6cm,yshift=0.1cm,scale=0.40]

\node at (0,-3) {$\mathcal{ISF}_3$};
\node[fill,gray!10,circle, inner sep=0pt,minimum size=1.3cm] (p1) at (0,0) {Ã¢ÂÂ¢};
\node[fill,gray!10,circle, inner sep=0pt,minimum size=1.3cm] (p2) at (0,5) {Ã¢ÂÂ¢};
\node[fill,gray!10,circle, inner sep=0pt,minimum size=1.3cm] (p3) at (-5,5) {Ã¢ÂÂ¢};
\node[fill,gray!10,circle, inner sep=0pt,minimum size=1.3cm] (p4) at (5,5) {Ã¢ÂÂ¢};

\node[fill,gray!10,circle, inner sep=0pt,minimum size=1.3cm] (p5) at (-2.5,9) {Ã¢ÂÂ¢};
\node[fill,gray!10,circle, inner sep=0pt,minimum size=1.3cm] (p6) at (5,9) {Ã¢ÂÂ¢};

\tikzstyle{every node}=[black, draw, circle, inner sep=0.4pt, scale=1, minimum width=4pt,scale=0.8]
\tikzstyle{every path}=[blue, line width=0.2pt]

\node (a1) at (0,1){$1$};
\node (a2) at (1,0){$2$};
\node (a3) at (-1,0){$3$};

\node (b1) at (0,6){$1$};
\node (b2) at (1,5){$2$};
\node (b3) at (-1,5){$3$};
 \draw[thick] (b1) -- (b3) ;

\node (c1) at (-5,6){$1$};
\node (c2) at (-4,5){$2$};
\node (c3) at (-6,5){$3$};
 \draw[thick] (c1) -- (c2) ;

\node (d1) at (5,6){$1$};
\node (d2) at (6,5){$2$};
\node (d3) at (4,5){$3$};
 \draw[thick] (d2) -- (d3) ;

\node (e1) at (-2.5,10){$1$};
\node (e2) at (-1.5,9){$2$};
\node (e3) at (-3.5,9){$3$};
\draw[thick] (e2)--(e1)--(e3);

\node (f1) at (5,10){$1$};
\node (f2) at (6,9){$2$};
\node (f3) at (4,9){$3$};
\draw[thick] (f1)--(f2)--(f3);

\tikzstyle{every path}=[line width=1.5pt]
 \draw (p1) -- (p2) ;
 \draw (p1) -- (p3) ;
 \draw (p1) -- (p4) ;
 \draw (p2) -- (p5) ;
 \draw (p3) -- (p5) ;
 \draw (p4) -- (p6) ;
\tikzstyle{every node}=[blue, inner sep=0.4pt, scale=.85, minimum width=4pt]
\node at (-3.5,2.2) {$(1,2)$};
\node at (-0.9,2.2) {$(1,3)$};
\node at (3.5,2.2) {$(2,3)$};

\node at (-4.7,7.2) {$(1,3)$};
\node at (-0.35,7.2) {$(1,2)$};
\node at (3.7,7.2) {$(1,2)$};

\end{scope}

\end{tikzpicture}
 \caption{Example of ER and ER$^*$-labelings on $\Pi_3$ and $\ISF_3$}
 \label{figure:exampleER}
\end{figure}
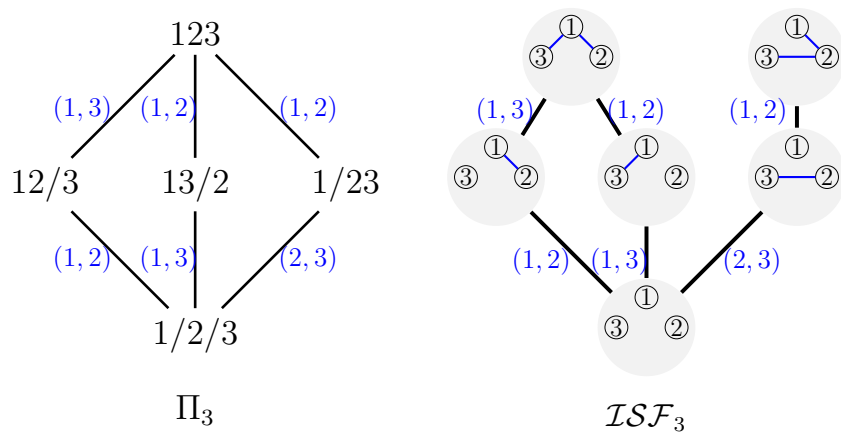

The reader may have noticed how similar are the labelings of $\Pi_3$ and $\ISF_3$.  This 
is no coincidence and as we will see later, the labeling of $\ISF_3$ can be obtained from the 
labeling of $\Pi_3$.

Using the definition of Whitney numbers, Definition~\ref{definition:ER}, and 
Theorem~\ref{theorem:muER}, we can describe the 
Whitney numbers of a poset using an ER-labeling (ER$^{*}$-labeling) by the enumeration of 
saturated chains as follows.

\begin{proposition}\label{proposition:ERwhitneynumbers}
Let $P$ be a graded poset with an ER-labeling (ER$^{*}$-labeling).  Then $|w_k(P)|$ is the number 
of 
ascent-free (increasing) saturated chains starting at $\hat{0}$  of length $k$.  Moreover, 
$|W_k(P)|$ 
is the number of increasing (ascent-free) saturated chains starting at $\hat{0}$ of length $k$.
\end{proposition}

In Table \ref{table:propositionERwhitneynumbers} we summarize in the conclusions of Proposition 
\ref{proposition:ERwhitneynumbers} to 
highlight the importance of this proposition and the fact that ER and ER$^*$-labelings switch the 
role of increasing and 
ascent-free saturated chains. This fact will be used later in a construction of Whitney duals.

\begin{table}
\begin{center}
\begin{tabular}{|c|c| c| } \hline
$\lambda$ is  an& $|w_k(P)|$ & $W_k(P)$ \\ \hline
 ER-labeling& $\#$ (ascent-free sat. chains& $\#$ (increasing sat. chains \\
  & of length $k$ starting at $\hat{0}$) &of length $k$ starting at $\hat{0}$)\\ \hline
 ER$^*$-labeling&  $\#$ (increasing sat. chains & $\#$ (ascent-free sat. chains\\
 &of length $k$ starting at $\hat{0}$) &of length $k$ starting at $\hat{0}$)\\\hline 
\end{tabular}
\end{center}
\caption{Proposition \ref{proposition:ERwhitneynumbers}}
\label{table:propositionERwhitneynumbers}
\end{table}

\subsection{Constructing Whitney Duals}\label{section:constructSec}
In the following we use edge labelings with the following property.

\begin{definition}\label{definition:ranktwoE}
Let $\lambda$ be an edge labeling on $P$. We say that $\lambda$ has the 
\emph{rank two switching property} if for every maximal chain 
$$\c : ( \hat{0}= x_0 \lessdot x_1 \lessdot \cdots \lessdot x_{k-1} \lessdot x_{k})$$ that 
has an increasing step $\lambda(x_{i-1} \lessdot x_i) < \lambda(x_i \lessdot x_{i+1})$ at rank $i$ 
there is a unique maximal chain 
$$\c' : ( \hat{0}= x_0 \lessdot x_1 \lessdot \cdots \lessdot x_{i-1} \lessdot x_i' \lessdot 
x_{i+1}\lessdot \cdots \lessdot x_{k-1} \lessdot x_{k}),$$
whose labels are the same as the ones from $\c$ except for $\lambda(x_{i-1} \lessdot 
x_i')=\lambda(x_{i} \lessdot x_{i+1})$ and $\lambda(x_{i}' \lessdot x_{i+1})=\lambda(x_{i-1} 
\lessdot x_i).$
 \end{definition}
\begin{remark}
For E-labelings, there is a simpler way to describe the rank two switching property. We can say that $\lambda$ has the rank two 
switching property  provided that for every interval $[x,y]$ with $\rho(y)-\rho(x)=2$, if $\lambda_1\lambda_2$ is the word of labels of the unique increasing 
maximal chain in the interval, then there exists a unique maximal chain in $[x,y]$ whose word of labels is $\lambda_2\lambda_1$.  We choose to give the seemingly more complicated definition because it closely resembles the condition for C-labelings that we provide later.
 \end{remark}

In Figure~\ref{figure:exampleER},  one can see  that the labeling of $\Pi_3$ given in 
Example~\ref{example:labelingposetofpartitions} has the rank two switching property.  Indeed, 
the increasing chain in the unique rank two interval of $\Pi_3$ is labeled by $(1,2)(1,3)$ and 
there is a unique chain labeled by $(1,3)(1,2)$.  In fact, $\Pi_n$ has the rank two switching property for all $n\geq 1$. One can verify 
this using the fact that there are only two types of rank two interval in $\Pi_n$. Each interval is  
isomorphic to $\Pi_3$ or to a boolean algebra of rank two.   

\begin{definition}\label{definition:quadraticexchanges}
Let $P$ be a graded poset and let $\lambda$ be an ER-labeling of $P$ with the rank two switching 
property.
Denote $\M_{[x,y]}$ the set of maximal chains in $[x,y]\subseteq P$ and let 
$$\c : (x = x_0 \lessdot x_1 \lessdot \cdots \lessdot x_{\ell-1} \lessdot x_{\ell}= y)\in 
\M_{[x,y]}$$ be a saturated chain having an ascent $\lambda(x_{i-1}\lessdot x_i)<\lambda(x_i 
\lessdot x_{i+1})$ at rank 
$i$. By the rank two switching property there is an element $\tilde x_i \in [x_{i-1},x_{i+1}]$ such 
that $\lambda(x_{i-1} \lessdot \tilde x_i)=\lambda(x_i \lessdot x_{i+1})$ and 
$\lambda(x_{i-1}\lessdot  x_i)= \lambda( \tilde x_i \lessdot x_{i+1})$. We say that the chain $\c 
\setminus 
\{x_i\} \cup \{\tilde x_i\}$, that is obtained from $\c$ after removing $x_i$ and adding $\tilde 
x_i$, was obtained by a \emph{quadratic exchange at rank $i$}.  We will  use the notation 
$U_i(\c)=\c'$ if $\c'$ is obtained from $\c$ by applying a quadratic exchange at rank level $i$ and 
whenever $\c$ does not have an ascent at rank $i$ we define $U_i(\c)=\c$.
\end{definition}

For every interval $[x,y]$ in $P$ we can define a labeled directed graph $G_{[x,y]}$ whose 
vertex 
set is $\M_{[x,y]}$ and where there is a labeled directed edge $\c_1 
\stackrel{U_i}{\rightharpoonup} \c_2$ if $\c_2$ is obtained from $\c_1$ by a quadratic exchange at 
rank $i$. 
We define also $S(\c):=\{\lambda(x_{i-1}\lessdot x_i)\,\mid\, i\in [\ell]\}$ the multiset of 
labels of the chain $\c$. Note that a quadratic exchange leaves the multiset of labels invariant, 
i.e., 
$S(\c_1)=S(\c_2)$ whenever  $\c_1$ and $\c_2$ are related by a sequence of quadratic exchanges.
Thus in general $G_{[x,y]}$ is not a connected graph, for example two chains in $\M_{[x,y]}$ with 
different multisets of labels belong to different connected components of $G_{[x,y]}$. 

We call a 
pair $(i<j)$ such that $\lambda(x_{i-1}\lessdot x_i) < \lambda(x_{j-1}\lessdot x_j)$ a 
\emph{label inversion} of $\c$.
Note that a quadratic exchange reduces the 
number of label 
inversions of $\c$ and, since the directed edges in $G_{[x,y]}$ are given by quadratic exchanges, 
this 
implies that $G_{[x,y]}$ does not contain directed cycles.

By our construction ascent-free chains in $\M_{[x,y]}$ are precisely the vertices of 
$G_{[x,y]}$ that have outdegree $0$, also known as \emph{sinks}. Indeed, quadratic exchanges can 
only happen at an ascending step of a saturated chain. In particular, by repeatedly applying 
quadratic exchanges we get the following lemma.

\begin{lemma}\label{lemma:switchinglemma}
For every chain $\c \in \M_{[x,y]}$ there exists at 
least one ascent-free maximal chain $\c^{\prime}\in \M_{[x,y]}$ such that $\c$ and $\c^{\prime}$ 
belong to the same connected component of $G_{[x,y]}$. Any such chain $\c^{\prime}$ satisfies 
$S(\c^{\prime})=S(\c)$.
\end{lemma}

Lemma \ref{lemma:switchinglemma} tells us that each connected component of $G_{[x,y]}$ has at least 
one vertex that is a sink, however this sink vertex is not necessarily unique. Ideally, for our 
construction, we would like to have a unique ascent-free chain $\c \in \M_{[x,y]}$ in each connected component of $G_{[x,y]}$. This would imply that the connected components of 
$G_{[x,y]}$ 
are indexed by  
ascent-free chains in $\M_{[x,y]}$. To guarantee this condition, we will use a classical result in 
graph theory and in the study of term rewriting systems, known either as the Diamond Lemma 
or Newman's Lemma (see \cite{Newman1942,Huet1980}). 

We say that a directed graph $G$ is \emph{confluent} if for 
every pair of vertices $x$ and $y$ in the same connected component of $G$ there are 
directed walks $x\rightsquigarrow u$ (a sequence of directed edges $x\rightharpoonup 
z \rightharpoonup \cdots 
\rightharpoonup u $) and $y\rightsquigarrow u$ that meet at a common vertex 
$u$ of $G$. Confluency has the following easy but interesting consequence.

\begin{lemma}\label{lemma:uniquesink}
 If $G$ is a confluent directed graph that does not contain infinite directed walks 
$x_0\rightharpoonup x_1\rightharpoonup x_2\rightharpoonup \cdots $ (or cycles) then every 
connected 
component of $G$ contains a unique sink vertex.
\end{lemma}

We say that $G$ is \emph{locally confluent} if for 
every pair of directed edges $v \rightharpoonup x$ and $v \rightharpoonup y$ there are 
directed walks $x\rightsquigarrow u$ and $y\rightsquigarrow u$ that meet at a common vertex 
$u$ of $G$. 

\begin{lemma}[Newman's Lemma c.f. \cite{Newman1942,Huet1980}]\label{lemma:newman}
A directed graph $G$ without infinite directed walks is confluent if and only if it is 
locally confluent. 
\end{lemma}

Lemma \ref{lemma:newman} simplifies the procedure of checking confluency by restricting to 
local confluency which is an easier condition to test, especially given that $G_{[x,y]}$ is a 
finite graph 
without cycles. In $G_{[x,y]}$ the local confluency condition can be tested by considering two 
labeled directed 
edges $\c \stackrel{U_i}{\rightharpoonup} \c_1$ and $\c \stackrel{U_j}{\rightharpoonup}\c_2$. Note 
that  since $U_i(\c)$ produces a unique chain by the rank two switching property, we may always 
assume $i\neq j$.  If $|i-j|>1$ then we can always obtain a chain $\c^{\prime}$ either as $\c 
\stackrel{U_i}{\rightharpoonup} \c_1 \stackrel{U_j}{\rightharpoonup} \c^{\prime}$ or $\c 
\stackrel{U_j}{\rightharpoonup}\c_2 \stackrel{U_i}{\rightharpoonup} \c^{\prime}$ since the elements 
involved in the ascending steps are not adjacent to each other.    In other words, if $|i-j|>1$, 
then $U_iU_j(\c) = U_jU_i(\c)$ for all chains $\c$ with ascents at $i$ and $j$.
When $j=i+1$ we say that $\c$ has a  
\emph{double-ascent} or that $\c$ has a \emph{critical condition at rank $i$}, i.e., for $\c : (x = 
x_0 \lessdot x_1 \lessdot \cdots \lessdot x_{\ell-1} \lessdot x_{\ell}= y)$  we have that
$$\lambda(x_{i-1}\lessdot x_i)<\lambda(x_i\lessdot x_{i+1})<\lambda(x_{i+1}\lessdot x_{i+2}).$$

\begin{definition}\label{definition:braidrelation}
Let $P$ be a graded poset and let $\lambda$ be an ER-labeling of $P$ with the rank two switching 
property. For every saturated chain $\c \in \M_{[x,y]}$ with a critical condition at rank $i$ we 
can obtain saturated chains $\c^{\prime}$ and $\c^{\prime\prime}$ by removing the ascents at ranks 
$i$ and $i+1$ by a sequence of exchanges $\c 
\stackrel{U_i}{\rightharpoonup} \c_1 \stackrel{U_{i+1}}{\rightharpoonup} \c_2 
\stackrel{U_i}{\rightharpoonup} 
\c^{\prime}$ and $\c 
\stackrel{U_{i+1}}{\rightharpoonup} \c_3 \stackrel{U{i}}{\rightharpoonup}\c_4 
\stackrel{U_{i+1}}{\rightharpoonup} 
\c^{\prime \prime}$.
We say that $\lambda$ satisfies the \emph{braid relation}  if for every such  $\c \in \M_{[x,y]}$ 
we have 
that $\c^{\prime}=\c^{\prime\prime}$.
In 
other words, we have $U_{i}U_{i+1}U_{i}(\c) = U_{i+1}U_{i}U_{i+1}(\c)$ for chains $\c$ which have a 
critical condition at rank  $i$. 
\end{definition}

Note that by the discussion above, if the ER-labeling $\lambda$ satisfies the braid relation then 
the graph $G_{[x,y]}$ is locally confluent, hence we have  the following corollary.

\begin{corollary}\label{corollary:uniquesink}
 Let $P$ be a graded poset and let $\lambda$ be an ER-labeling of $P$ satisfying 
 \begin{itemize}
  \item the rank two switching property, and
  \item the braid relation.
 \end{itemize}

 Then for every interval $[x,y]$ in $P$, we have that 
each connected component of $G_{[x,y]}$ has a unique sink, i.e., a unique ascent-free saturated 
chain.
\end{corollary}

\subsubsection{Quotient posets}
We now turn our attention to quotient posets, which is the other main tool we use for the construction of Whitney duals.  We begin with a definition.

\begin{definition}\label{definition:quotientposet}
Let $P$ be a graded poset and let $\sim$ be an equivalence relation on $P$ such that if $x\sim y$,  
then 
$\rho(x) = \rho(y)$.  We define the \emph{quotient poset} $P/\sim$ to be the set of equivalence 
classes ordered by  $X\leq Y$ if and only if there exists $x\in X$, $y\in Y$ and $z_1,z_2, \dots, 
z_k\in P$ such that 
\begin{equation}\label{equation:equivalencerelation}
x= z_0\leq z_1\sim z_2 \leq\cdots\leq z_{n-1}\sim z_k\leq z_{k+1}= y.
\end{equation}
\end{definition}

For any element $x\in P$ we will denote by $[x]$ its corresponding equivalence 
class in $P/\sim$. 
The next proposition follows from Definition \ref{definition:quotientposet}.

\begin{proposition} \label{proposition:quotientposet}
Let $P$ be a graded poset and let $\sim$ be an equivalence relation on $P$ such that if $x\sim y$,  
then 
$\rho(x) = \rho(y)$.  Then we have the following.
\begin{enumerate}
\item $P/\sim$ is a poset.
\item For $X,Y\in P/\sim$,  $X\cover Y$ if and only if $x\cover y$ for some $x\in X$ and 
$y\in Y$.
\item $P/\sim$ is graded and for $X\in P/\sim$, we have $\rho(X) = \rho(x)$ for all $x\in X$.

\end{enumerate}
\end{proposition}

\begin{proof}Part (2) can be easily verified from the definitions and part (3) is a consequence of 
parts (1) and (2) since they imply that the function $\rho(X) = \rho(x)$, for 
an arbitrary $x\in X$, is a well-defined rank function. 
We show that $P/\sim$ together with the relation $\leq$ satisfies the three properties of a poset.
\begin{enumerate}
\item (Reflexive) This is clear.
\item (Antisymmetric) Suppose that $X\leq Y$ and $Y\leq X$.  Since elements in each 
equivalence class have the same rank, then for any elements $x\in X$ and $y \in Y$ 
we have that $X\leq Y$ implies $\rho(x) \leq \rho(y)$.  Similarly, $Y\leq X$ implies 
$\rho(y)\leq 
\rho(x)$ and hence $\rho(x)=\rho(y)$.  Thus in (\ref{equation:equivalencerelation}) none of the 
inequalities can 
be strict, implying $x\sim y$ and hence $X=Y$.
\item (Transitive)  Suppose that $X\leq Y$ and $Y\leq Z$.  Then by the definition of $\le$ in 
(\ref{equation:equivalencerelation}), $X\leq Y$ implies that there exist $x\in X$, $y\in Y$ and 
$u_1,\dots, u_{k}$ such 
that 
\begin{align}\label{equation:chain1}
 x= u_0\leq u_1\sim u_2 \leq\cdots\leq u_{k-1}\sim u_k\leq u_{k+1}= y.
\end{align}
Also $Y\leq Z$ implies that there are $y'\in Y$, $z\in Z$ and $w_1,\dots, w_s$ such that
\begin{align}\label{equation:chain2}
 y'=w_0\leq w_1\sim w_2 \leq \cdots\leq w_{n-1}\sim w_n\leq w_{n+1}= z.
\end{align}
Since both $y$ and $y'$ are in $Y$, we have that $y\sim y'$. This together 
with (\ref{equation:chain1}) and (\ref{equation:chain2}) imply that $X\leq 
Z$.\qedhere
\end{enumerate}
\end{proof}

The poset $P/\sim$ with the relation defined above is called the \emph{quotient 
poset} of $P$ by the relation $\sim$.

\begin{definition}\label{definition:lambdarelation}
Given a poset $P$, let $C(P)$ denote the poset whose elements are saturated chains 
of $P$  starting at $\hat{0}$ ordered by inclusion. We call $C(P)$ the \emph{chain poset of 
$P$}.  
Figure~\ref{fig:chainPoset} depicts $\Pi_3$ and $C(\Pi_3)$.    If $\c\in C(P)$, we write $e(\c)$ 
for 
the 
maximal element of $\c$, i.e., the element of $P$ where $\c$ terminates. Suppose that $\lambda$ is 
an 
ER-labeling of $P$ with the rank two switching property.
Let $\sim_\lambda$ be the equivalence relation on $C(P)$ defined by $\c_1 \sim_\lambda \c_2$ 
whenever, $\c_1$ and $\c_2$ are in the same connected component 
of 
$G_{[\hat{0},e(\c_1)]}$.  We will use $Q_\lambda(P)$ to denote $ C(P)/\sim$. 
 
\end{definition}

Note that by the nature of the quadratic exchanges, for every  $X \in Q_\lambda(P)$  and $\c,\c'\in 
X$ we have that $e(\c)=e(\c')$.  Thus,  we can also define $e(X)=e(\c)$ for any 
$\c \in X$. 

\begin{figure}
\begin{subfigure}{0.25\textwidth}
                \centering
\begin{tikzpicture}[line join=bevel,scale=1]

\tikzstyle{every node}=[inner sep=3pt, scale=.8, minimum width=4pt]
 \node (n102030) at (0,0)  {$1/ 2/ 3$};

  \node (n12i30) at (-2,2) {$12/ 3$};
  \node (n13i20) at (0,2) {$13/2$};
  \node (n1023i) at (2,2)  {$1/ 23$};

 \node (n123ii) at (0,4){$123$};
 
 \draw (n123ii) -- (n1023i) ;
 \draw (n123ii)-- (n13i20);
  \draw (n123ii) -- (n12i30);  

  \draw [] (n13i20) -- (n102030);
  \draw [] (n12i30)-- (n102030);
  \draw [] (n1023i)--(n102030);

\tikzstyle{every node}=[blue, inner sep=0.4pt, scale=0.7, minimum width=4pt]

\node at (-1.5,1){$(1,2)$};
\node at (0.45,1){$(1,3)$};
\node at (1.5,1){$(2,3)$};

\node at (-1.5,3){$(1,3)$};
\node at (0.45,3){$(1,2)$};
\node at (1.5,3){$(1,2)$};

\end{tikzpicture}

		  \caption{$\Pi_3$ with labeling $\lambda$}
                  \label{fig:examplePi32}

\end{subfigure}~\hspace{-10 pt}  
\begin{subfigure}{0.5\textwidth}
                \centering
\begin{tikzpicture}[scale=0.6]

\tikzstyle{every node}=[inner sep=3pt, scale=.7, minimum width=4pt]
 \node (n102030) at (0,-1)  {$\hat{0}$};

  \node (n12i30) at (-3,2.5) {$\hat{0} \lessdot 12/3$};
  \node (n13i20) at (0,2.5) {$\hat{0} \lessdot 13/2$};
  \node (n1023i) at (3,2.5)  {$\hat{0} \lessdot 1/23$};

 \node[text width=1.6cm] (n123iia) at (-3,6){$\hat{0} \lessdot 12/3\allowbreak \lessdot 123$};

 \node[text width=1.6cm] (n123iiia) at (0,6){$\hat{0} \lessdot 13/2\allowbreak\lessdot 123$};
 
 \node[text width=1.6cm]  (n123iiiia) at (3,6){$\hat{0} \lessdot 1/23 \allowbreak \lessdot 123$};

  \draw (n123iia) -- (n12i30);  
  \draw (n123iiia) -- (n13i20);  
  \draw (n123iiiia) -- (n1023i);  

  \draw [] (n13i20) -- (n102030);
  \draw [] (n12i30)-- (n102030);
  \draw [] (n1023i)--(n102030);

\tikzstyle{every node}=[blue, inner sep=0.4pt, scale=0.7, minimum width=4pt]

\node at (-2.5,1){$(1,2)$};
\node at (0.5,1){$(1,3)$};
\node at (2.5,1){$(2,3)$};

\node at (-3.5, 4){$(1,3)$};
\node at (0.5,4){$(1,2)$};
\node at (3.5,4){$(1,2)$};
\end{tikzpicture}
                \caption{$C(\Pi_3)$}
                  \label{fig:exampleCPI3}

\end{subfigure}~	\hspace{-80  pt}
\begin{subfigure}{.5\linewidth}
                \centering

\begin{tikzpicture}[scale=0.6]

\tikzstyle{every node}=[inner sep=3pt, scale=.8, minimum width=4pt]

\node (a) at (1,-1) { $\emptyset$};

\node (b) at (-2,2.5)  {$\{(1,2)\}$};

\node (c) at (1,2.5)  {$\{(1,3)\}$};

\node (d) at (4,2.5)  {$\{(2,3)\}$};

\node (e) at (-.5,6)  {$\{(1,2), (1,3)\}$};
\node (f) at (4,6)  {$\{(1,2), (2,3)\}$};

 \draw (c)-- (a) -- (b);
\draw (a)--(d)--(f);
\draw (a)--(d)--(f);
\draw (c)--(e)--(b);
 
\tikzstyle{every node}=[blue, inner sep=0.4pt, scale=0.7, minimum width=4pt]

\node at (-1.4,1){$(1,2)$};
\node at (1.6,1){$(1,3)$};
\node at (3.5,1){$(2,3)$};

\node at (-2, 4){$(1,3)$};
\node at (1.35,4){$(1,2)$};
\node at (4.75,4){$(1,2)$};

\end{tikzpicture}
 \caption{$Q_\lambda(\Pi_3)$ with labeling $\lambda^*$}
                  \label{fig:exampleQPi}

\end{subfigure}~

         \caption{}
  \label{fig:chainPoset}

\end{figure}

\begin{example} Consider again $\Pi_3$ and its chain poset $C(\Pi_3)$ shown in 
Figure~\ref{fig:exampleCPI3}. Every element of $C(\Pi_3)$ is in its own equivalence class 
except for the chains 
$1/2/3\cover 12/3\cover 123$ and $1/2/3\cover 13/2\cover 123$ since 
$1/2/3\cover 12/3\cover 123 \stackrel{U_1}{\rightharpoonup} 1/2/3\cover 13/2\cover 123$ is the only 
directed edge in $G_{[1/2/3, 123]}$. Taking the quotient to obtain $Q_\lambda(\Pi_3)$, we get the  
poset in Figure~\ref{fig:exampleQPi}, where we have  identified the equivalence classes by the 
underlying set of labels on the chains.  By comparing with Figure~\ref{figure:exampleER}, one can 
observe that $\ISF_3$ and $Q_\lambda(\Pi_3)$ are isomorphic.
\end{example}

Note that by the definition of $\sim_{\lambda}$, if $\lambda$ satisfies the conditions of Corollary 
\ref{corollary:uniquesink} then each equivalence class $X \in 
Q_{\lambda}(P)$ contains a unique ascent-free maximal chain in $[\hat{0},e(X)]$. In fact this is a
correspondence between ascent-free saturated chains starting at $\hat{0}$ in $P$ of 
length $k$ and equivalence classes in $Q_\lambda(P)$ of rank $k$. Hence, using Proposition 
\ref{proposition:ERwhitneynumbers} and the definition of 
Whitney numbers of the second kind in Equation \ref{definition:whitneysecondkind} we are able 
to conclude at this point that $|w_k(P)| = W_k(Q_\lambda(P))$.

We now turn our 
attention to the task of satisfying the other half of Definition \ref{definition:whitneyduals}, that is, we 
would like to have in addition that  $W_k(P) = |w_k(Q_\lambda(P))|$.  This will allow us 
to conclude that 
$Q_\lambda(P)$ is a Whitney Dual 
of $P$. Our strategy will be to define an edge labeling $\lambda^*$ on $Q_\lambda(P)$ that under 
certain conditions is an ER$^*$-labeling. We will then show that the saturated 
chains from $\hat{0}$ in $Q_{\lambda}(P)$ under the newly defined ER$^*$-labeling and the ones 
in $P$ under the labeling $\lambda$ are in a label-preserving bijection. This together with 
Proposition \ref{proposition:ERwhitneynumbers} imply that $P$ and $Q_{\lambda}(P)$ are Whitney 
duals.

To define this labeling recall that, by  definition,
$\c_1 \sim_\lambda\c_2$ implies  $S(\c_1)= S(\c_2)$.   In 
light of this, we will use $S(X)$ to denote the multiset of labels in any chain in $X$. Moreover, 
if $X\cover Y$ in  $Q_\lambda(P)$  then there exists a unique element in $S(Y)\setminus S(X)$.  
Define an edge labeling $\lambda^*$ on $Q_\lambda(P)$ by
\begin{equation}\label{equation:lambdastar}
\lambda^*(X\cover Y)= S(Y)\setminus S(X).
\end{equation}
 This edge labeling for $Q_\lambda(\Pi_3)$ appears in 
Figure~\ref{fig:exampleQPi}.  We will now consider a pair of structural lemmas that will be useful 
in the 
following discussion.

\begin{lemma}\label{lemma:main}
 Let $X_1\cover X_2 \cover \cdots \cover X_k$ be a saturated chain in $Q_\lambda(P)$.  
\begin{enumerate}
\item [(a)] We have that $e(X_1)\cover e(X_2) \cover \cdots \cover e(X_k)$  is a saturated chain in 
\item[(b)] For any chain $\c \in X_1$ we have that $\c \cup \{e(X_2),\dots,e(X_i)\} \in X_i$ for 
all $i=1,\dots k$.
\item[(c)]  If $X_1=X_1'\cover X_2'\cover\cdots \cover X_k'$ is another saturated chain with 
$e(X_i) = e(X_i')$ for all $i$, then $X_i = X_i'$ for all $i$.

\end{enumerate}
\end{lemma}
\begin{proof}
First we show (a) holds. Consider the cover $X_{i-1}\cover X_i$.  By 
Proposition~\ref{proposition:quotientposet}, there exists a $\c \in X_{i-1}$ and $\d\in X_{i}$ with 
$\c \cover \d$ in $C(P)$.   By the definition of the poset $C(P)$ we have that  $\d= \c\cup 
\{e(X_i)\}$ and so $e(X_{i-1}) = 
e(\c)\cover 
e( \d) = e(X_i)$.  Thus, $e(X_1)\cover e(X_2)\cover \cdots \cover e(X_k)$ is a saturated chain in 
$[e(X_1), e(X_k)]$.

To show (b), we use induction on $k$. The case when $k=1$ is trivial. Now suppose $k \geq 2$.  By 
induction, for any  $\c \in X_1$ we have $\c \cup \{e(X_2),\dots,e(X_{i})\} 
\in X_{i}$ for $i \le k-1$. We want to show that $\c \cup \{e(X_2),\dots, e(X_{k-1}),e(X_k)\} \in X_k$.  First, 
note that by part (a), $\c \cup \{e(X_2),\dots,e(X_{k-1}), e(X_k))\}$ is a saturated chain in $P$. Now 
since 
$X_{k-1}\cover X_{k}$,  Proposition~\ref{proposition:quotientposet} implies that there exists $\d 
\in X_{k-1} $ and $\d'\in X_k$ with $\d' = \d \cup\{e(X_k)\}$.  Since  $\c \cup 
\{e(X_2),\dots,e(X_{k-1})\}$ and $\d$ both belong to $X_{k-1}$,  we know these two chains are 
related by a sequence of 
quadratic moves.  It follows that $\c \cup \{e(X_2),\dots,e(X_{k-1}), e(X_{k})\}$ and $\d'$ are 
related by the exact same moves and so are in the same equivalence class.  Thus, $\c \cup 
\{e(X_2),\dots,e(X_{k-1}), e(X_k))\}\in X_k$.

Note that part (c) follows directly from part (b).
\end{proof}

\begin{lemma}\label{lemma:saturatedchainssamelabelsinPandQ}
 Let $\C:(X_1\cover X_2 \cover \cdots \cover X_k)$ be a saturated chain in $Q_\lambda(P)$ and  let
$\c:(e(X_1)\cover e(X_2) \cover \cdots \cover e(X_k))$ be the corresponding saturated chain in $P$. 
The 
words of labels of these two saturated chains under their respective labelings are equal, i.e., 
$\lambda^*(\C)=\lambda(\c)$.
\end{lemma}
\begin{proof}
 By Equation~\ree{equation:lambdastar}, $\lambda^*(X_{i-1}\cover X_{i})= 
S(X_i)\setminus S(X_{i-1}) = S(\d')\setminus S(\d) = \lambda(e(X_{i-1})\cover e(X_i))$, where
$\d \in X_{i-1}$ and $\d' \in X_{i}$ are such that $\d \lessdot \d'$.
\end{proof}

To prove that the labeling $\lambda^*$ given in  Equation (\ref{equation:lambdastar}) is an 
ER${}^*$-labeling  of $Q_\lambda(P)$ we will need the additional following condition.

\begin{definition}\label{definition:cancellative}
Let $\lambda$ be an ER-labeling on $P$ with the rank two switching property. We say 
that $\lambda$ is \emph{cancellative} if for every $z<x<y$ 
in $P$,
$\c\in \M_{[z,x]}$ and $\c_1,\c_2 \in \M_{[x,y]}$ we have that
$$\c\cup \c_1 \sim_{\lambda} \c \cup \c_2 \text{ implies }  \c_1 \sim_{\lambda} \c_2.$$ 
\end{definition}

\begin{proposition}\label{proposition:QERStarLabel}
Let $P$ be a graded poset and let $\lambda$ be an ER-labeling of $P$ satisfying
 \begin{itemize}
  \item the rank two switching property, 
  \item the braid relation, and
  \item the cancellative property.
 \end{itemize}
 
Then the labeling $\lambda^*$ of $Q_\lambda(P)$ given by  Equation~\ree{equation:lambdastar} is an 
ER$^*$ labeling.
\end{proposition}
\begin{proof}
Let $[X,Y]$ be an interval in $Q_{\lambda}(P)$.  We will show that there is a unique
ascent-free maximal chain in $[X,Y]$.

{\bf The interval $[X,Y]$ contains an ascent-free maximal chain.}  
Pick 
any $\c \in X$ and let $X=X_1\cover X_2\cover \cdots \cover X_k=Y$ be a maximal chain in $[X,Y]$.  
By 
Lemma~\ref{lemma:main} parts (a) and (b), $e(X_1)\cover e(X_2) \cover \cdots \cover 
e(X_k)$ is a saturated chain in $[e(X), e(Y)]$ and $\c\cup \{e(X_2), e(X_3),\dots, e(X_k)\} \in Y$. 
The chain  $e(X_1)\cover 
e(X_2) \cover \cdots \cover e(X_k)$ 
may not be ascent-free in $[e(X), e(Y)]$, but it is related to one by a sequence of quadratic 
exchanges.    Suppose 
the related ascent-free chain is $e(X)=x_1'\cover x_2'\cover \cdots \cover x_k'=e(Y)$.   For each $1\leq i \leq k$, let $X_i'$ 
be the equivalence class containing $\c \cup \{x_2',x_3', \dots, x_i'\}$.  Then it must be the 
case that $X_k' = Y$ since we are using quadratic exchanges on $\c\cup \{e(X_2), e(X_3),\dots, 
e(X_k)\}\in Y$.  
In $C(P)$ we have that $\c \cover \c \cup\{x_2'\} \cover 
\cdots\cover \c \cup \{x_2',x_3', \dots, x_k'\}$ and hence, by the quotient 
poset definition of $Q_{\lambda}(P)$ we have that $X=X_1'\cover X_2'\cover \cdots 
\cover X_k'=Y$ is a maximal chain in $[X,Y]$. 
 Moreover, by  Lemma~\ref{lemma:saturatedchainssamelabelsinPandQ}, the labels along this 
chain are the same as along 
$x_1'\cover x_2'\cdots \cover x_k'$.  It follows that $[X,Y]$ has an ascent-free maximal chain.

{\bf The ascent-free maximal chain found above is unique.} 
Suppose this 
was not the case and that $X=X_1\cover X_2\cover \cdots \cover X_k=Y$ and $X=X_1'\cover X_2'\cover 
\cdots \cover X_k'=Y$ are both ascent-free maximal chains in $[X,Y]$.  Pick any $\c \in X$.  Then 
by  Lemma~\ref{lemma:main} part (b), $\c\cup \{e(X_2), e(X_3),\dots, e(X_k)\}$ and $\c\cup 
\{e(X_2'), e(X_3'),\dots, e(X_k')\}$ are both chains in $Y$. By the cancellative property $\c\cup 
\{e(X_2), e(X_3),\dots, e(X_k)\} \sim \c\cup \{e(X_2'), e(X_3'),\dots, e(X_k')\}$ implies that 
$e(X_1)\cover e(X_2)\cover \cdots \cover e(X_k)$ and $e(X_1')\cover e(X_2') \cover \cdots \cover e(X_k')$ 
are related by quadratic exchanges.  It follows that both of these chains are in the same  
connected component of $G_{[e(X),e(Y)]}$ and by Lemma~\ref{lemma:saturatedchainssamelabelsinPandQ}, 
they are both ascent-free.  However, Corollary~\ref{corollary:uniquesink}, asserts  that there is a 
unique 
ascent-free maximal chain  in each connected component implying that $e(X_i) = e(X_i')$ for 
all $i$.   Applying Lemma~\ref{lemma:main} part (c) we conclude that $X_i =X_i'$ for all $i$.  It 
follows that there is a unique ascent-free maximal chain in each interval  and so $\lambda^*$ is an 
ER$^*$ labeling.
\end{proof}

Before stating the main theorem of this section let us discuss a very important class of labelings 
that satisfy the conditions of Proposition \ref{proposition:QERStarLabel}. 

\begin{theorem}\label{theorem:standardEWlabelings}
Let $\lambda$ be an $ER$-labeling satisfying 
 \begin{itemize}
  \item the rank two switching property, and
  \item in each interval each ascent-free maximal chain has a unique word of labels; 
 \end{itemize}
 then $\lambda$ satisfies
  \begin{itemize}
  \item the rank two switching property, 
  \item the braid relation, and
  \item the cancellative property.
 \end{itemize}
\end{theorem}
\begin{proof}
 We need to check that $\lambda$ satisfies the braid relation and the cancellative property. 

{\bf The braid relation is satisfied by $\lambda$.} 
Indeed, let $\c$ be a saturated chain that has a 
critical condition at rank $i$.  We obtain saturated chains $\c^{\prime}$ and $\c^{\prime\prime}$ 
by 
removing the ascents at ranks $i$ and $i+1$ by a sequence of exchanges $\c 
\stackrel{U_i}{\rightharpoonup} \c_1 \stackrel{U_{i+1}}{\rightharpoonup} \c_2 
\stackrel{U_i}{\rightharpoonup} 
\c^{\prime}$ and $\c \stackrel{U_{i+1}}{\rightharpoonup} \c_3 \stackrel{U_i}{\rightharpoonup} \c_4 
\stackrel{U_{i+1}}{\rightharpoonup} 
\c^{\prime \prime}$. By the definition of a quadratic exchange we have that  
$\c^{\prime}$ and $\c^{\prime\prime}$, when restricted to the interval between  ranks $i-1$ and 
$i+2$, have the same ascent-free word of labels and 
hence $\c^{\prime}=\c^{\prime\prime}$.

{\bf The cancellative property is satisfied by $\lambda$.} We are going to prove, using induction on the value of 
$\rho(x)-\rho(z)$, that for every $z<x<y$ in $P$, $\c\in \M_{[z,x]}$ and $\d,\d' \in \M_{[x,y]}$ we 
have that $\c\cup \d \sim_{\lambda} \c \cup \d' \text{ implies }  \d \sim_{\lambda} \d'.$

When $\rho(x)-\rho(z)=1$ we have  $\c:(z\cover x)$ and without loss of generality we can 
assume, 
perhaps after applying enough quadratic exchanges, that $\d$ and $\d'$ are ascent-free. Hence 
the only possible ascents must happen in the step $z\cover x \cover d_1$ for $\c \cup \d$ and in 
the 
step $z\cover 
x \cover d'_1$ for $\c\cup \d'$. Note that then quadratic exchanges on $\c\cup \d$ can only shuffle 
the label $\lambda(z \cover x)$ across the word $\lambda(\d)$ and quadratic exchanges on 
$\c\cup \d'$ 
can only shuffle $\lambda(z \cover x)$ across the word $\lambda(\d')$. Because we have $\c\cup \d 
\sim_{\lambda} \c 
\cup \d'$ and the braid relation, Corollary \ref{corollary:uniquesink} implies that the ascent-free word of labels obtained after all the quadratic exchanges have been 
applied to both $\c\cup \d$ and $\c\cup \d'$ is the same. Hence we originally had $\lambda(\d)=\lambda(\d')$. Uniqueness 
of the ascent-free word of labels implies then $\d=\d'$.

Now consider the case when $\rho(x)-\rho(z)=k>1$ and we have chains $\c\in \M_{[z,x]}$ and 
$\d,\d' \in \M_{[x,y]}$ such that $\c\cup \d \sim_{\lambda} \c \cup \d'$. Note first that 
$z<c_{k-1}<y$, and we will consider instead the saturated chains $\c^{\prime}:=\c\setminus \{x\} 
\in \M_{[z,c_{k-1}]}$, $\m:=\{c_{k-1}\}\cup 
\d \in \M_{[c_{k-1},y]}$ and $\m^{\prime}=\{c_{k-1}\}\cup \d' \in \M_{[c_{k-1},y]}$.  We have that
$\rho(c_{k-1})-\rho(z)=\rho(x)-\rho(z)-1$ and $\c^{\prime}\cup \m=\c\cup \d \sim_{\lambda} \c 
\cup \d'=\c^{\prime}\cup \m^{\prime}$, hence by induction we conclude that $\m 
\sim_{\lambda} \m^{\prime}$. Now if 
we consider $c_{k-1}\cover x < y$, the argument above says that$\{c_{k-1}\cover x\}\cup \d= 
\m \sim_{\lambda} \m^{\prime}=\{c_{k-1}\cover x\}\cup \d'$ in $\M_{[c_{k-1},y]}$ and since
$\rho(x)-\rho(c_{k-1})=1$ we are back in the base case that implies $\d \sim_{\lambda} \d'$.
\end{proof}

\begin{remark}
Note that in  Theorem \ref{theorem:standardEWlabelings} we can replace the second condition by the stronger requirement that  maximal chains have unique word of labels. 
\end{remark}

We are now in a position to provide names to the 
type of edge labelings that allow us to construct Whitney duals. We call these labelings 
EW-labelings, where the letter ``W" comes from  the fact that they provide sufficient conditions to 
construct Whitney duals.

\begin{definition}\label{definition:EW}
Let $\lambda:\E(P)\rightarrow \Lambda$ be an ER-labeling of $P$.  We 
say $\lambda$ is an \emph{EW-labeling} if it satisfies
 \begin{itemize}
  \item the rank two switching property, and
  \item in each interval each maximal chain has a unique word of labels.
 \end{itemize}
\end{definition}

\begin{definition}\label{definition:generalizedEW}
Let $\lambda:\E(P)\rightarrow \Lambda$ be an ER-labeling of $P$.  We 
say $\lambda$ is a \emph{generalized EW-labeling} if it satisfies
 \begin{itemize}
  \item the rank two switching property, 
  \item the braid relation, and
  \item the cancellative property.
 \end{itemize}
\end{definition}

The following proposition provides an insight into the construction process that generates 
$Q_{\lambda}(P)$ from $P$ using $\lambda$. According to Proposition 
\ref{proposition:bijectionbetweensaturatedchains}, we can   think  of $Q_{\lambda}(P)$ 
as a poset that is obtained from $P$ by ``pulling apart" saturated chains 
from $\hat{0}$.  

\begin{proposition} \label{proposition:bijectionbetweensaturatedchains}
Let $\lambda$ be a generalized EW-labeling 
of $P$.  There is a label preserving bijection from the set of saturated 
chains from $[\hat{0}]$ of length $k$ in $Q_{\lambda}(P)$ and the set of saturated chains from 
$\hat{0}$ of length $k$ in $P$. In particular, there is a label preserving bijection 
$\M_{Q_{\lambda}(P)}\rightarrow \M_{P}$ between maximal 
chains.
\end{proposition}

\begin{proof} 
Fix $k$ and let $S_{Q,k}$ be the set of saturated 
chains from $[\hat{0}]$ of length $k$ in $Q_{\lambda}(P)$ and $S_{P,k}$ be the set of saturated 
chains 
from $\hat{0}$ of length $k$ in $P$. Let 
$\varphi : S_{Q,k} \rightarrow  S_{P,k}$  be defined by $$\varphi([\hat{0}]=X_0\cover 
X_1\cover \cdots \cover X_k) = (\hat{0}= e(X_0)\cover e(X_1)\cover \cdots \cover e(X_k)).$$  
By Lemma~\ref{lemma:main} part (a) we know $\varphi$ is well-defined and by 
Lemma~\ref{lemma:saturatedchainssamelabelsinPandQ} we know that $\varphi$ preserves the word of 
labels.  By 
Lemma~\ref{lemma:main} part (c),  $\varphi$ is injective.

Finally, we show that $\varphi$ is  surjective.    Let $\hat{0}=x_0\cover x_1\cover \cdots \cover 
x_k \in S_{P,k}$.  Let $\c_i = \{x_0,x_1,\dots, x_i\}$, then by definition $\c_0\cover 
\c_1\cover \cdots \lessdot\c_k$ is in $C(P)$.  By the definition of a quotient poset, 
$[\c_0]\cover [\c_1]\cover 
\cdots \cover  [\c_k]$ is in $S_{Q,k}$ and it  is a preimage of 
$\hat{0}=x_0\cover 
x_1\cover \cdots \cover x_k$.
\end{proof}

Definition  \ref{definition:whitneyduals}  
together with Propositions \ref{proposition:ERwhitneynumbers}, 
\ref{proposition:QERStarLabel}  and \ref{proposition:bijectionbetweensaturatedchains}
imply our main theorem, that we are ready to state in the 
language of Definitions \ref{definition:EW} and \ref{definition:generalizedEW}.

\begin{theorem}\label{theorem:EW}
Let $P$ be a graded poset with a generalized EW-labeling $\lambda$.  Then 
$Q_\lambda(P)$ is a Whitney dual of $P$.
\end{theorem}

\begin{remark}In \cite{DleonHallam2017} the authors defined the related concept of 
an $\overline{\mbox{EW}}$-labeling. The reason for the use of an overline in that article is that 
the conditions on 
those labelings are more restrictive but imply the conditions of 
Definition 
\ref{definition:EW}. 
While the definition of 
an  $\overline{\mbox{EW}}$-labeling greatly simplifies the proofs of the theorems we have presented 
here, there are posets with EW-labelings, but no known  $\overline{\mbox{EW}}$-labeling.   See 
Section~\ref{weightedPartSec} for some examples.

\end{remark}

\subsection{Chain-edge labelings}\label{section:CWlabelings}

We show in this subsection that the definitions and constructions given for EW-labelings also extend to the 
generality of chain-edge labelings with the same consequences with respect to Whitney duality.

\begin{definition}
Let $\M\E(P)$ denote the set of pairs $(\m,e)$ where $\m$ is a maximal chain in $P$ and $e$ is an 
edge in the Hasse diagram of $\m$. A \emph{chain-edge labeling} or 
\emph{C-labeling} of $P$ is a map $\lambda: \M\E(P)\rightarrow \Lambda$, where $\Lambda$ is some 
poset of labels, satisfying the condition 
that whenever two maximal chains coincide along the bottom $d$ edges then their labels also 
coincide on these $d$ edges.   
\end{definition}
\begin{definition}\label{definition:rooted_interval} A \emph{rooted interval} $[x,y]_{\r}$ in $P$ 
is a pair $([x,y],\r)$ where 
$[x,y]$ is 
an interval in $P$ and $\r$ is a saturated chain from 
$\hat{0}$ to $x$.
\end{definition}

The rationale behind Definition \ref{definition:rooted_interval} is that in a C-labeling, when we 
want to restrict to a smaller 
interval $[x,y]$ in $P$, the labels depend on the initial saturated chain $\r$ from 
$\hat{0}$ to $x$.

\begin{definition}
A C-labeling $\lambda$ of $P$ naturally induces a C-labeling $\lambda_{\r}$ in a rooted 
interval $[x,y]_{\r}$ by letting the labels of a maximal chain $\c$ of $[x,y]$ be the ones 
corresponding to the maximal chain $\r \cup \c$ in $[\hat{0},y]$. A  C-labeling is a 
\emph{CR-labeling} if in every rooted interval $[x,y]_{\r}$ there is a unique increasing maximal 
chain.
\end{definition}
It was shown by Bj\"orner and Wachs \cite{BjornerWachs1996} that a CR-labeling on a poset $P$ 
has the same implications with respect to M\"obius numbers as described in Theorem \ref{theorem:muER} in the case of an ER-labeling.
Hence we can describe the 
Whitney numbers of a poset with a CR-labeling by the enumeration of 
saturated chains in  the same way than Proposition \ref{proposition:ERwhitneynumbers}. 
We have that $|w_k(P)|$ is the number of 
ascent-free saturated chains starting at $\hat{0}$  of length $k$ and
$|W_k(P)|$ is the number of increasing saturated chains starting at $\hat{0}$ of length $k$ as 
before.

\begin{definition}\label{definition:ranktwoC} 
We say a C-labeling has the \emph{rank two switching property} if for every maximal chain 
of the 
form
$$\m:(\hat{0}= m_0 \lessdot m_1 \lessdot \cdots \lessdot m_{k})$$ with an 
ascending step $\lambda(\m, m_{i-1}\lessdot m_i)< \lambda(\m, m_{i} \lessdot m_{i+1})$   at some 
rank $i<k$ there is a unique 
element $m'_i \neq m_{i}$  such that the chains $\m$ and
$$\m':(\hat{0}= m_0 \lessdot m_1  \lessdot \cdots \lessdot m_{i-1}\lessdot m'_i \lessdot m_{i+1}  
\cdots\lessdot m_{k})$$
have the same word of labels except at rank $i$ where $\lambda(\m',m_{i-1}\lessdot m'_i)= 
\lambda(\m,m_{i}\lessdot m_{i+1})$ and 
$\lambda(\m',m'_{i}\lessdot m_{i+1})=\lambda(\m, m_{i-1}\lessdot m_i)$. Moreover, the rank two switching property requires a consistency condition, that for any other 
maximal  
chain $\widetilde \m$ that coincides with $\m$ in the first $i+2$ elements ($\widetilde m_j=m_j$ 
whenever $j\le i+1$) the choice of the unique element also coincides, i.e., $\widetilde 
m'_i=m'_i$. 
\end{definition}

\begin{remark}
Note that in Definition \ref{definition:ranktwoC} there is an additional 
consistency condition that is not present in Definition \ref{definition:ranktwoE}. This condition 
guarantees that the restriction of the rank two switching property for intervals of the form 
$[\hat{0},y]$ is well-defined when $y$ any  element of $P$ which is not necessarily maximal.
\end{remark}

In the situation of Definition \ref{definition:ranktwoC} we say that the chain 
$\m'$, is obtained from $\m$ by a \emph{quadratic exchange at rank 
$i$} and will  use the notation 
$\m'=U_i(\m)$. If $\m$ does not have an ascent at rank $i$ we define $U_i(\m)=\m$.

As in the discussion after Definition \ref{definition:quadraticexchanges}, we define 
graphs $G_{[x,y]_{\r}}$ given by quadratic exchanges but this time the elements are maximal chains 
in 
a rooted interval $[x,y]_{\r}$. To check confluency in $G_{[x,y]_{\r}}$ we invite the reader to 
verify that by the consistency condition of the rank two switching property for C-labelings in 
Definition \ref{definition:ranktwoC}, if a 
maximal chain $\c$ has ascents at ranks $i$ and $j$ with $|i-j|>1$ it is necessarily true that 
$U_iU_j(\c) = U_jU_i(\c)$.

We say that $\c : (\hat{0} = 
x_0 \lessdot x_1 \lessdot \cdots \lessdot x_{\ell-1} \lessdot x_{\ell}= y)$ has a 
\emph{double-ascent} or a \emph{critical condition at rank $i$} if 
$$\lambda(\c, x_{i-1}\lessdot x_i)<\lambda(\c, x_i\lessdot x_{i+1})<\lambda(\c, x_{i+1}\lessdot 
x_{i+2}).$$
For every saturated chain $\c$ from $\hat{0}$ with a critical condition at rank $i$ we say that 
$\lambda$ satisfies the \emph{braid relation}  if we have $U_{i}U_{i+1}U_{i}(\c) = 
U_{i+1}U_{i}U_{i+1}(\c)$.

When we have a  CR-labeling $\lambda$ satisfying the rank two switching property and the braid 
relation, we ensure local confluency in $G_{[x,y]_{\r}}$ and hence, by Lemma \ref{lemma:newman}, the conclusion of Corollary 
\ref{corollary:uniquesink} holds.  Therefore  each connected component of $G_{[x,y]_{\r}}$ has a 
unique 
sink, i.e., a unique ascent-free saturated 
chain. We use the exact same definitions of chain 
poset $C(P)$ and quotient poset 
$Q_{\lambda}(P)$ given in Definition~\ref{definition:lambdarelation}; and also give the same 
definition of the edge labeling $\lambda^*$ on $Q_\lambda(P)$ of Equation \ref{equation:lambdastar}, 
that is,
$$\lambda^*(X\cover Y)= S(Y)\setminus S(X). $$  
Note that $\lambda^*$ is actually an E-labeling on 
$Q_{\lambda}(P)$, 
i.e. does not depend on maximal chains, even though the labeling $\lambda$ of $P$ is a C-labeling. 
We 
would like to conclude that $\lambda^*$ is also an ER$^*$-labeling in this case. To do this, we show 
that the lemmas and propositions for 
E-labelings that appeared in the previous subsection still hold in the C-labeling scenario.  

It is straightforward to verify that Lemma \ref{lemma:main} is still valid in our new 
setting, but we need a C-labeling analogue (Lemma 
\ref{lemma:saturatedchainssamelabelsinPandQforClabeling} below) of Lemma 
\ref{lemma:saturatedchainssamelabelsinPandQ} to be able to produce  an analogue of Proposition 
\ref{proposition:bijectionbetweensaturatedchains}. 

\begin{lemma}\label{lemma:saturatedchainssamelabelsinPandQforClabeling}
Let $\D:(X_1\cover X_2 \cover \cdots \cover X_k)$ be a saturated chain in $Q_\lambda(P)$ and  let
$\d:(e(X_1)\cover e(X_2) \cover \cdots \cover e(X_k))$  be the corresponding saturated chain 
in $P$. Let $\c \in X_1$ and let $\lambda_{\c}$  be the induced labeling coming from $\lambda$ on 
the rooted 
interval $[e(X_1),e(X_k)]_{\c}$. Then the  words of labels of $\D$ and $\d$ under 
their respective labelings are equal, i.e., $\lambda^*(\D)=\lambda_{\c}(\d)$.
\end{lemma}
\begin{proof}
By Lemma \ref{lemma:main} we have that for all $i \le k$, $\c \cup 
\{e(X_1),e(X_2),\cdots ,e(X_{i})\}\in X_{i}$. Then 
by Equation~\ree{equation:lambdastar}, we have that 
\begin{align*}
 \lambda^*(X_{i-1}\cover X_{i})&= 
S(X_i)\setminus S(X_{i-1})\\
&= S(\c \cup \{e(X_1),e(X_2),\cdots ,e(X_{i})\})\setminus 
S(\c\cup \{e(X_1),e(X_2),\cdots e(X_{i-1})\}) \\
&= \lambda(\c \cup \d, 
e(X_{i-1})\cover e(X_i)).\qedhere
\end{align*}
\end{proof}

 \begin{definition}\label{definition:cancellativeC}
Let $\lambda$ be a CR-labeling on $P$ with the rank two switching property. We say 
that $\lambda$ is \emph{cancellative} if for every $z<x<y$ 
in $P$, $\r \in \M_{[\hat{0},z]}$, $\c\in \M_{[z,x]}$ and $\c_1,\c_2 \in \M_{[x,y]}$ we have that
$$\c\cup \c_1 \sim_{\lambda_{\r}} \c \cup \c_2 \text{ implies }  \c_1 \sim_{\lambda_{\r \cup \c}} 
\c_2.$$ 
\end{definition}

The reader can verify that the proof of Proposition \ref{proposition:QERStarLabel} is still valid 
when $\lambda$ is a CR-labeling with the corresponding properties, replacing whenever necessary 
intervals in $P$ by rooted intervals to take into account the labeling. We then have that a 
CR-labeling $\lambda$ of $P$ with the rank two switching property, the braid relation and the 
cancellative property induces the ER$^*$ labeling $\lambda^*$ of $Q_\lambda(P)$ given by  
Equation~\ree{equation:lambdastar}. 
\begin{definition}\label{definition:generalizedCWandCW}
A \emph{generalized CW-labeling}  $\lambda:\M\E(P)\rightarrow \Lambda$ is a CR-labeling  that 
satisfies
 \begin{itemize}
  \item the rank two switching property, 
  \item the braid relation, and
  \item the cancellative property.
 \end{itemize}

We say $\lambda$ is a \emph{CW-labeling} if it satisfies
 \begin{itemize}
  \item the rank two switching property, and
  \item in each rooted interval each ascent-free maximal chain has a unique word of labels.
 \end{itemize}
\end{definition}

 As with EW-labelings, it turns out in this scenario  that every CW-labeling is a generalized CW-labeling but we 
do not know if the converse is true.

We are now ready to state the main theorem for (generalized) 
CW-labelings that follows from Definition  \ref{definition:whitneyduals}  
together with the C-labeling analogues of Propositions \ref{proposition:ERwhitneynumbers}, 
\ref{proposition:bijectionbetweensaturatedchains}, and 
\ref{proposition:QERStarLabel}.

\begin{theorem}\label{theorem:CW}
Let $P$ be a graded poset with a (generalized) CW-labeling $\lambda$.  Then 
$Q_\lambda(P)$ is a Whitney dual of $P$.
\end{theorem}

\section{The Whitney dual $Q_\lambda(P)$}\label{section:whitneydual}
In this section we first give a formula for the M\"obius function of 
$Q_\lambda(P)$ given that $\lambda$ is a generalized CW-labeling. In the second part of the section 
we provide a somewhat simpler description 
of $Q_\lambda(P)$ given that $\lambda$ is a CW-labeling. This characterization only applies to 
CW-labelings in the strict sense of the definition, so our description does not include all 
generalized CW-labelings.

\subsection{The M\"obius function of $Q_{\lambda}(P)$}

We can characterize the M\"obius numbers of $Q_{\lambda}(P)$ using the fact that $\lambda^*$ of Equation \ref{equation:lambdastar} is 
an ER${}^*$-labeling. An interesting fact is that $Q_{\lambda}(P)$ belongs to the famous family of posets whose M\"obius 
numbers are $0$ or $\pm 1$. Hersh and M\'esz\'aros in \cite{HershMeszaros2017} have defined a 
family 
of edge labelings, that they coined SB-labelings, and that allowed them to conclude that a lattice 
with 
such a labeling has M\"obius numbers  $0$ or $\pm 1$. Their result partially answers a question 
posed by  Bj\"orner  and Greene on why posets with these M\"obius values are plentiful in 
combinatorics. It is still an open problem to find a characterization of when the posets
$Q_{\lambda}(P)$ are lattices. The family of posets $Q_{\lambda}(P)$  provide a 
plethora of examples of posets whose M\"obius 
numbers are $0$ or $\pm 1$.

Fix $x\in P$ and let $X^1, X^2,\dots, X^n$ be the different elements of $Q_\lambda(P)$ such that for all $i$, $e(X^i)=x$. Since there is exactly one increasing maximal chain 
in $[\hat{0}, x]$, Definition \ref{definition:lambdarelation} implies that there is exactly one $X^i$ 
that contains this maximal increasing chain and all other $X^j$ for $j\neq i$ do not contain any 
increasing maximal chain.   As we see in the next proposition it is exactly this class which has a 
nonzero M\"obius value $\mu([[\hat{0}],X])$ in $Q_\lambda(P)$.

\begin{proposition}\label{proposition:mobiusQlambda}
 Let $\lambda$ be a generalized CW-labeling of $P$; $X,Y \in Q_{\lambda}(P)$ and 
$\c\in X$.
 Then in $Q_\lambda(P)$ we have
$$
\mu([X,Y])= \begin{cases}
(-1)^{\rho(Y)-\rho(X)} & \mbox{if $Y$ contains $\c \cup \d$, where $\d$ is the unique 
increasing }\\ & \mbox{maximal chain in $\M_{[e(X),e(Y)]_{\c}}$,}\\
0 & \mbox{otherwise.}
\end{cases}
$$
\end{proposition}
\begin{proof}
Since $\lambda$ is a CR-labeling, there exists a unique maximal chain $\d$ in 
$[e(X),e(Y)]_{\c}$ which is increasing.  If $Y\in Q_\lambda(P)$ is the class that contains 
$\c\cup\d$ then, by the definition of $Q_{\lambda}(P)$ and 
Equation~\ree{equation:lambdastar},  
there is an increasing saturated chain which terminates at $Y$ say $X=X_1\cover X_2\cover 
\cdots \cover X_k=Y$.  Now suppose that there was another increasing chain $X=X_1'\cover 
X_2'\cover \cdots \cover X_k'$ in $Q_\lambda(P)$ with $e(X_k')=e(Y)$.  Then by  
Lemma~\ref{lemma:saturatedchainssamelabelsinPandQforClabeling} there would be a corresponding 
maximal chain in 
$[e(X), e(Y)]_{\c}$ which is increasing.  Since 
there is a unique increasing maximal chain in $[e(X),e(Y)]_{\c}$, we know that $e(X_i) =e(X_i')$ 
for all $i$.  But (the CW-labeling version of) Lemma~\ref{lemma:main} part (c) implies that 
$X_i=X_i'$ which is impossible.  Since (the CW-labeling version of)
Proposition \ref{proposition:QERStarLabel} asserts that $\lambda^*$ 
given in Equation (\ref{equation:lambdastar}) is an ER$^*$-labeling of $Q_{\lambda}(P)$ Theorem \ref{theorem:muER} gives the desired result.
\end{proof}
\begin{remark}
 Note that in the EW-version of Proposition \ref{proposition:mobiusQlambda} the chain $\c \in X$ is 
irrelevant.
\end{remark}

\subsection{Another description of $Q_{\lambda}(P)$}

Let  $\Lambda$ be a poset and let $w$ be a word with letters in the alphabet $\Lambda$. Assume 
that whenever $w_i<w_{i+1}$ we 
are allowed to do exchanges on $w$ of the form 
$$w_1w_2\cdots w_{i-1}w_iw_{i+1}w_{i+2}\cdots w_n \stackrel{i}{\rightarrow} w_1w_2\cdots 
w_{i-1}w_{i+1}w_{i}w_{i+2}\cdots w_n.$$
It is not hard to check that this type of exchange produces a locally confluent relation and 
after using Newman's Lemma \ref{lemma:newman} we can conclude that there is a unique ascent-free 
word $w'$ that is related to $w$ in this manner. We define $\sort(w):=w'$ to be this unique 
ascent-free word.  For example, if $\Lambda= \mathbb{Z}$, $w=85324$ then $\sort(w) = 85432$.

\begin{definition}\label{rLambdaDef}
Let $P$ be a poset with a CW-labeling $\lambda$. Let $R_\lambda(P)$ be the poset whose elements are 
pairs $(x,w)$ where $x\in P$ and $w$ is the word of labels of an ascent-free saturated chain 
$\c \in \M_{[\hat{0},x]}$ (note that by the definition of CW-labeling $w$ uniquely determines $\c$);
and such that $(x,w) \cover (y,u)$ whenever $x\cover y$ and $u = \sort(w \lambda(\c,x\cover 
y))$ ($uv$ here means the concatenation of the words $u$ and $v$). 
\end{definition}

\begin{example}
 
If we consider the EW-labeling of $\Pi_3$ given in 
Example~\ref{example:labelingposetofpartitions}, we obtain the poset  
$R_\lambda(\Pi_3)$ depicted in Figure~\ref{rLambdaPi3Fig}. Comparing $R_\lambda(\Pi_3)$  and 
$Q_\lambda(\Pi_3)$ given in Figure~\ref{fig:exampleQPi}, we can observe directly 
that these two posets are isomorphic.
\end{example}

\begin{figure}
\begin{tikzpicture}[scale=.7]

\tikzstyle{every node}=[inner sep=3pt, scale=.8, minimum width=4pt]

\node (a) at (1,-1) { $(1/2/3, \emptyset$)};

\node (b) at (-2,2.5)  {$(12/3, (1,2))$};

\node (c) at (1,2.5)  {$(13/2, (1,3))$};

\node (d) at (4,2.5)  {$(1/23, (2,3))$};

\node (e) at (-.5,6)  {$(123, (1,3)(1,2))$};
\node (f) at (4,6)  {$(123, (2,3)(1,2))$};

 \draw (c)-- (a) -- (b);
\draw (a)--(d)--(f);
\draw (a)--(d)--(f);
\draw (c)--(e)--(b);

\end{tikzpicture}
\caption{$R_\lambda(\Pi_3)$} \label{rLambdaPi3Fig}
\end{figure}
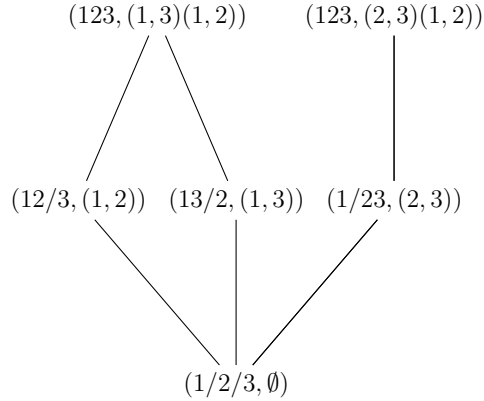

\begin{theorem}\label{theorem:secondcaracterizationQlambda}
If $\lambda$ is a CW-labeling of $P$, then $R_\lambda(P)\cong Q_\lambda(P)$.
\end{theorem}
\begin{proof}
Let $\varphi: R_\lambda(P) \rightarrow Q_\lambda(P)$ be given by $\varphi((x,w)) = [\c]$ 
where $\c \in \M_{[\hat{0},x]}$ is the unique ascent-free chain determined by $w$. Using the fact 
that when $\lambda$ is a CW-labeling each equivalence class in $Q_\lambda(P)$ contains a unique 
ascent-free saturated chain, one can see that $\varphi$ is a well-defined bijection. We want to see 
that $\varphi$ and $\varphi^{-1}$ are poset maps. 

We now show that $\varphi$ is an order-preserving map. Since we are working with finite posets it is enough to 
show that   $\varphi$ preserves cover relations. Suppose that $(x,w) \cover (y,u)$ in 
$R_\lambda(P)$. Let $\c$ and $\d$ be  the unique saturated chains from $\hat{0}$ 
determined by $w$ and $u$, respectively. Then  $\varphi((x,w)) = [\c]$ and $\varphi((y,u)) = [\d]$.
Let $\d'=\c \cup \{y\}$, Since $\d'$ is a saturated chain from $\hat{0}$ to $y$ with word of 
labels $w \lambda(\c\cup \{y\},x\cover y)$, we can use the rank two switching property to see that $\d'$ 
is 
equivalent to an ascent-free chain with labels $\sort(w \lambda(\c, x\cover y))=u$. Since $u$ 
determines a unique saturated chain we have that  $\d' \sim \d$.  Moreover, $\c \cover \d'$ in 
$C(P)$, so we have $\varphi((x,w))= [\c] \cover [\d] =\varphi((y,u))$.  

Now we show $\varphi^{-1}$ is also an order-preserving map. To see why, suppose that $X \cover 
Y$ in $Q_{\lambda}(P)$; and let $x=e(X)$, $y= e(Y)$, $\varphi^{-1}(X)=(x,w)$ and 
$\varphi^{-1}(Y)=(y,u)$.  By the definition of 
$Q_\lambda(P)$ and Proposition \ref{proposition:quotientposet}, there are chains $\c \in X$ and 
$\d'\in Y$ such that $\d'=\c \cup \{y\}$ and hence $x\cover y$. We may assume, without loss of generality, 
that $\c$ is the ascent-free chain in $X$ with word $w$ (otherwise, apply quadratic exchanges until 
you obtain an ascent-free chain in $X$).
Note that $w\lambda(\d',x\cover y)$ is the word of 
labels of $\d'$. Let $\d$ be the ascent-free chain in $Y$ with word of labels $u$. Since $\d \sim 
\d'$ we also have that the word of labels of $\d$ is $\sort(w\lambda(\c,x\cover y))$ and so 
$u=\sort(w\lambda(\c,x\cover y))$. We obtain then that $\varphi^{-1}(X)=(x,w)\cover (y,u)=
\varphi^{-1}(Y)$ as desired.
\end{proof}

The new characterization of $Q_\lambda(P)$ that was given  in Theorem 
\ref{theorem:secondcaracterizationQlambda} can be helpful providing combinatorial descriptions of 
Whitney duals (see Section \ref{section:noncrossingpartitiondual}).  

\section{Examples of posets with Whitney labelings}\label{section:examplesWposets}
In this section we give several examples of posets with Whitney labelings.  By  
Theorem~\ref{theorem:CW}, 
this implies that these posets also have Whitney duals.
\subsection{Geometric lattices}

In \cite{Stanley1974} Stanley introduced an edge labeling for geometric lattices that is an 
ER-labeling (In fact, as shown by Bj\"orner in \cite{Bjorner1982}  it is an EL-labeling).   We give 
the 
definition 
below and then show it is also an EW-labeling.

\begin{definition}\label{geomLabDef}
Let $L$ be a geometric lattice with set of atoms $A(L)$.  Fix a total order on $A(L)$.  Now define 
$\lambda:\E(L) \rightarrow A(L)$ by  setting $\lambda(x\cover y)=a$ where $a$ is the smallest atom 
such that $x\vee a = y$.  We call $\lambda$ a \emph{minimum labeling of $L$}. Note that this 
labeling can be different for different total orders on $A(L)$
\end{definition}

\begin{example}\label{example:labelingPIn}
The labeling $\lambda$ of $\Pi_n$ in Example~\ref{example:labelingposetofpartitions} is  a minimum 
labeling. 
Here we associate each atom of $\Pi_n$ with the ordered pair $(i,j)$ where $ij$ is the unique 
nontrivial block in the atom and such that $i<j$.  We then order the atoms lexicographically. 
\end{example}
  
\begin{proposition}\label{rk2Geom}
For any geometric lattice $L$ a minimum labeling of $L$ is an 
EW-labeling. 
\end{proposition}  
\begin{proof} 
It was shown in~\cite{Stanley1974} that a minimum labeling is an ER-labeling.  Also, for any 
interval $[x,y]$ 
the labels along any maximal chain uniquely determine the chain since one can read off the elements 
of the chain by taking joins of $x$ with the labels along the chain.  Thus it suffices to 
show that a minimum labeling has the rank two switching property.

Let 
$\lambda$ be a minimum labeling of $L$, let $[x,y]$ be an interval of rank two and suppose that 
$ij$ is the word of labels of the increasing chain,  $x\cover x\vee i \cover x\vee i\vee j =y$. 
Since $L$ is geometric and $j$ is an atom not underneath $x$, $x\cover x\vee j\cover y$.  Observe 
that 
 $\lambda(x\cover x\vee j) =j$, since if this was  not the case, this would imply $\lambda(x\vee i 
\cover y) <j$ which is a contradiction.
  Moreover,  $i$ is not below $x\vee j$  and $i$  is below $y$.   Since there is a unique 
increasing 
chain in $[x,y]$,  $i$ is the smallest atom that appears as a label in $[x,y]$.  It follows that 
$\lambda(x\vee j \cover y) = i$. We conclude that the chain $x\cover x\vee j \cover y$ has the word 
of labels $ji$. Since joins are 
unique, there is only one chain in $[x,y]$ with word of labels $ji$ and thus any minimum labeling 
satisfies the rank two switching property.
\end{proof}

We have the following theorem as a corollary.

\begin{theorem}[\cite{DleonHallam2017}]
Every geometric lattice has a Whitney dual.
\end{theorem}

\begin{remark}For the poset $\Pi_n$ the authors proved in \cite{DleonHallam2017} that $\ISF_n$ is the Whitney dual corresponding to the minimal labeling of Example \ref{example:labelingPIn}.
\end{remark}

\subsection{The noncrossing partition lattice}\label{section:noncrossingpartitiondual}

\begin{figure}
\begin{center}
\begin{tikzpicture}[scale=0.9]
\tikzstyle{every node}=[inner sep=3pt, scale=1, minimum width=4pt]
\node (0) at (0,0) {$1/2/3/4$};

\node (a) at  (-5, 2)  {$12/3/4$};
\node (b)  at  (-3, 2)  {$13/2/4$};
\node (c) at  (-1, 2)  {$14/2/3$};

\node (d) at  (5, 2)  {$1/2/34$};
\node (e) at  (3, 2)  {$1/24/3$};
\node (f) at  (1, 2)  {$1/23/4$};

\node (g) at  (-5, 4)  {$123/4$};
\node (h) at  (-3, 4)  {$124/3$};
\node (i) at  (-1, 4)  {$12/34$};

\node (j) at  (5, 4)  {$1/234$};
\node (k) at  (3, 4)  {$14/23$};
\node (l) at  (1, 4)  {$134/2$};

\node (1) at (0,6) {$1234$};

\draw[red] (0)--(a);
\draw[red] (0)--(b);
\draw[red] (0)--(c);
\draw[red] (b)--(g);
\draw[red] (c)--(h);
\draw[red] (c)--(l);
\draw[red] (f)--(g);
\draw[red] (f) -- (k);
\draw[red] (e)--(h);
\draw[red] (d)--(i);
\draw[red] (d)--(l);
\draw[red] (l)--(1);
\draw[red] (k)--(1);
\draw[red] (j)--(1);

\draw [thick, dashed,blue] (0)--(e);
\draw [thick, dashed,blue] (0)--(f);
\draw [thick, dashed,blue] (a)--(g);
\draw [thick, dashed,blue] (a)--(h);
\draw[thick, dashed,blue] (c)--(k);
\draw[thick, dashed,blue] (e)--(j);
\draw[thick, dashed,blue] (d)--(j);
\draw[thick, dashed,blue] (h)--(1);
\draw[thick, dashed,blue] (i)--(1);

\draw[thick, dotted] (0)--(d);
\draw[thick, dotted] (a)--(i);
\draw[thick, dotted] (b)--(l);
\draw[thick, dotted] (g)--(1);
\draw[thick, dotted] (f)--(j);

\end{tikzpicture}
\caption{Edge labeling of $\NC_4$.  For clarity, the edge labels are represented by  line patterns.  
The (red) solid lines represent the label $1$, the (blue) dashed lines represent $2$, and the 
(black) dotted lines represent $3$.}\label{fig:NC4}
\end{center}
\end{figure}
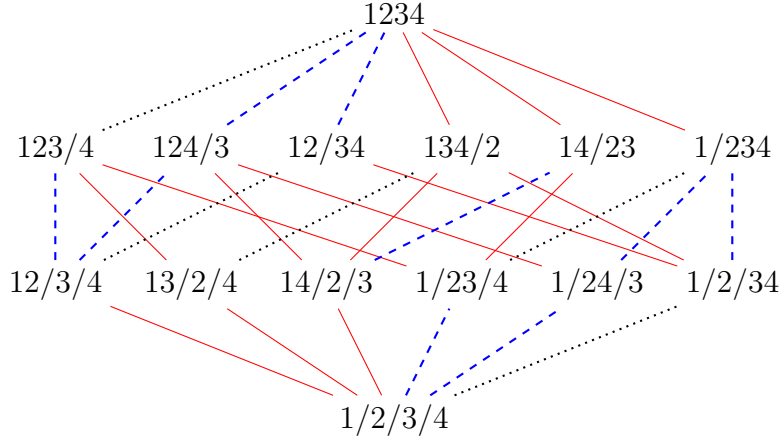
We say a partition $\pi=B_1 / B_2/\cdots /B_k$ of 
$[n]$ is \emph{noncrossing} if there are no $a<b<c<d$ such that $a,c\in B_i$ and $b,d\in B_j$ for 
some $i\neq j$.
For 
example, $124/35/67$ is not a noncrossing partition since $2<3<4<5$ and $\{2,4\}$ and $\{3,5\}$ are 
in two different blocks, but $127/36/45$ is noncrossing.  The 
\emph{noncrossing partition lattice}, denoted $\NC_n$, is the set of noncrossing partitions 
of $[n]$ ordered by refinement. As the name suggest, $\NC_n$ is a lattice and has many nice 
combinatorial properties 
(see Simions survey article~\cite{Simion1997} for more information).  $\NC_n$ is an induced subposet 
of $\Pi_n$, but it is not a sublattice of $\Pi_n$.  Figure~\ref{fig:NC4} depicts $\NC_4$.   

In~\cite{Stanley1997}, Stanley found a beautiful connection between $\NC_n$ and a set of 
combinatorial objects known as parking functions.  A \emph{parking function} of $n$ is a sequence 
of $n$ positive integers $(p_1,p_2,\dots,p_n)$ with the property that when it is rearranged 
in a weakly increasing order $p_{i_1}<p_{i_2}<\cdots < p_{i_n}$, then $p_{i_j}\le j$ for all $j$.
An edge labeling of $\NC_n$ is given in \cite{Stanley1997} with 
the property that the words of labels along the maximal chains are 
exactly the parking functions of $n-1$.   We will show that this edge labeling is in fact an 
EW-labeling, establishing that $\NC_n$ has a Whitney dual.  To describe the labeling of Stanley, 
first note that just as in the partition lattice $\Pi_n$, the cover relation is given by merging 
two blocks 
together.   Suppose that $\sigma$ is obtained from $\pi$ by merging $B_i$ and $B_j$, where $\min 
B_i < \min B_j$, then define 
\begin{align}\label{equation:parkingfunctionlabeling}
\lambda_{\NC}(\pi \cover \sigma)  = \max \{ a \in B_i \mid a<\min B_j\} .
\end{align}
See Figure~\ref{fig:NC4} to see this labeling for $\NC_4$.

The work in~\cite{Stanley1997} uses a slightly different definition of ER-labeling. 
There, an ER-labeling is defined as a labeling such that each interval has a unique 
weakly increasing maximal chain. It is not hard to see that the labeling in 
\eqref{equation:parkingfunctionlabeling} does not fit this definition.  However, we have chosen to 
define an ER-labeling as a labeling where each maximal interval has a unique strictly increasing 
maximal chain.  Under this definition, one can check that $\lambda_{\NC}$ is indeed an ER-labeling. 

In \cite{Stanley1997} this ER-labeling is used to prove that there is a local
$\sym_n$-action on the maximal chains of $\NC_{n+1}$.    This action is local in the sense that if a 
transposition of the form $(i, i+1)$ acts on a maximal chain it only changes the chain in at 
most the element at rank $i$.  
Suppose that $[x,y]$ is an interval of rank two in $\NC_n$ and such that $\rho(x) =i-1$ and $\rho(y) 
=i+1$.  Then for a maximal chain $\c$ in  $[x,y]$,
$$
(i,i+1) \c=
\begin{cases}
\c' & \mbox{ if $\c$ has a strict ascent or strict descent in $[x,y]$,}\\
\c & \mbox{otherwise,}
\end{cases}
$$
where $\c'$ is the unique maximal chain $[x,y]$ with the same label set as $\c$ which reverses the 
labels in $\c$.  In other words, the action switches strict ascents and strict descents 
and leaves equal labels fixed.   Note that this local action of $\sym_{n-1}$ coincides with the 
action on the set of parking functions where the transposition $(i,i+1)$ permutes the letters $i$ 
and $i+1$ of a parking function. Under this action, there is exactly one weakly decreasing 
parking function in each orbit.  The fact that this action exists immediately 
implies that 
Stanley's labeling of $\NC_n$ has the rank two switching property.  This, together with the 
fact that the maximal 
chains are in one-to-one correspondence with parking functions (which are all different) implies 
that the labeling $\lambda_{\NC}$ in \eqref{equation:parkingfunctionlabeling} satisfies the 
conditions of Definition 
\ref{definition:EW}, and so is an EW-labeling. Hence by Theorem \ref{theorem:WimpliesWhitney} we 
conclude that $\NC_n$ has a Whitney dual.
\begin{theorem}
The labeling $\lambda_{\NC}$ is an EW-labeling of $\NC_n$. Hence $Q_{\lambda}(\NC_n)$ is a Whitney dual of $\NC_n$. 
\end{theorem}

We will now use Theorem \ref{theorem:secondcaracterizationQlambda} to provide a more familiar combinatorial description of the Whitney dual $Q_{\lambda_{\NC}}(\NC_n)$ of $\NC_n$.

Recall that a 
Dyck path of order $n$ is a lattice path from $(0,0)$ to $(n,n)$ that never goes below the line $y=x$ and only takes steps in the directions of the vectors $(1,0)$ (East) and $(0,1)$ (North).   
We will consider Dyck paths $D$ that come with a special labeling. Given an increasing sequence 
$b_1<b_2<\cdots< b_{n+1}$ of positive integers, we label the point $(i-1, 0)$ of $D$ by $b_i$ (see Figure \ref{fig:labeledboard}).  In
Figure~\ref{fig:labeledDyck} we illustrate two labeled Dyck paths.

\begin{figure}

\begin{subfigure}[t]{0.5\textwidth}
\centering
\begin{tikzpicture}[scale=0.9]

\draw (0,0)-- (1,0) (2,0)--(2.5,0) (0,.5)-- (1,.5) (2,.5)--(2.5,.5)(0,2.5)-- (1,2.5) (2,2.5)--(2.5,2.5);
\draw (0,0)-- (0,2.5);
\draw (0.5,0)-- (0.5,2.5);

\draw[dotted] (.5,0) -- (2, 0);
\draw[dotted] (.5,.5) -- (2, .5);
\draw[dotted] (.5,2.5) -- (2, 2.5);

\draw (2.5,0)-- (2.5,2.5);
\draw (2.,0)-- (2.,2.5);
\draw (1.,0)-- (1.,2.5);

\tikzstyle{every node}=[scale=0.8]
\node at (.1,-.25) {$b_1$};
\node at (.6,-.25) {$b_2$};
\node at (1.1,-.25) {$b_3$};

\node at (2.6,-.25) {$b_{n+1}$};
\node at (2.1,-.25) {$b_{n}$};

\end{tikzpicture}
\caption{Labeling of the bottom row of the grid}\label{fig:labeledboard}
\end{subfigure}~
\begin{subfigure}[t]{0.5\textwidth}
\centering
\begin{tikzpicture}
\draw (0,0)-- (1.5,0);
\draw (0,.5)-- (1.5, .5);
\draw (0,1)-- (1.5, 1);
\draw (0,1.5)-- (1.5, 1.5);

\draw (0,0) -- (0,1.5);
\draw (.5,0) -- (.5,1.5);
\draw (1,0) -- (1,1.5);
\draw (1.5,0) -- (1.5,1.5);

\node at (0,-.25) {$1$};
\node at (.5,-.25) {$5$};
\node at (1,-.25) {$6$};
\node at (1.5,-.25) {$7$};

\draw [ultra thick,red] (0,0) --(0,1);
\draw [ultra thick,red] (0,1)--(.5,1);
\draw [ultra thick,red] (.5,1)--(1,1);
\draw [ultra thick,red] (1,1)--(1,1.5);
\draw [ultra thick,red] (1, 1.5)--(1.5,1.5);

\begin{scope} [shift ={(3,0)}]

\draw (0,0)-- (1,0);
\draw (0,.5)-- (1, .5);
\draw (0,1)-- (1, 1);

\draw (0,0) -- (0,1);
\draw (.5,0) -- (.5,1);
\draw (1,0) -- (1,1);

\node at (0,-.25) {$2$};
\node at (.5,-.25) {$3$};
\node at (1,-.25) {$4$};

\draw [ultra thick,blue] (0,0) --(0,.5);
\draw [ultra thick,blue] (0,.5)--(.5,.5);
\draw [ultra thick,blue] (.5,1)--(.5,.5);
\draw [ultra thick,blue] (.5,1)--(1,1);
\end{scope}
\end{tikzpicture}
\caption{Two labeled Dyck paths}\label{fig:labeledDyck}
\end{subfigure}~
\caption{}\label{DyckPathFig}
\end{figure}

We now define a process of ``merging" two labeled Dyck paths $D_1$ and $D_2$ to obtain a new 
labeled Dyck path $D$. Suppose that $D_1$ and $D_2$ have disjoint and 
noncrossing label sets $B=\{b_1,b_2,\dots, b_j\}$ and $C=\{c_1,c_2,\dots, c_k\}$, where both sets 
are written in increasing order and $b_1<c_1$.    Since the sets are noncrossing then there exists 
an $i$ such that $b_i<c_1<c_2<\cdots < c_k<b_{i+1}$ (where we use the convention $b_{j+1}=\infty$). 
Then, the new lattice path $D$, will be a path from $(0,0)$ to $(j+k, j+k)$ whose labels 
along the bottom row are $b_1,b_2, \dots b_i, c_1,c_2,\dots, c_k,b_{i+1},\dots,b_j$.  From left to 
right  until we reach the vertical line labeled $b_{i}$, $D$ looks exactly the same as $D_1$.  
In the line labeled $b_{i}$ in $D$ we add all the north steps that $D_1$ had originally at 
$b_i$ plus one additional north step followed by an additional east step from the line labeled 
$b_i$ 
to the line labeled $c_1$. Then we glue $D_2$ where we left off in the line labeled $c_1$. After we 
finish gluing $D_2$, we glue the remaining part of $D_1$ that goes from the line labeled $b_{i}$ to the line 
labeled $b_j$. As an example,  suppose we wish to obtain a labeled Dyck path $D$ by merging the two labeled Dyck paths $D_1$ and $D_2$ in Figure~\ref{fig:labeledDyck} on label sets $\{1,3,6,7\}$ and $\{2,3,4\}$ respectively. We start by 
creating a grid from $(0,0)$ to $(6,6)$ and label the bottom row with the (ordered) union of 
the two labeled sets (see Figure~\ref{Dyck1}).  Since $1$ is the largest element in $\{1,3,5,6,7\}$ 
smaller than all the elements of $\{2,3,4\}$, we add in $D$ a new north step at the line labeled 
$1$ and add a new east step afterwards between lines labeled $1$ and $2$.  Since at the line 
labeled $1$, $D_1$ had two north steps, $D$ will have now $3$ north steps (see 
Figure~\ref{Dyck2}). Next, glue $D_2$ where we left off (see 
Figure~\ref{Dyck3}) and then the remaining part of $D_1$ to 
obtain $D$ (see Figure~\ref{Dyck4}).

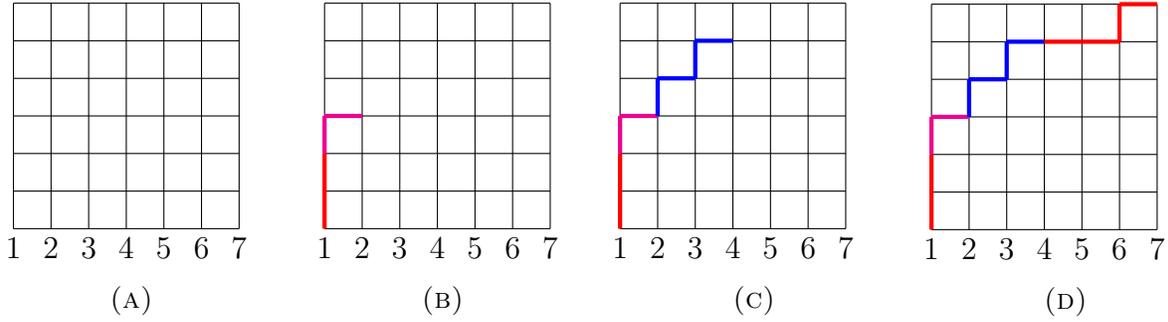
\begin{figure}
\begin{subfigure}{0.25\textwidth}
\centering

\begin{tikzpicture}
\draw (0,0)-- (3,0);
\draw (0,.5)-- (3, .5);
\draw (0,1)-- (3, 1);
\draw (0,1.5)-- (3, 1.5);
\draw (0,2)-- (3, 2);
\draw (0,2.5)-- (3, 2.5);

\draw (0,3)-- (3, 3);

\draw (0,0) -- (0,3);
\draw (.5,0) -- (.5,3);
\draw (1,0) -- (1,3);
\draw (1.5,0) -- (1.5,3);
\draw (2,0) -- (2,3);
\draw (2.5,0) -- (2.5,3);
\draw (3,0) -- (3,3);

\node at (0,-.25) {$1$};
\node at (.5,-.25) {$2$};
\node at (1,-.25) {$3$};
\node at (1.5,-.25) {$4$};
\node at (2,-.25) {$5$};
\node at (2.5,-.25) {$6$};
\node at (3,-.25) {$7$};
\end{tikzpicture}
\caption{}\label{Dyck1}

\end{subfigure}~
\begin{subfigure}{0.25\textwidth}
\centering
\begin{tikzpicture}
\draw (0,0)-- (3,0);
\draw (0,.5)-- (3, .5);
\draw (0,1)-- (3, 1);
\draw (0,1.5)-- (3, 1.5);
\draw (0,2)-- (3, 2);
\draw (0,2.5)-- (3, 2.5);

\draw (0,3)-- (3, 3);

\draw (0,0) -- (0,3);
\draw (.5,0) -- (.5,3);
\draw (1,0) -- (1,3);
\draw (1.5,0) -- (1.5,3);
\draw (2,0) -- (2,3);
\draw (2.5,0) -- (2.5,3);
\draw (3,0) -- (3,3);

\node at (0,-.25) {$1$};
\node at (.5,-.25) {$2$};
\node at (1,-.25) {$3$};
\node at (1.5,-.25) {$4$};
\node at (2,-.25) {$5$};
\node at (2.5,-.25) {$6$};
\node at (3,-.25) {$7$};

\draw [ultra thick,red] (0,0) --(0,1);
\draw [ultra thick,magenta] (0,1) --(0,1.5);
\draw [ultra thick,magenta] (0,1.5) --(.5, 1.5);

\end{tikzpicture}
\caption{}\label{Dyck2}
\end{subfigure}~
\begin{subfigure}{0.25\textwidth}

\begin{tikzpicture}
\draw (0,0)-- (3,0);
\draw (0,.5)-- (3, .5);
\draw (0,1)-- (3, 1);
\draw (0,1.5)-- (3, 1.5);
\draw (0,2)-- (3, 2);
\draw (0,2.5)-- (3, 2.5);

\draw (0,3)-- (3, 3);

\draw (0,0) -- (0,3);
\draw (.5,0) -- (.5,3);
\draw (1,0) -- (1,3);
\draw (1.5,0) -- (1.5,3);
\draw (2,0) -- (2,3);
\draw (2.5,0) -- (2.5,3);
\draw (3,0) -- (3,3);

\node at (0,-.25) {$1$};
\node at (.5,-.25) {$2$};
\node at (1,-.25) {$3$};
\node at (1.5,-.25) {$4$};
\node at (2,-.25) {$5$};
\node at (2.5,-.25) {$6$};
\node at (3,-.25) {$7$};

\draw [ultra thick,red] (0,0) --(0,1);
\draw [ultra thick,magenta] (0,1) --(0,1.5);
\draw [ultra thick,magenta] (0,1.5) --(.5, 1.5);
\draw [ultra thick,blue] (.5, 1.5)--(.5,2);
\draw [ultra thick,blue] (.5,2)--(1,2);
\draw [ultra thick,blue] (1,2)--(1,2.5);
\draw [ultra thick,blue] (1,2.5)--(1.5,2.5);

\end{tikzpicture}
\caption{}\label{Dyck3}
\end{subfigure}~
\begin{subfigure}{0.25\textwidth}
\begin{tikzpicture}
\draw (0,0)-- (3,0);
\draw (0,.5)-- (3, .5);
\draw (0,1)-- (3, 1);
\draw (0,1.5)-- (3, 1.5);
\draw (0,2)-- (3, 2);
\draw (0,2.5)-- (3, 2.5);

\draw (0,3)-- (3, 3);

\draw (0,0) -- (0,3);
\draw (.5,0) -- (.5,3);
\draw (1,0) -- (1,3);
\draw (1.5,0) -- (1.5,3);
\draw (2,0) -- (2,3);
\draw (2.5,0) -- (2.5,3);
\draw (3,0) -- (3,3);

\node at (0,-.25) {$1$};
\node at (.5,-.25) {$2$};
\node at (1,-.25) {$3$};
\node at (1.5,-.25) {$4$};
\node at (2,-.25) {$5$};
\node at (2.5,-.25) {$6$};
\node at (3,-.25) {$7$};

\draw [ultra thick,red] (0,0) --(0,1);
\draw [ultra thick,magenta] (0,1) --(0,1.5);
\draw [ultra thick,magenta] (0,1.5) --(.5, 1.5);
\draw [ultra thick,blue] (.5, 1.5)--(.5,2);
\draw [ultra thick,blue] (.5,2)--(1,2);
\draw [ultra thick,blue] (1,2)--(1,2.5);
\draw [ultra thick,blue] (1,2.5)--(1.5,2.5);
\draw [ultra thick, red] (1.5,2.5)--(2,2.5);
\draw [ultra thick,red] (2,2.5) --(2.5,2.5);
\draw [ultra thick,red] (2.5,2.5) --(2.5,3);
\draw [ultra thick,red] (2.5,3) --(3,3);
\end{tikzpicture}
\caption{}\label{Dyck4}\label{DyckPathFig2}
\end{subfigure}

\caption{The steps  required in merging the Dyck paths in Figure~\ref{DyckPathFig}.  New north and 
east steps are in magenta.}
\end{figure}

In order to verify that the resulting labeled lattice path is also a 
labeled Dyck path (that is, it has the same number of north and east steps and is always above the 
diagonal), we rely on an equivalent definition of a Dyck path. A \emph{ballot sequence} of 
length $2n$ is a $\{0,1\}$-string $s_1s_2\cdots s_{2n}$ with the same number of 
$1$'s and $0$'s and such that for every $i\in [2n]$ the subword $s_1s_2\cdots s_i$ has at least as 
many $1$'s as $0$'s. It is well-known that a lattice path that takes only north and east 
steps is a Dyck path if and only if the sequence obtained associating to each north step a $1$ 
and to each east step a $0$ is a ballot sequence. Relying on this equivalent definition, we see 
that in the resulting path $D$ the number of north steps and east steps is equal and the construction never breaks the property that every preamble in $D$ 
contains at least as many north steps as east steps. Hence $D$ is a well-defined labeled Dyck path.

Let $\NCDyck_n$ be the set whose objects are collections of labeled Dyck paths such that their underlying sets of labels form a noncrossing partition of $[n]$. We provide $\NCDyck_n$ with a partial order by defining for $F,F' \in \NCDyck_n$ the cover relation $F\cover F'$  whenever $F'$ can be 
obtained from $F$ by merging exactly two of the labeled Dyck paths in $F$. Note here that each labeled Dyck path can be represented by its set of labels together with an exponent for each label. The exponent of an element $i$ being the number of north steps in the vertical line labeled $i$ in its Dyck path. This notation extends to the elements in $\NCDyck_n$. For example, we can denote the collection of Dyck paths in Figure~\ref{fig:labeledDyck} by $1^25^06^17^0/2^13^14^0$. In Figure \ref{fig:NCDyck} we illustrate $\NCDyck_4$.

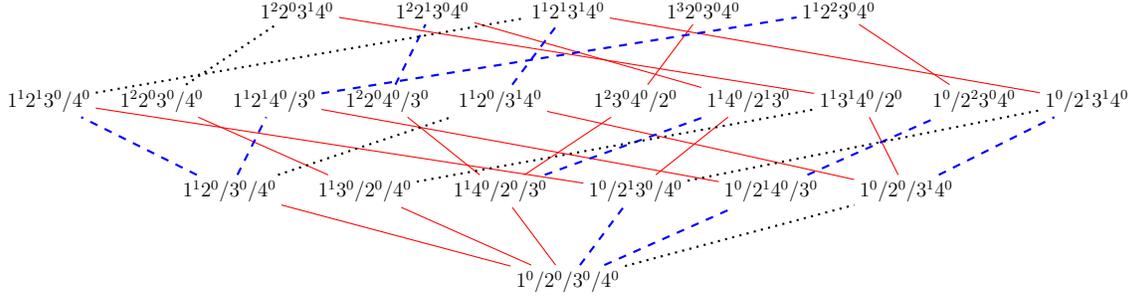
\begin{figure}
\begin{center}
    \begin{tikzpicture}[scale=0.6]
\tikzstyle{every node}=[inner sep=3pt, scale=0.65, minimum width=4pt]
\node (0) at (0,0) {$1^0/2^0/3^0/4^0$};

\node (a) at  (-7.5, 2)  {$1^12^0/3^0/4^0$};
\node (b)  at  (-4.5, 2)  {$1^13^0/2^0/4^0$};
\node (c) at  (-1.5, 2)  {$1^14^0/2^0/3^0$};

\node (d) at  (7.5, 2)  {$1^0/2^0/3^14^0$};
\node (e) at  (4.5, 2)  {$1^0/2^14^0/3^0$};
\node (f) at  (1.5, 2)  {$1^0/2^13^0/4^0$};

\node (g) at  (-11.5, 4)  {$1^12^13^0/4^0$};
\node (g1) at  (-9, 4)  {$1^22^03^0/4^0$};
\node (h) at  (-6.5, 4)  {$1^12^14^0/3^0$};
\node (h1) at  (-4, 4)  {$1^22^04^0/3^0$};

\node (i) at  (-1.5, 4)  {$1^12^0/3^14^0$};

\node (l) at  (6.5, 4)  {$1^13^14^0/2^0$};
\node (l1) at  (1.5, 4)  {$1^23^04^0/2^0$};

\node (k) at  (4, 4)  {$1^14^0/2^13^0$};

\node (j) at  (11.5, 4)  {$1^0/2^13^14^0$};
\node (j1) at  (9, 4)  {$1^0/2^23^04^0$};

\node (t1) at (0,6) {$1^12^13^14^0$};
\node (t2) at (-3,6) {$1^22^13^04^0$};
\node (t3) at (6,6) {$1^12^23^04^0$};
\node (t4) at (-6,6) {$1^22^03^14^0$};
\node (t5) at (3,6) {$1^32^03^04^0$};

\draw[red] (0)--(a);
\draw[red] (0)--(b);
\draw[red] (0)--(c);
\draw[red] (b)--(g1);
\draw[red] (c)--(h1);
\draw[red] (c)--(l1);
\draw[red] (f)--(g);
\draw[red] (f) -- (k);
\draw[red] (e)--(h);
\draw[red] (d)--(i);
\draw[red] (d)--(l);
\draw[red] (l)--(t4);
\draw[red] (l1)--(t5);
\draw[red] (k)--(t2);

\draw[red] (j)--(t1);
\draw[red] (j1)--(t3);

\draw [thick, dashed,blue] (0)--(e);
\draw [thick, dashed,blue] (0)--(f);
\draw [thick, dashed,blue] (a)--(g);
\draw [thick, dashed,blue] (a)--(h);
\draw[thick, dashed,blue] (c)--(k);
\draw[thick, dashed,blue] (e)--(j1);
\draw[thick, dashed,blue] (d)--(j);
\draw[thick, dashed,blue] (h)--(t3);
\draw[thick, dashed,blue] (h1)--(t2);

\draw[thick, dashed,blue] (i)--(t1);

\draw[thick, dotted] (0)--(d);
\draw[thick, dotted] (a)--(i);
\draw[thick, dotted] (b)--(l);
\draw[thick, dotted] (g)--(t1);
\draw[thick, dotted] (g1)--(t4);

\draw[thick, dotted] (f)--(j);

\end{tikzpicture}
 \caption{$\NCDyck_4$ (the colors follow the same standard as in Figure \ref{fig:NC4})}
    \label{fig:NCDyck}
\end{center}
\end{figure}

\begin{theorem}
For all $n\ge 1$,
$Q_{\lambda_{\NC}}(\NC_n)\cong \NCDyck_n$.
\end{theorem}
\begin{proof}
Theorem \ref{theorem:secondcaracterizationQlambda} characterizes 
the poset $Q_{\lambda_{\NC}}(\NC_n)$ as being isomorphic to the poset $R_{\lambda_{\NC}}(\NC_n)$ whose 
elements are pairs $(\pi,w)$ where $\pi \in \NC_n$ and $w$ is the word of labels of an ascent-free 
chain in $[\hat{0},\pi]$. We show that $R_{\lambda_{\NC}}(\NC_n)\cong \NCDyck_n$.

Since maximal chains in $\NC_n$ are labeled with parking functions, when $\pi=\{[n]\}$ is the partition with a single block, we have that $w$ is a weakly decreasing parking function of length $n-1$, which are known to be in bijection with Dyck paths.  For example if $n=4$ the weakly decreasing parking functions are $(1,1,1)$, 
$(2,1,1)$, $(2,2,1)$, $(3,1,1)$ and $(3,2,1)$. The bijection assigns to a parking function with $k_i$ occurrences of the label $i$ the Dyck path with $k_i$ north steps on the line $x=i-1$. In our notation, the pairs $([4],w) \in R_{\lambda_{\NC}}(\NC_4)$ can be 
represented as $1^32^03^04^0$, $1^22^13^04^0$, $1^12^23^04^0$, $1^22^03^14^0$ and $1^12^13^14^0$. 
Now, it is not hard to see that any interval of the form $[\hat{0},B_1/B_2/\cdots/B_k]$ is 
isomorphic to the product of smaller noncrossing 
partition lattices $\NC_{B_1}\times \NC_{B_2}\times \cdots \times \NC_{B_k}$, where $\NC_{B_j}$ is the 
lattice of noncrossing partitions of $B_j\subset [n]$.   
 Moreover the labels in any cover relation in 
$[\hat{0},B_1/B_2/\cdots/B_k]$ depend only on the two blocks being merged. So any 
ascent-free maximal chain can be represented as a noncrossing partition where each of the blocks 
$B_j$ have been decorated with exponents representing an ascent-free maximal chain in $\NC_{B_j}$. 
Note that words of labels on maximal chains of $\NC_{B_j}$ are  ``parking functions'' on $B_j$, 
that is, if $B_j=\{b_1<b_2<\cdots <b_l\}$ then in the word of labels of a maximal chain the 
number of occurrance of the letter $b_i$ is greater or equal to $i$ (an equivalent definition of 
a parking function). For example, the chain in $[\hat{0},1457/23/6/89]$ with word of labels 
$(8,4,4,2,1)$ is represented by $1^14^25^07^0/2^13^0/6^0/8^19^0$. Since $\lambda_{\NC}$ is an 
EW-labeling, the cover relation $(\pi,w)\lessdot (\pi',w')$ in $R_{\lambda_{\NC}}(\NC_n)$ is 
completely determined by the cover relation $\pi \lessdot \pi'$. Hence $(\pi',w')$ is obtained from 
$(\pi,w)$ by merging two blocks $B_i$ and $B_j$ of $\pi$ such that $\min B_i < \min B_j$ and 
$w'=\sort(wp)$ where $p=\lambda_{\NC}(\pi \lessdot \pi')=\max \{ a \in B_i \mid a<\min B_j\}$. The reader can note that this is equivalent to the definition of merging labeled Dyck paths. Indeed, in our notation this amounts to merging the weighted blocks on the sets $B_i$ and $B_j$ and increasing the exponent of $p\in B_i$ by one. For example,  if in $1^14^25^07^0/2^13^0/6^0/8^19^0$ we 
merge the blocks with sets $\{1,4,5,7\}$ and $\{2,3\}$ we get $1^22^13^04^25^07^0/6^0/8^19^0$. If 
we further merge the blocks with sets  $\{1,2,3,4,5,7\}$ and $\{6\}$ we get 
$1^22^13^04^25^16^07^0/8^19^0$. We then have 
that $R_{\lambda_{\NC}}(\NC_n)$  is indeed isomorphic to the poset $\NCDyck_n$.
\end{proof}

It is interesting to note the  well-known fact that the M\"obius function value of $\NC_n$ is (up 
to a sign) the Catalan number $C_{n-1}$. This information is recovered here since $\NCDyck_n$ is a 
Whitney dual of $\NC_n$ and its maximal elements are Dyck paths of order $n-1$ which  are Catalan objects.

\subsection{The weighted partition poset and the poset of rooted spanning 
forests}\label{weightedPartSec} Here we discuss 
the original example that motivated our work. This example was noticed by 
 Gonz\'alez D'Le\'on and Wachs in \cite{DleonWachs2016} while
studying a poset of partitions where each block has a weight that is a natural number. This 
poset, known as the poset of weighted partitions and denoted $\Pi_n^w$, was originally introduced 
by Dotsenko and Khoroshkin in \cite{DotsenkoKhoroshkin2007} and is 
related to the study of the operad of Lie algebras with two compatible brackets. The authors 
of \cite{DleonWachs2016} realized that the Whitney numbers of the first and 
second kind were switched with respect to those of the poset of rooted spanning forests $\SF_n$ on 
$[n]$ studied by Reiner \cite{Reiner1978} and Sagan \cite{Sagan1983}. Since the two pairs of 
Whitney numbers were already computed, by direct comparison, 
we can conclude that $\Pi_n^w$ and $\SF_n$ are Whitney duals. In this section we use the 
theory developed in Section \ref{section:whitneylabelings} to give a different proof of this fact.

A weighted partition $B_1^{v_1}/B_2^{v_2}/\cdots/B_t^{v_t}$ of $[n]$ is a partition 
$B_1/B_2/\cdots/B_t$ of $[n]$ such that each block $B_i$ is assigned a weight $v_i \in 
\{0,1,2,\dots,|B_i|-1\}$. The {\it poset of weighted partitions} $\Pi_n^{w}$  is the set of weighted 
partitions of $[n]$ with 
order relation given by 
$$
A_1^{w_1}/A_2^{w_2}/\cdots/A_s^{w_t} \le B_1^{v_1}/ B_2^{v_2}/\cdots/B_t^{v_t}
$$
if the following
conditions hold:
\begin{itemize}
 \item $A_1/A_2/\cdots /A_s \le B_1/B_2/ \cdots /B_t $ in $\Pi_n$
 \item if $B_k=A_{i_1}\cup A_{i_2}\cup \cdots \cup A_{i_l} $ then 
 $v_k-(w_{i_1} + w_{i_2} + \cdots + w_{i_l})\in \{0,1,\dots,l-1\}$.
\end{itemize}
Equivalently, we can define the covering relation  by 
$$A_1^{w_1}/A_2^{w_2}/\dots /A_s^{w_s}\lessdot B_1^{v_1}/ B_2^{v_2}/ \dots /B_{s-1}^{v_{s-1}}$$ 
if the following conditions hold:
\begin{itemize}
 \item $A_1/A_2/\cdots /A_s \lessdot B_1/B_2/ \cdots /B_{s-1} $ in $\Pi_n$
 \item if $B_k=A_{i}\cup A_{j}$, where $i \ne j$, then $v_k-(w_{i} + w_{j}) \in \{0,1\}$
 \item if $B_k = A_i$ then $v_k = w_i$.
 \end{itemize}

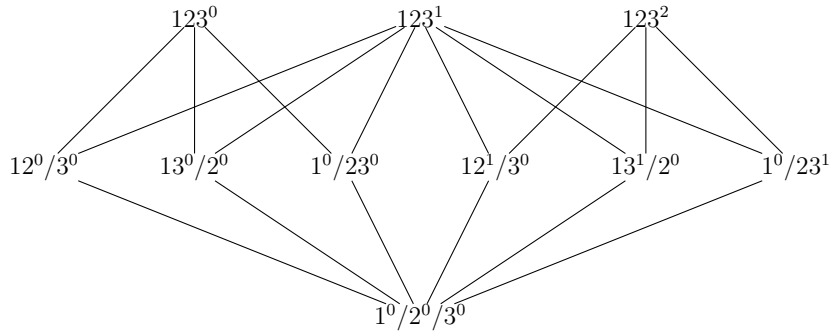
\begin{figure}[h]

\begin{center} 
\begin{tikzpicture}[line join=bevel,scale=1]

\tikzstyle{every node}=[inner sep=0pt, scale=0.8, minimum width=4pt]
\node (n1232) at (3,4) {$123^{2}$};
  \node (n13020) at (-3,2) {$13^{ 0}/ 2^{ 0}$};
  \node (n102030) at (0,0)  {$1^{0}/ 2^{0}/ 3^{0}$};
  \node (n1231) at (0,4) {$123^{1}$};
  \node (n12030) at (-5,2) {$12^{0}/ 3^{0}$};
  \node (n13120) at (3,2)  {$13^{1}/ 2^{0}$};
  \node (n1230) at (-3,4){$123^ {0}$};
  \node (n10230) at (-1,2)  {$1^{0}/ 23^{0}$};
  \node (n12130) at (1,2)  {$12^{ 1}/ 3^{0}$};
  \node (n10231) at (5,2) {$1^{0}/ 23^ {1}$};

  \draw (n1231) -- (n10230) ;
  \draw [] (n13020) -- (n102030);
  \draw [] (n1232) -- (n13120);
  \draw [] (n1231)-- (n13020);
  \draw [] (n10230)--(n102030);
  \draw [] (n1230) -- (n10230);
  \draw [] (n1231) -- (n13120);
  \draw [] (n12030)-- (n102030);
  \draw [] (n1231) --(n12130);
  \draw [] (n1232) -- (n12130);
  \draw [] (n13120) --(n102030);
  \draw [] (n1231) --(n10231);
  \draw [] (n1230) -- (n13020);
  \draw [] (n1230)  -- (n12030);
  \draw [] (n12130) --  (n102030);
  \draw [] (n1232)  --  (n10231);
  \draw [] (n10231)  --  (n102030);
  \draw [] (n1231) -- (n12030);

\end{tikzpicture}
\end{center}
\caption[]{Weighted partition poset for $n=3$}\label{figure:weightedn3k2}
\end{figure}

The poset $\Pi_n^w$ has a minimum element $$\hat 0:= 1^0/2^0 / \dots /n^0$$ and  $n$
maximal elements 
$$[n]^0, \, [n]^1, \dots,  [n]^{n-1}.$$ Note  that  for all   $i=0,1,\dots,n-1$, the maximal  
intervals 
$[\hat 0, [n]^i]$ and $[\hat 0, [n]^{n-1-i}]$ are isomorphic to each other, and the two maximal 
intervals $[\hat 0, [n]^0]$ and $[\hat 0, [n]^{n-1}]$ are isomorphic to $\Pi_n$. See 
Figure~\ref{figure:weightedn3k2} for the example of $\Pi_3^w$.

A \emph{rooted spanning forest} on $[n]$ is a spanning forest of the complete graph on vertex 
set $[n]$  such that in every connected component there is a unique specially marked vertex, called 
the 
\emph{root}.
Let $\SF_n$ be the set of rooted spanning forests on $[n]$. For $F\in\SF_n$ let $E(F)$ 
denote the edge set of $F$ and $R(F)$ be the set of roots in $F$. The set $\SF_n$ has the structure 
of a poset with order relation given by $F_1\le F_2$ whenever 
$$E(F_1)\subseteq E(F_2) \text{ and } R(F_2) \subseteq R(F_1).$$
Equivalently, the cover relation $F_1\lessdot F_2$ occurs if $F_2$ is obtained from $F_1$ 
by adding a new edge $\{x,y\}\in E(F_2)$ such that $x,y \in R(F_1)$ and by choosing either $x$ or 
$y$ as the new root of its component. Note that $R(F_2)$ is either $R(F_1)\setminus \{x\}$ or  
$R(F_1)\setminus \{y\}$.
See Figure \ref{figure:spanning_forests_3} for the 
example of $\SF_3$. For more on this poset see \cite{Reiner1978,Sagan1983}.

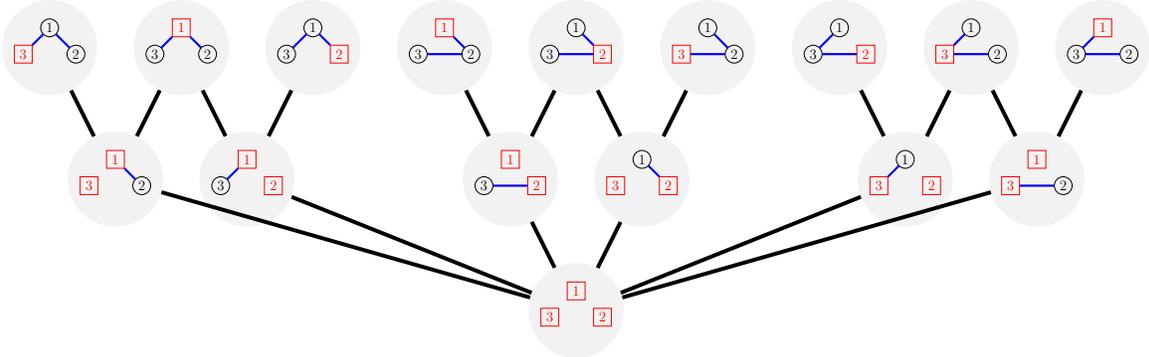
\begin{figure}[h]
\begin{center} 
\begin{tikzpicture}[scale=0.35]


\tikzstyle{every node}=[fill,gray!10,circle, inner sep=0pt,minimum size=80,scale=0.45]
\node (p1) at (13,0.25) {};
\node (q1) at (-4.5,5.25) {};
\node (q2) at (0.5,5.25) {};
\node (q3) at (10.5,5.25) {};
\node (q4) at (15.5,5.25) {};
\node (q5) at (25.5,5.25) {};
\node (q6) at (30.5,5.25) {};

\node (r1) at (-7,10.25) {};
\node (r2) at (-2,10.25) {};
\node (r3) at (3,10.25) {};
\node (r4) at (8,10.25) {};
\node (r5) at (13,10.25) {};
\node (r6) at (18,10.25) {};
\node (r7) at (23,10.25) {};
\node (r8) at (28,10.25) {};
\node (r9) at (33,10.25) {};

\tikzstyle{every node}=[black, draw, circle, inner sep=1, minimum size=15,scale=0.45]
\tikzstyle{every path}=[blue, line width=0.2pt]

\node[rectangle,color=red] (a1) at (13,1){$1$};
\node[rectangle,color=red]  (a2) at (14,0){$2$};
\node[rectangle,color=red]  (a3) at (12,0){$3$};

\node[rectangle,color=red]  (b11) at (0.5,6){$1$};
\node[rectangle,color=red]  (b12) at (1.5,5){$2$};
\node (b13) at (-0.5,5){$3$};
 \draw[thick] (b11) -- (b13) ;

\node[rectangle,color=red]  (c11) at (-4.5,6){$1$};
\node (c12) at (-3.5,5){$2$};
\node[rectangle,color=red]  (c13) at (-5.5,5){$3$};
 \draw[thick] (c11) -- (c12) ;

\node[rectangle,color=red]  (d11) at (10.5,6){$1$};
\node[rectangle,color=red]  (d12) at (11.5,5){$2$};
\node (d13) at (9.5,5){$3$};
 \draw[thick] (d12) -- (d13) ;

\node (b21) at (25.5,6){$1$};
\node[rectangle,color=red]  (b22) at (26.5,5){$2$};
\node[rectangle,color=red]  (b23) at (24.5,5){$3$};
 \draw[thick] (b21) -- (b23) ;

\node (c21) at (15.5,6){$1$};
\node[rectangle,color=red]  (c22) at (16.5,5){$2$};
\node[rectangle,color=red]  (c23) at (14.5,5){$3$};
 \draw[thick] (c21) -- (c22) ;

\node[rectangle,color=red]  (d21) at (30.5,6){$1$};
\node (d22) at (31.5,5){$2$};
\node[rectangle,color=red]  (d23) at (29.5,5){$3$};
 \draw[thick] (d22) -- (d23) ;

\node (e11) at (-7,11){$1$};
\node (e12) at (-6,10){$2$};
\node[rectangle,color=red]  (e13) at (-8,10){$3$};
\draw[thick] (e12)--(e11)--(e13);

\node[rectangle,color=red]  (e21) at (-2,11){$1$};
\node (e22) at (-1,10){$2$};
\node (e23) at (-3,10){$3$};
\draw[thick] (e22)--(e21)--(e23);

\node (e31) at (3,11){$1$};
\node[rectangle,color=red]  (e32) at (4,10){$2$};
\node (e33) at (2,10){$3$};
\draw[thick] (e32)--(e31)--(e33);

\node[rectangle,color=red]  (f11) at (8,11){$1$};
\node (f12) at (9,10){$2$};
\node (f13) at (7,10){$3$};
\draw[thick] (f11)--(f12)--(f13);

\node (f21) at (13,11){$1$};
\node[rectangle,color=red]  (f22) at (14,10){$2$};
\node (f23) at (12,10){$3$};
\draw[thick] (f21)--(f22)--(f23);

\node (f31) at (18,11){$1$};
\node (f32) at (19,10){$2$};
\node[rectangle,color=red]  (f33) at (17,10){$3$};
\draw[thick] (f31)--(f32)--(f33);

\node (g11) at (23,11){$1$};
\node[rectangle,color=red]  (g12) at (24,10){$2$};
\node (g13) at (22,10){$3$};
\draw[thick] (g11)--(g13)--(g12);

\node (g21) at (28,11){$1$};
\node (g22) at (29,10){$2$};
\node[rectangle,color=red]  (g23) at (27,10){$3$};
\draw[thick] (g21)--(g23)--(g22);

\node[rectangle,color=red]  (g31) at (33,11){$1$};
\node (g32) at (34,10){$2$};
\node (g33) at (32,10){$3$};
\draw[thick] (g31)--(g33)--(g32);

\tikzstyle{every path}=[line width=1.5pt]
 \draw (p1) -- (q1) ;
 \draw (p1) -- (q2) ;
 \draw (p1) -- (q3) ;
 \draw (p1) -- (q4) ;
 \draw (p1) -- (q5) ;
 \draw (p1) -- (q6) ;

 \draw (q1) -- (r1) ;
 \draw (q1) -- (r2) ;
 \draw (q2) -- (r2) ;
 \draw (q2) -- (r3) ;
 \draw (q3) -- (r4) ;
 \draw (q3) -- (r5) ;
 \draw (q4) -- (r5) ;
 \draw (q4) -- (r6) ;
 \draw (q5) -- (r7) ;
 \draw (q5) -- (r8) ;
 \draw (q6) -- (r8) ;
 \draw (q6) -- (r9) ;

\end{tikzpicture}
\end{center}
\caption[]{$\SF_3$ (here the roots are marked with red squares)}\label{figure:spanning_forests_3}
\end{figure}

Let $F \in \SF_n$ and $x,y$ be vertices in $F$, we say that $x$ is the \emph{parent} 
of $y$ (and $y$ the \emph{child} of $x$) if $\{x,y\}$ is an edge of $F$ that belongs to the unique path from $y$ to the root of their component in $F$. We  denote $x=p(y)$ whenever $x$ is the parent of $y$. We say that 
$x$ 
is an \emph{ancestor} of $y$ ($y$ is a descendant of $x$) if $x$  is in the unique path from $y$ to 
the root of their connected component in $F$. To every forest $F\in \SF_n$, we can associate an  
element  of $\Pi_n^w$,
\begin{align}\label{definition:piofanF}\pi(F):=\{V(T)^{w_T}\mid T 
\text{ a tree in } F\},
\end{align}
where $w_T$ is the number of \emph{descents} in $T$, i.e., edges 
$\{x,p(x)\}$ in $T$ where $p(x)>x$. As 
an example, if $F$ is the forest formed by the trees $T_1$ to $T_6$ of 
Figure~\ref{figure:example_unique_increasing_chain},  then 
$w_{T_1}= 0$, $w_{T_2}= 0$, $w_{T_3}= 
0$, $w_{T_4}= 3$, $w_{T_5}= 1$, $w_{T_6}= 1$ and thus $\pi(F) = 
1689\underline{11}\,\underline{14}^{3}/23^{0}/4^0/5^0/\underline{10}\,\underline{12}^{1} 
/7\underline{13}^{1}$ (where we have underlined the two digit numbers). 
The following propositions were proved in \cite{DleonWachs2016} and \cite{Sagan1983}.

\begin{proposition}[\cite{DleonWachs2016}]\label{proposition:mobiusweightedpartitions}
For all $x\in \Pi_n^w$,
$$\mu([\hat{0},x])=(-1)^{\rho(x)}|\{F\in \SF_n \mid \pi(F)=x\}|.$$
In particular,
$$\mu([\hat{0},[n]^i])=(-1)^{n-1}|\T_{n,i}|,$$
where $\T_{n,i}$ is the set of rooted trees with $i$ descents.
\end{proposition}

\begin{proposition}[\cite{Sagan1983}]\label{proposition:mobiusrootedforests}
For all $F\in \SF_n$,
$$\mu([\hat{0},F])=\begin{cases}
(-1)^{\rho(F)} & \text{if $F$ is a forest in which every nonroot vertex is a leaf} \\
0 & \text{otherwise}.
  \end{cases} $$
\end{proposition}

Note that the map $\pi:\SF_n\rightarrow \Pi_n^w$ defined in Equation~\ree{definition:piofanF} is a bijection when we restrict the domain to the 
set of forests in which every one of its nonroot vertices are leaves. Furthermore, note that a 
forest $F \in \SF_n$ and its associated weighted partition $\pi(F)\in \Pi_n^w$ have the same rank 
in their respective posets. We obtain the following theorem as a corollary of the previous two 
propositions.

\begin{theorem}[Gonz\'alez D'Le\'on - 
Wachs, personal communication]\label{theorem:weightedpartitions_rootedforests_whitneyduals}
We have that
\begin{align*}
 w_k(\Pi_n^w)&=|\{F\in \SF_n \mid \rho(F)=n-k\}|=W_k(\SF_n)\\
 w_k(\SF)&=|\{x\in \Pi_n^w \mid \rho(x)=n-k\}|=W_k(\Pi_n^w).
\end{align*}
 Hence, the posets $\Pi_n^w$ and $\SF_n$ are Whitney duals.
\end{theorem}

\subsubsection{A CW-labeling of $\Pi_n^w$}

In the following we will use the structure of the poset $\SF_n$ and a surjective map from the poset 
of saturated chains $C(\Pi_n^w)$ to $\SF_n$ to give a CW-labeling $\lambda_C$ for $\Pi_n^w$. 
Theorem \ref{theorem:CW}  will then imply that $Q_{\lambda_C}(\Pi_n^w)$ is a Whitney dual of $\Pi_n^w$. 
We will also show that $Q_{\lambda_C}(\Pi_n^w)\simeq \SF_n$ providing a different proof of 
Theorem \ref{theorem:weightedpartitions_rootedforests_whitneyduals}.

We  first discuss some structural properties of $\SF_n$. For a rooted tree $T$ on vertex 
set $V(T)$, let $r(T)$ denote the \emph{root} of $T$; and  for any pair of vertices $v,w\in V(T)$ 
define the \emph{distance} $d(v,w)$ to be the number of edges in the unique path between $v$ and 
$w$. Note that by the definition of the cover relations in $\SF_n$, a saturated 
chain $\hat{0}=F_0\lessdot F_1 \lessdot \cdots \lessdot F_k=F$ in $\SF_n$ can be seen as a 
step-by-step instruction set on how to build the forest $F$. Here in the $i$-th step, exactly two 
trees $T_{i,1}$ and $T_{i,2}$ of $F_{i-1}$ are combined by adding the edge $\{r(T_{i,1}),r(T_{i,2})\}$ to get a new tree $T'_i$ in $F_i$ whose root $r(T'_i)\in 
\{r(T_{i,1}),r(T_{i,2})\}$. We call this process \emph{merging} the trees $T_{i,1}$ and $T_{i,2}$ by the roots  $r(T_{i,1})$ and  $r(T_{i,2})$. At $F_0=\hat{0}$ every element of $[n]$ is a root ($R(F)=[n]$), but at 
each step $i$ there is exactly one element $v_i \in \{r(T_{i,1}),r(T_{i,2})\}$ that stops being a 
root, i.e., $v_i \in R(F_{i-1})$ and $v_i \notin R(F_i)$. This process defines an ordered listing 
$v_1,v_2,\dots,v_k$ of the non-root vertices of $F$. 

We can also consider the converse situation:
let $v_1,v_2,\dots,v_k$ be an ordered listing of the non-root vertices of $F$ and let 
$\hat{0}=F_0, F_1, \cdots ,F_k=F$ be the sequence defined by obtaining $F_{i}$  from 
$F_{i-1}$ by adding the edge $\{v_i,p(v_i)\}$ (where $p(v_i)$ is the parent of $v_i$ in $F$) and letting $R(F_i)=R(F_{i-1})\setminus\{v_i\}$. It 
is 
clear that the rank $\rho(F_i)=i$ in $\SF_n$, but it is not clear if the set $\hat{0}=F_0, F_1, 
\cdots ,F_k=F$ forms a saturated chain in $\SF_n$ since a cover relation in $\SF_n$ happens {\bf 
exactly} when the two trees in $F_{i-1}$ are merged using an edge between their roots. The 
following lemma characterizes which  
sequences 
$v_1,v_2,\dots,v_k$ give a valid saturated chain $\hat{0}=F_0\lessdot F_1 \lessdot \cdots \lessdot 
F_k=F$ in $\SF_n$.

An ordered listing $v_1,v_2,\dots,v_k$  of some subset of $V(F)$ is said to be a \emph{linear 
extension} if whenever $v_i$ is a descendant of $v_j$ in $F$ then we have that $i<j$.

\begin{lemma}\label{lemma:building_forests}
 The sequence $\hat{0}=F_0,F_1, \cdots , F_k=F$ in which $F_{i}$  is obtained from 
$F_{i-1}$ by adding the edge $\{v_i,p(v_i)\}$ and setting $R(F_i)=R(F_{i-1})\setminus\{v_i\}$ is a 
saturated chain in $\SF_n$ if and only if the ordered listing $v_1,v_2,\dots,v_k$ of the 
non-root vertices of $F$ is a linear extension.
\end{lemma}
\begin{proof}
 First let $\hat{0}=F_0\lessdot F_1 \lessdot \cdots \lessdot F_k=F$ be a saturated chain in 
$\SF_n$ and suppose that the associated ordered listing $v_1,v_2,\dots,v_k$ is not a linear 
extension. Then there are 
non-root vertices $v_i$ and $v_j$ in $F$ such that $v_i$ is a descendant of $v_j$ and $j<i$. We 
choose $v_i$ and $v_j$ such that $d(v_i,v_j)$ is minimal and we claim that in this case it must be 
that $v_j=p(v_i)$. Otherwise, there is a vertex $v_l$ that is a descendant of $v_j$ and an ancestor 
of $v_i$. If $l < j$ then the pair $(v_i,v_l)$ satisfies the condition above with 
$d(v_i,v_l)<d(v_i,v_j)$ and if $l >j$ then the pair $(v_l,v_j)$ satisfies the condition above 
with 
$d(v_l,v_j)<d(v_i,v_j)$. Now, if $v_j=p(v_i)$ and $j<i$ this implies that in the step $i$ between 
$F_{i-1}$ and $F_{i}$ we added the edge $\{v_i,v_j\}$ but $v_j \notin R(F_{i-1})$ since $v_j$ has 
been already removed from the set of roots in step $j$.  This implies that $F_i$ does not cover 
$F_{i-1}$, that is a contradiction. We conclude that 
$v_1,v_2,\dots,v_k$ is a linear 
extension. 

On the other hand, let the ordered listing $v_1,v_2,\dots,v_k$ be a linear extension of the non-root 
vertices of $F$ and let $\hat{0}=F_0,F_1, \cdots , F_k=F$ in which $F_{i}$ is defined as stated in 
the lemma. Note that $F_{k-1}$ is the forest obtained from $F$ by 
removing the edge $\{v_k,p(v_k)\}$ and is such that $R(F_{k-1})=R(F)\cup \{v_k\}$. We also clearly 
have 
that $v_1,v_2,\dots,v_{k-1}$ is a linear extension of the non-root vertices of $F_{k-1}$. 
Since we know that the trivial sequence $F_0=\hat{0}$ is a saturated chain in $\SF_n$ we assume by 
induction that $\hat{0}=F_0\lessdot F_1 \lessdot \cdots \lessdot F_{k-1}$ is also one. Note that 
$p(v_k) 
\neq v_j$ for $j=1,\dots,k$ since the ordered listing is a linear extension and so  $p(v_k)$ 
cannot appear before $v_k$. But this implies that both $v_k$ and $p(v_k)$ are in $R(F_{k-1})$ and 
that $F_k$ is obtained from $F_{k-1}$ by adding the edge $\{v_k,p(v_k)\}$  such that 
$R(F_{k})=R(F_{k-1})\setminus \{v_k\}$.  Thus $F_{k-1}\lessdot F_k$ and so $\hat{0}=F_0\lessdot 
F_1 \lessdot \cdots \lessdot F_k=F$ is a saturated chain  in $\SF_n$.
\end{proof}

We 
define the 
\emph{cost} of a rooted tree $T$ as 
\begin{align}
 \gamma(T)=\sum_{v \in V(T)}d(v,r(T)).
\end{align}
See Figure \ref{figure:example_weight_of_a_tree} for an example.
 
\begin{figure}[h]
\begin{center} 
\begin{tikzpicture}[scale=1]

\tikzstyle{every node}=[black, draw, circle, inner sep=0.4pt, scale=1, minimum width=4pt]
\tikzstyle{every path}=[black, line width=0.005in]

\node[pin={[color=red, pin distance=3pt, inner sep=0pt]120:\tiny$1$}] (f1) at (-1,6){$1$};
\node[pin={[color=red, pin distance=3pt, inner sep=0pt]120:\tiny$0$}]  (f2) at (0,7){$2$};
\node[pin={[color=red, pin distance=3pt, inner sep=0pt]120:\tiny$2$}]  (f3) at (0.5,5){$3$};
\node[pin={[color=red, pin distance=3pt, inner sep=0pt]30:\tiny$2$}]  (f4) at (1.5,5){$4$};
\node[pin={[color=red, pin distance=3pt, inner sep=0pt]120:\tiny$1$}]  (f5) at (0,6){$5$};
\node[pin={[color=red, pin distance=3pt, inner sep=0pt]30:\tiny$1$}]  (f6) at (1,6){$6$};
 \draw[] (f2) -- (f1);
 \draw[] (f2) -- (f5);
 \draw[] (f2) -- (f6);
 \draw[] (f3) -- (f6);
 \draw[] (f4) -- (f6);

\end{tikzpicture}
\end{center}
\caption[]{A tree $T$ with cost $\gamma(T)=7$}\label{figure:example_weight_of_a_tree}
\end{figure}
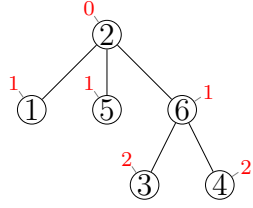

Recall that  $C(\Pi_n^w)$ is the poset 
of saturated chains from $\hat{0}$ in 
$\Pi_n^w$.
For a chain $$\c:(\hat{0}=x_0\lessdot x_1 \lessdot \cdots \lessdot x_{\ell})\in C(\Pi_n^w)$$  let 
$\c_i$ be the subchain of $\c$ consisting of its bottom $i+1$ elements, i.e.,
$$\c_i:(x_0\lessdot x_1 \lessdot \cdots \lessdot x_{i}).$$ 
We recursively define 
 a map $\F:C(\Pi_n^w)\rightarrow \SF_n$ as follows:
\begin{enumerate}
 \item Let $\F(\hat{0}_{C(\Pi_n^w)})=\hat{0}_{\SF_n}$, the rooted forest on $[n]$ with no edges. 
Note 
that $\pi(\hat{0}_{\SF_n})=\hat{0}_{\Pi_n^w}$ where the function $\pi$ is defined in Equation~\ree{definition:piofanF}.
 \item  Let $\c:(\hat{0}=x_0\lessdot x_1 
\lessdot \cdots \lessdot x_{k}) \in C(\Pi_n^w)$. 
The cover relation  
$$x_{k-1}:=A_1^{w_1}/A_2^{w_2}/\dots /A_s^{w_s}\lessdot B_1^{v_1}/ B_2^{v_2}/ \dots 
/B_{s-1}^{v_{s-1}}=:x_k$$ in $\Pi_n^w$ is such that exactly two weighted blocks $A_i^{w_i}$ 
and $A_j^{w_j}$ of 
$x_{k-1}$ are 
merged into a weighted block $B_m^{v_m}$ of $x_k$ where  $B_m=A_{i}\cup A_{j}$ and $v_m-(w_{i} 
+ 
w_{j}) \in \{0,1\}$.  We assume that we have recursively defined 
$\F(\c_{k-1})$ as a rooted spanning forest with the property that $\pi(\F(\c_{k-1}))=x_{k-1}$. 
We then define $\F(\c)$ to be the rooted spanning forest obtained from 
$\F(\c_{k-1})$ by connecting the two trees $T_i$ and 
$T_j$ with vertex sets $V(T_i)=A_i$ and 
$V(T_j)=A_j$ using the edge $\{r(T_i),r(T_j)\}$ and choosing the root of the new tree to be
$\min\{r(T_i),r(T_j)\}$ if 
$v_m-(w_{i} 
+ w_{j})=0$ or   $\max\{r(T_i),r(T_j)\}$ if $v_m-(w_{i} + w_{j})=1$.
\end{enumerate}
Note that by construction and equation \eqref{definition:piofanF}, we have that $\pi(\F(\c))=x_{k}$ 
and so the map $\F$ is inductively well-defined.  Moreover, it has  the property that for every $\c \in C(\Pi_n^w)$ 
we have that $\pi(\F(\c))=e(\c)$ (where, just as before, $e(\c)$ is maximum element of $\c$). Indeed, the 
blocks of $x_k$ are the same as the ones of $x_{k-1}$ except for $A_i^{w_i}$ and 
$A_j^{w_j}$ that now form the block $(A_i \cup A_j)^{v_m}$. It is not hard to see that $\F$ is 
sending cover relations in $C(\Pi_n^w)$ to cover relations in $\SF_n$ and so $\F$ is a rank 
preserving poset map.

\begin{lemma}\label{lemma:saturatedchainsinPIandSF}
Let $\widetilde \F:C(\Pi_n^w)\rightarrow C(\SF_n)$ be defined for $\c:(\hat{0}=x_0\lessdot 
x_1\lessdot \cdots \lessdot x_{k}) \in C(\Pi_n^w)$ by
$$\widetilde \F(\c)=\F(\c_0)\lessdot \F(\c_1)\lessdot \cdots \lessdot \F(\c_k),$$
and $\widetilde \pi:C(\SF_n)\rightarrow C(\Pi_n^w)$ defined for $\FF:(\hat{0}=F_0\lessdot 
F_1\lessdot \cdots \lessdot F_k)\in C(\SF_n)$ by 
$$\widetilde \pi(\FF)=(\pi(F_0)\lessdot \pi(F_{1})\lessdot 
\cdots \lessdot \pi(F_{k})).$$
Then we have that these maps are well-defined and that $\widetilde \pi \circ 
\widetilde \F=id_{C(\Pi_n^w)}$ and $\widetilde \F\circ \widetilde \pi=id_{C(\SF_n)}$. Hence, 
$\widetilde \F$ and $\widetilde \pi$ define an isomorphism $C(\Pi_n^w)\simeq C(\SF_n)$. 
\end{lemma}
\begin{proof}
Note that if $F\lessdot F'$, we have that $F'$ is obtained from $F$ by merging two trees $T_1$ and 
$T_2$ to get a tree $T'$ of $F'$ such that $r(T')=r(T_1)$. Hence we have that $\pi(F)$ and $\pi(F')$ 
are almost identical except that they differ in two weighted blocks $V(T_1)^{w_{T_1}}$ and 
$V(T_2)^{w_{T_2}}$ 
of $\pi(F)$ and one weighted block $(V(T_1)\cup V(T_2))^{w_{T_1}+w_{T_2}+\chi(r(T_1)>r(T_2))}$ of 
$\pi(F')$ (where $\chi(A)=1$ if the statement $A$ is satisfied and $0$ otherwise). This means 
exactly that $\pi(F)\lessdot \pi(F')$ so $\widetilde \pi$ is a well-defined map that preserves 
cover relations and hence is order-preserving. The comments preceding this lemma also 
imply that $\widetilde \F$ is a well-defined order-preserving map. 

By the recursive definition of $\F$ it follows that $\widetilde \pi \circ \widetilde 
\F=id_{C(\Pi_n^w)}$. The reader can verify using induction on the length $k$ of a chain 
$\FF:(\hat{0}=F_0\lessdot F_1\lessdot \cdots \lessdot F_k)\in C(\SF_n)$ that $\widetilde \F\circ 
\widetilde \pi=id_{C(\SF_n)}$. Hence we obtain the desired isomorphism $C(\Pi_n^w)\simeq C(\SF_n)$.
\end{proof}
\begin{remark}
Note that Lemma \ref{lemma:saturatedchainsinPIandSF} implies that the information encoded in the 
saturated chain $\c:(\hat{0}=x_0\lessdot x_1\lessdot \cdots \lessdot x_{k}) \in C(\Pi_n^w)$ can be 
recovered uniquely from the saturated chain 
$$\widetilde \F(\c)=\F(\c_0)\lessdot \F(\c_1)\lessdot \cdots \lessdot \F(\c_k)  \in C(\SF_n).$$
We will use this fact in the construction of a C-labeling of $\Pi_n^w$.
\end{remark}

We define first an E-labeling $\lambda_{\SF}:\E(\SF_n)\rightarrow \Lambda$ as follows: Let 
$F\lessdot F'$ be such that $T_i$ and $T_j$ are the trees of $F$ that have been merged to get a 
tree $T'$ of $F'$ and assume without loss of generality that $r(T_i)<r(T_j)$. We define
\begin{align}\label{equation:Elabelingforests}
\lambda_{\SF}(F\lessdot F')= 
\begin{cases}
\left(-\gamma(T_j),r(T_i),r(T_j)\right) & \text{if $r(T')=r(T_i)$} \\
\left(-\gamma(T_i),r(T_j),r(T_i)\right) & \text{if $r(T')=r(T_j)$},
  \end{cases} 
\end{align}
and we define $\Lambda$ to be the poset $\ZZ^3$ with lexicographic order.

We will define now a C-labeling $\lambda_C:\M\E(\Pi_n^w)\rightarrow \Lambda$ as follows: 
Let $\c:(\hat{0}=x_0\lessdot x_1 \lessdot \cdots \lessdot x_{\rho(\Pi_n^w)})$ be a maximal chain 
of $\Pi_n^w$. We define
\begin{align}\label{equation:Clabelingweightedpartitions}
\lambda_C(\c, x_{i-1} \lessdot x_i)= \lambda_{\SF}(\F(\c_{i-1})\lessdot \F(\c_i)).
\end{align}
The bottom to top 
construction 
that  we have used to define $\lambda_C$, i.e.,  using the information on the saturated chains from 
$\hat{0}$ to $x_{i-1}$, guarantees that this labeling is a C-labeling of $\Pi_n^w$. Indeed, for any 
maximal chain $\tilde \c$ that coincides with $\c$ in the bottom $d$ 
elements we have that $\F(\tilde \c_i)=\F(\c_i)$ for $i=0,1,\dots,d-1$ and so $\tilde \c$ shares 
the same labels with $\c$ along the first $d-1$ edges. We will prove that $\lambda_C$ is in fact a 
CW-labeling.

Note that the definition of $\lambda_C$ says that for a chain 
$$\c:( \hat{0}=x_0\lessdot x_1 \lessdot \cdots \lessdot x_k),$$
the label $\lambda_C(\c, x_{i-1} \lessdot x_{i})$ 
only depends on the forests $\F(\c_{i-1})$ and $\F(\c_i)$. 
Hence the label sequence 
$$\lambda_C(\c, x_i \lessdot x_{i+1}), \lambda_C(\c, x_{i+1} \lessdot x_{i+2}), \dots , 
\lambda_C(\c, x_{j-1}\lessdot x_{j})$$
only depends on the sequence
$\F(\c_{i})\lessdot \F(\c_{i+1})\lessdot \cdots \lessdot \F(\c_{j})$.
\begin{lemma}\label{lemma:increasingmergingoftrees}
Let the sequence  $F_i\lessdot F_{i+1}\lessdot \cdots \lessdot F_{j}$ be a saturated chain in $\SF_n$ such 
that the following holds:
\begin{itemize}
 \item The forest $F_{i+1}$ is obtained from $F_{i}$ by merging two trees $T_1$ and $T_2$ of $F_{i}$ 
to obtain a new tree $T'$ in $F_{i+1}$ with $r(T')=r(T_1)$, and
 \item the forest $F_{j}$ is obtained by merging $T'$ with another tree $T_3$ to obtain a new 
tree $T''$ with root $r(T'')=r(T_3)\neq r(T')=r(T_1)$.
\end{itemize}
Then
$$\lambda_{\SF}(F_{i}\lessdot F_{i+1})>\lambda_{\SF}(F_{j-1}\lessdot F_{j}).$$
\end{lemma}

\begin{proof}
Under these assumptions we have that $$\lambda_{\SF}(F_{i}\lessdot 
F_{i+1})=(-\gamma(T_2),r(T_1),r(T_2))>(-\gamma(T'),r(T_3),r(T'))=\lambda_{\SF}(F_{j-1}\lessdot 
F_{j}),$$
since $T_2$ is a proper subtree of $T'$ and so $\gamma(T')>\gamma(T_2)$.
\end{proof}

\begin{proposition}\label{proposition:weightedpartitionsCR}
 $\lambda_C:\M\E(\Pi_n^w)\rightarrow \Lambda$ is a CR-labeling.
\end{proposition}
\begin{proof}
To show that $\lambda_C$ is a $CR$-labeling we have to show that in each rooted interval 
$[x,y]_{\c}$ there is a unique increasing chain. Another way to describe this is to say 
that for any saturated chain $\c:(\hat{0}=x_0\lessdot x_1 \lessdot \cdots \lessdot x_k=x)$ and 
 $y\in \Pi_n^w$ such that $x< y$ there is a unique saturated chain 

$$\tilde \c:(\hat{0}=x_0\lessdot x_1 \lessdot \cdots \lessdot x_k\lessdot x_{k+1}\lessdot \cdots 
\lessdot x_{\rho(y)}=y)$$ such that

\begin{align}\lambda_C(\tilde \c, x_k \lessdot x_{k+1})<\lambda_C(\tilde \c, x_{k+1} \lessdot 
x_{k+2})<\cdots < \lambda_C(\tilde \c, x_{\rho(y)-1}\lessdot 
x_{\rho(y)}).\label{equation:increasing_condition}
\end{align}

Note that by the comments that precede Lemma \ref{lemma:increasingmergingoftrees} this label 
sequence only depends on the saturated chain
$\F(\tilde \c_{k})\lessdot \F(\tilde \c_{k+1})\lessdot \cdots \lessdot \F(\tilde \c_{\rho(y)})$ in 
$\SF_n$. Hence if we are able to determine that there is a unique saturated chain $\F(\tilde 
\c_{k})=F_k \lessdot F_{k+1} \lessdot \cdots \lessdot F_{\rho(y)}$ in $\SF_n$ with $\pi( 
F_{\rho(y)})=y$ that has a sequence of labels that is increasing, then as a consequence of Lemma 
\ref{lemma:saturatedchainsinPIandSF}, the chain 
$$\hat{0}=x_0\lessdot x_1 \lessdot \cdots \lessdot x_k=x=\pi(F_k)\lessdot \pi(F_{k+1})\lessdot 
\cdots \lessdot \pi(F_{\rho(y)})=y$$ is the unique saturated chain $\tilde \c$ with the 
desired property.

 Lemma \ref{lemma:increasingmergingoftrees} implies that for $\tilde \c$ to satisfy the increasing 
condition in equation \eqref{equation:increasing_condition} from steps $k+1$ to $\rho(y)$ we can 
only merge trees in a way that after a root $r(T_1)$ has been chosen in a step between $k+1$ 
and $\rho(y)$ the same root has to continue being a root in all consecutive steps.
Now, recall from the recursive definition of the map $\F$ that the process of selecting a root 
depends on the value of $u \in \{0,1\}$ where $A_1^{w_1}$ 
and $A_2^{w_2}$ are the blocks of $x_{s-1}$ that are merged to obtain the block $(A_1\cup 
A_2)^{w_1+w_2+u}$ of $x_{s}$. By the definition of the order relation in $\Pi_n^w$, we have that 
each block in $y$ is of the form $(A_{1}\cup A_{2}\cup \cdots \cup A_{l})^{w_1+w_2+\cdots + 
w_l+v}$ where $v \in \{0,1,\dots,l-1\}$ and $A_1^{w_1}, A_2^{w_2},\dots,A_l^{w_l}$ are blocks 
of $x_{k}$. Assume that in $\F(\c_k)=F_k$ the trees corresponding to these weighted blocks are 
$T_1, 
T_2,\dots,T_l$ and without loss of generality assume that the indexing is such that 
$r(T_1)<r(T_2)<\cdots < r(T_l)$. The reader can easily check that there is a unique tree 
$T'$ obtained by merging the $l$ trees by the roots step by step 
selecting at each step the same root $r(T_j)$ such that exactly $v$ of the 
other $l-1$ roots are 
smaller than $r(T_j)$. In fact this tree is the one where $j=v+1$. All labels that come from 
the step-by-step merging process that creates $T'$ are then of the form 
$(-\gamma(T_i),r(T_{v+1}),r(T_i))$ for $i\in \{1,\dots,l\}\setminus\{v+1\}$. If there are two trees 
$T_i$ and $T_j$ for $i,j\in \{1,\dots,l\}\setminus\{v+1\}$ such that $\gamma(T_i)<\gamma(T_j)$ then
$$(-\gamma(T_j),r(T_{v+1}),r(T_j))<(-\gamma(T_i),r(T_{v+1}),r(T_i)),$$ and if 
$\gamma(T_i)=\gamma(T_j)$ but 
$r(T_i)<r(T_j)$ then $$(-\gamma(T_i),r(T_{v+1}),r(T_i))<(-\gamma(T_j),r(T_{v+1}),r(T_j)).$$ 
Hence there is a unique increasing way of constructing $T'$ by attaching the roots $r(T_i)$ to the 
selected root $r(T_{v+1})$ by going first in reverse order of $\gamma(T_i)$ and then in order of 
$r(T_i)$.  

As an example of the argument above, consider the forest formed by the trees $T_1$ to $T_6$ of 
Figure \ref{figure:example_unique_increasing_chain} with
$r(T_1)<r(T_2)<r(T_3)<r(T_4)<r(T_5)<r(T_6)$. Suppose that we want to find an increasing maximal 
chain in the rooted interval 
$$\left[1689\underline{11}\,\underline{14}^{3}/23^{0}/4^0/5^0/\underline{10}\,\underline{12}^{1} 
/7\underline{13}^{1}, 
123456789\,\underline{10}\,\underline{11}\,\underline{12}\,\underline{13}\,\underline{14}^{	
7}\right]_{\c}$$ where $\c$ is a saturated chain from $\hat{0}$ to 
$1689\underline{11}\,\underline{14}^{3}/23^{0}/4^0/5^0/\underline{10}\,\underline{12}^{1}
/7\underline{13}^{1}$ such that $\F(\c)$ is the forest $\{T_1,T_2,T_3,T_4,T_5,T_6\}$. Since 
$7-w_{T_1}-w_{T_2}-w_{T_3}-w_{T_4}-w_{T_5}-w_{T_6}=2$, the unique increasing chain $\tilde \c$ in 
this rooted interval produces a tree $T'$ whose root $r(T')=r(T_3)$ (since the edges 
$\{r(T_3),r(T_1)\}$ and $\{r(T_3),r(T_2)\}$ will create exactly the $2$ additional descents in 
$T'$). Since $\gamma(T_4)>\gamma(T_1)=\gamma(T_5)=\gamma(T_6)>\gamma(T_2)$ the new steps in $\tilde 
\c$ consist in adding first the edge $\{r(T_3),r(T_4)\}$ then the edges $\{r(T_3),r(T_1)\}$, 
$\{r(T_3),r(T_5)\}$, $\{r(T_3),r(T_6)\}$ in increasing order of their roots and finally the edge 
$\{r(T_3),r(T_2)\}$. This will give the sequence of labels
$$(-7,5,6)<(-1,5,2)<(-1,5,12)<(-1,5,13)<(0,5,4).$$

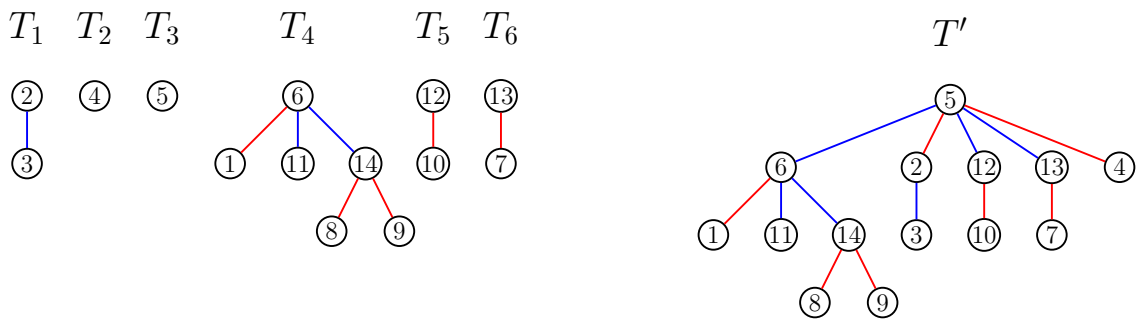
\begin{figure}
 \begin{tikzpicture}[scale=0.9]

\begin{scope}[xshift=0]
\node  at (-10,8){\large$T_1$};
\node  at (-9,8){\large$T_2$};
\node  at (-8,8){\large$T_3$};
\node  at (-6,8){\large$T_4$};
\node  at (-4,8){\large$T_5$};
\node  at (-3,8){\large$T_6$};

\tikzstyle{every node}=[black, draw, circle, inner sep=0.4pt, scale=0.8, minimum width=14pt]
\tikzstyle{every path}=[blue, line width=0.01in]

\node (a1) at (-10,6){$3$};
\node (a2) at (-10,7){$2$};
 \draw[color=blue] (a1) -- (a2);

\node (d1) at (-9,7){$4$};
\node (e1) at (-8,7){$5$};

\node (f1) at (-7,6){$1$};
\node (f2) at (-6,7){$6$};
\node  (f3) at (-5.5,5){$8$};
\node  (f4) at (-4.5,5){$9$};
\node (f5) at (-6,6){$11$};
\node  (f6) at (-5,6){$14$};
 \draw[color=red] (f2) -- (f1);
 \draw[] (f2) -- (f5);
 \draw[] (f2) -- (f6);
 \draw[color=red] (f3) -- (f6);
 \draw[color=red] (f4) -- (f6);

\node (b1) at (-4,6){$10$};
\node (b2) at (-4,7){$12$};
 \draw[color=red] (b1) -- (b2);

\node (c1) at (-3,6){$7$};
\node (c2) at (-3,7){$13$};
 \draw[color=red] (c1) -- (c2);
\end{scope}

\begin{scope}[xshift=260,yshift=-30]

\node  at (-5.5,9){\large$T^{\prime}$};
\tikzstyle{every node}=[black, draw, circle, inner sep=0.4pt, scale=0.8, minimum width=14pt]
\tikzstyle{every path}=[blue, line width=0.01in]

\node (a1) at (-6,6){$3$};
\node (a2) at (-6,7){$2$};
 \draw[color=blue] (a1) -- (a2);

\node (d1) at (-3,7){$4$};
\node (e1) at (-5.5,8){$5$};

\node (f1) at (-9,6){$1$};
\node (f2) at (-8,7){$6$};
\node  (f3) at (-7.5,5){$8$};
\node  (f4) at (-6.5,5){$9$};
\node (f5) at (-8,6){$11$};
\node  (f6) at (-7,6){$14$};
 \draw[color=red] (f2) -- (f1);
 \draw[] (f2) -- (f5);
 \draw[] (f2) -- (f6);
 \draw[color=red] (f3) -- (f6);
 \draw[color=red] (f4) -- (f6);

\node (b1) at (-5,6){$10$};
\node (b2) at (-5,7){$12$};
 \draw[color=red] (b1) -- (b2);

\node (c1) at (-4,6){$7$};
\node (c2) at (-4,7){$13$};
 \draw[color=red] (c1) -- (c2);

 \draw[color=red] (e1) -- (a2);
 \draw[color=red] (e1) -- (d1);
 \draw[color=blue] (e1) -- (f2);
 \draw[color=blue] (e1) -- (c2);
 \draw[color=blue] (e1) -- (b2);

\end{scope}

\end{tikzpicture}
 \caption{Example of the unique tree that can be constructed along an increasing chain in a rooted 
interval $[x,y]_{\c}$ of $\Pi_n^w$. Here the edges have been distinguished using red for descending edges and  blue for ascending edges.} 
 \label{figure:example_unique_increasing_chain}
\end{figure}

Finally, the sequence $\F(\tilde \c_{k})=F_k\lessdot F_{k+1}\lessdot \cdots \lessdot 
F_{\rho(y)}$ describes a process of merging the trees in $F_k$ until we obtain 
the various trees in $F_{\rho(y)}$. By the discussion above we have that each tree $T' 
\in F_{\rho(y)}$ uniquely determines the subsequence of steps to build it given that the 
labels 
$(-\gamma(T_i),r(T_{v+1}),r(T_i))$ are required to satisfy equation 
\eqref{equation:increasing_condition}. Furthermore, since these labels are all distinct and belong 
to the total order $\Lambda=\mathbb{Z}^3$,  there is a unique way of organizing the labels (the unique increasing 
shuffle of all the increasing subsequences of labels for all trees) for the various trees 
$T'$ in an increasing order, hence there is a unique sequence $\F(\tilde \c_{k})=F_k\lessdot 
F_{k+1}\lessdot \cdots \lessdot 
F_{\rho(y)}$ with $\pi(F_{\rho(y)})=y$, that satisfy the increasing condition in equation 
\eqref{equation:increasing_condition}.
\end{proof}

\begin{lemma}\label{lemma:ranktwoswitchingweightedpartitions}
 The labeling $\lambda_C$ satisfies the rank two switching property. Moreover for 
every maximal chain $\c$ in $\Pi_n^w$ and $i=1,...,\rho(\Pi_n^w)$ we have that $\F(\c)=\F(U_i(\c))$, where $U_i(\c)$ is the unique saturated chain obtained after possibly applying a quadratic exchange to $\c$ at rank $i$.
\end{lemma}
\begin{proof}
Any rank two increasing sequence
$$\lambda_C(\c,x_{k-1}\lessdot x_k)<\lambda_C(\c,x_k\lessdot x_{k+1})$$ is determined by a 
sequence $\F(\c_{k-1})\lessdot \F(\c_{k})\lessdot \F(\c_{k+1})$ of forests. In any such sequence we 
have two possible cases.
\begin{enumerate}
 \item[{\bf Case I}:] We start with trees $T_1,T_2,T_3,T_4$ of $\F(\c_{k-1})$ and obtain trees 
$T'$ 
and $T''$ of $\F(\c_{k+1})$ such that $T'$ is obtained by merging $T_1$ and $T_2$ with 
$r(T')=r(T_1)$; and $T''$ is obtained by merging $T_3$ and $T_4$ with $r(T'')=r(T_3)$. In this 
case, 
the label sequence is $(-\gamma(T_2),r(T_1),r(T_2))<(-\gamma(T_4),r(T_3),r(T_4))$ and there is 
another 
forest $\F(\c_{k-1})\lessdot F \lessdot \F(\c_{k+1})$ obtained by merging first $T_3$ and 
$T_4$ and then merging $T_1$ and $T_2$ in $F$ to get $\F(\c_{k+1})$. Let $\tilde 
x_k=\pi( F)$ and observe (by considering the definition of the cover relations in $\Pi_n^w$) that 
$\pi(\F(\c_{k-1}))\cover \tilde x_k \cover \pi(\F(\c_{k+1}))$. Then the chain $\tilde \c= \c \cup 
\{\tilde x_k\} \setminus  \{x_k\}$ satisfies that $\lambda_C(\tilde\c,x_{k-1}\lessdot \tilde 
x_k)=(-\gamma(T_4),r(T_3),r(T_4))=\lambda_C(\c,x_{k}\lessdot x_{k+1})$ and 
$\lambda_C(\tilde\c, \tilde x_{k}\lessdot  
x_{k+1})=(-\gamma(T_2),r(T_1),r(T_2))=\lambda_C(\c,x_{k-1}\lessdot x_{k}).$

 \item[{\bf Case II}:]
In this case we start with trees $T_1,T_2,T_3$ of $\F(\c_{k-1})$ and obtain a tree $T''$ 
in $\F(\c_{k+1})$ by first merging $T_1$ and $T_2$ to get $T'$ with $r(T')=r(T_1)$ in $\F(\c_{k})$ 
and then merge $T'$ and $T_3$ in $\F(\c_{k})$ to get $T''$. Note that by Lemma 
\ref{lemma:increasingmergingoftrees}, it is necessarily true that $r(T'')=r(T')=r(T_1)$, 
otherwise the labels would not be increasing.  In this case, 
the label sequence is $(-\gamma(T_2),r(T_1),r(T_2))<(-\gamma(T_3),r(T_1),r(T_3))$.
Again here there is 
another 
forest $\F(\c_{k-1})\lessdot F \lessdot \F(\c_{k+1})$ obtained by merging first $T_1$ and 
$T_3$ to get a tree $\tilde T'$ with $r(\tilde T')=r(T_1)$ in $ F$  and then merging $\tilde 
T'$ and $T_2$ in $ F$ to get again the same $T''$ in $\F(\c_{k+1})$. We let $\tilde 
x_k=\pi( F)$ and observe here again that 
$\pi(\F(\c_{k-1}))\cover \tilde x_k \cover \pi(\F(\c_{k+1}))$. Then $\tilde \c= \c \cup 
\{\tilde x_k\} \setminus  \{x_k\}$ satisfies that $\lambda_C(\tilde\c,x_{k-1}\lessdot \tilde 
x_k)=(-\gamma(T_3),r(T_1),r(T_3))=\lambda_C(\c,x_{k}\lessdot x_{k+1})$ and 
$\lambda_C(\tilde\c, \tilde x_{k}\lessdot  
x_{k+1})=(-\gamma(T_2),r(T_1),r(T_2))=\lambda_C(\c,x_{k-1}\lessdot x_{k}).$
\end{enumerate}
Note that in the two cases above the choice of the element $\tilde x_k$ is unique. A 
label of the form $(-\gamma(T_2),r(T_1),r(T_2))$ determines exactly that the trees of 
$\F(\c_{k-1})$ with roots $r(T_1)$ an $r(T_2)$ are being merged into a tree $T'$ of $\F(\c_{k})$  
with root $r(T')=r(T_1)$. Hence Lemma \ref{lemma:saturatedchainsinPIandSF} implies the uniqueness 
of the resulting saturated chain in $\Pi_n^w$.  Note that a common feature of the two cases above 
is that after the quadratic exchange we have $\F(\tilde \c_{k+1})=\F(\c_{k+1})$. 
Hence applying a quadratic exchange at level $k$ does not affect the sequence
$$\F(\tilde \c_{k+1})\lessdot \F(\tilde \c_{k+2})\lessdot \cdots \lessdot \F(\tilde 
\c_{\rho(\Pi_n^w)}),$$
nor the labels $\lambda_C(\c,x_{i-1}\lessdot x_i)$ for $i=k+2,\dots,\rho(\Pi_n^w)$ since they 
depend only on this sequence. In particular we have $\F(\c)=\F(\tilde \c)=\F(U_i(\c))$. Also, 
the choice of the element $\tilde x_k$ in the two cases above only depends on $\F(\c_{k+1})$ and so 
it has to be the same for any other chain that coincides with $\c$ in the bottom $k+2$ elements. We 
conclude then that $\lambda_C$ satisfies the rank two switching property. 
\end{proof} 

We note here that whenever $\c \in C(\Pi_n^w)$, we have $\widetilde \F(\c)=(\F( \c_{0})\lessdot 
\F(\c_{1})\lessdot \cdots \lessdot \F(\c_{k}))$ in $C(\SF_n)$. Since the sequence of ordered 
pairs 
$(r(T_{i,1}),r(T_{i,2}))$ of roots that are being merged at each step $i$  (selecting $r(T_{i,1})$ 
as the new root) provides enough 
information to reconstruct the element of $C(\SF_n)$, by Lemma 
\ref{lemma:saturatedchainsinPIandSF}, it also provides enough information to reconstruct the 
elements in $ C(\Pi_n^w)$. Hence the sequences of labels given by $\lambda_C$
uniquely determine elements in $ C(\Pi_n^w)$. This also implies that the same is true for 
all maximal chains in rooted intervals $[x,y]_{\c}$, or otherwise we can extend two maximal chains 
$\c_1$ and $\c_2$ with the same word of labels to $\c \cup \c_1$ and $\c\cup \c_2$ with the same 
property. This together with Proposition \ref{proposition:weightedpartitionsCR}, Lemma 
\ref{lemma:ranktwoswitchingweightedpartitions} and Theorem \ref{theorem:CW} imply
the following corollary.
\begin{corollary}
 $\lambda_C$ is a CW-labeling and hence $Q_{\lambda_C}(\Pi_n^w)$ is a Whitney dual of $\Pi_n^w$
\end{corollary}

Note here that Lemma \ref{lemma:ranktwoswitchingweightedpartitions} also implies  that the map 
$\F:C(\Pi_n^w)\rightarrow \SF_n$ has the property that if $\c$ and $\c'$ are such that $[\c]=[\c']$ 
in $Q_{\lambda_C}(\Pi_n^w)$ then $\F(\c)=\F(\c')$. Hence we obtain a well-defined map 
$\overline \F: Q_{\lambda_C}(\Pi_n^w)\rightarrow \SF_n$ given by $\overline \F(X)=\F(\c)$ for any 
$\c \in X$. 

\begin{theorem}\label{theorem:QSF}
 The map $\overline \F: Q_{\lambda_C}(\Pi_n^w)\rightarrow \SF_n$ is a poset isomorphism.
\end{theorem}
\begin{proof}
To be able to prove that this is a poset isomorphism we should show that $\overline \F$ is a bijection and that 
$\overline \F$ and $\overline \F^{-1}$ are both poset (order preserving) maps.

{\bf $\overline \F$ is a bijection:} To prove this we will show that for 
any $F \in \SF_n$ there exist a unique ascent-free chain $\c \in C(\Pi_n^w)$ such that 
$\F(\c)=F$. The conclusion then follows from the fact that for every  $X \in 
Q_{\lambda_C}(\Pi_n^w)$ there is a unique ascent-free chain $\c \in C(\Pi_n^w)$ such that $X=[\c]$.

Recall from Lemma \ref{lemma:building_forests} that a sequence $\hat{0}=F_0\lessdot F_1 
\lessdot \cdots \lessdot F_k=F$ is a saturated chain in $\SF_n$ if and only if the associated 
ordered listing $v_1,v_2,\dots,v_k$ of the non-root vertices of $F$ is a linear 
extension.
Call $T(v_i)$ the induced rooted subtree of $F$ formed by all descendants of $v_i$ (including 
itself) and recall that $p(v_i)$  is the
the parent of $v_i$ with respect to $F$. 
Since
$v_1,v_2,\dots,v_k$ is a linear extension,  every $v_j \in T(v_i)$ satisfies $j<i$. This
implies that every saturated chain $\hat{0}=F_0\lessdot F_1 \lessdot \cdots \lessdot F_k=F$ has 
associated labels of the form $\lambda_{\SF}(F_{i-1}\lessdot F_i)=(-\gamma(T(v_i)),p(v_i),v_i)$. 
 Hence all saturated chains 
$\hat{0}=F_0\lessdot F_1 \lessdot \cdots \lessdot F_k=F$ in $[\hat{0}, F]$ for a given $F \in 
\SF_n$ have the same set of labels 
$$L(F):=\{(-\gamma(T(v_i)),p(v_i),v_i)\mid i\in [k]\}.$$ 
The labels in $L(F)$  are clearly all different and come from a totally ordered set, so there is a 
unique ascent-free way to order 
them $$(-\gamma(T(v_{i_1})),p(v_{i_1}),v_{i_1})>(-\gamma(T(v_{i_2})),p(v_{i_2}),v_{i_2})>\cdots > 
(-\gamma(T(v_{i_k})),p(v_{i_k}),v_{i_k}).$$ 
Let $v_{i_1},v_{i_2},\dots,v_{i_k}$ be the ordered listing that we obtain in this 
way.  We want to check that this sequence is also a linear extension which, by  Lemma 
\ref{lemma:building_forests}, implies that it has an associated saturated chain from 
$\hat{0}$ to $F$. Indeed, if this ordered listing is not a linear extension then there are vertices 
$v_{i_l}$ and $v_{i_j}$ such that $v_{i_l}$ is an ancestor of $v_{i_j}$ and $l<j$. But then we have 
that $(-\gamma(T(v_{i_l})),p(v_{i_l}),v_{i_l})<(-\gamma(T(v_{i_j})),p(v_{i_j}),v_{i_j})$ since 
$\gamma(T(v_{i_l}))>\gamma(T(v_{i_j}))$, a contradiction. 
Hence the sequence $v_{i_1},v_{i_2},\dots,v_{i_k}$ is a linear extension that gives a valid chain 
$\hat{0}=F_0\lessdot F_1 \lessdot \cdots \lessdot F_k=F$ in 
$\SF_n$  and so by Lemma \ref{lemma:saturatedchainsinPIandSF} $\c:(\hat{0}=\pi(F_0)\lessdot 
\pi(F_1) \lessdot \cdots \lessdot \pi(F_k)=\pi(F))$ 
is the unique ascent-free chain with $\F(\c)=F$. 

{\bf $\overline \F$ and $\overline \F^{-1}$ are poset maps:} The fact that $\F$ is order 
preserving implies, using Definition \ref{definition:quotientposet}, Proposition 
\ref{proposition:quotientposet} and the well-definedness of 
$\overline \F$,  that $\overline \F$ is also order preserving. Now, if we have $F\lessdot F'$ 
in 
$\SF_n$, 
given that $\F$ is surjective, there is a chain $\c$ such that $F=\F(\c)$ and the recursive 
definition of $\F$ implies that $F'=\F(\c \cup \pi(F'))$. But $\c \lessdot (\c \cup \pi(F'))$ in 
$C(\Pi_n^w)$ and so $\overline \F^{-1}(F)=[\c] \lessdot [(\c \cup \pi(F'))]=\overline
\F^{-1}(F')$ in $Q_{\lambda_C}(\Pi_n^w)$. Since all posets are finite this implies $\overline
\F^{-1}$ is order preserving.
\end{proof}

\begin{remark}
 As we mentioned at the beginning of this subsection, Theorem \ref{theorem:QSF} provides a new proof of Theorem 
\ref{theorem:weightedpartitions_rootedforests_whitneyduals} as a corollary.
\end{remark}

\subsubsection{A different Whitney dual for $\Pi_n^w$}
In \cite{DleonWachs2016}, Gonz\'alez D'Le\'on and Wachs gave an ER-labeling (that is in fact an 
EL-labeling) for $\Pi_n^w$, quite different from the CR-labeling constructed above. 

The map $\lambda_E:\E(\Pi_n^w)\rightarrow \Lambda_n$ was defined as follows: let $x\lessdot y $ in 
$\Pi^w_n$ so that $y$ is obtained from $x$ by merging two blocks $A^{w_A}$ and $B^{w_B}$ into a 
new block $(A \cup B)^{w_A + w_B+u}$, where $u \in \{0,1\}$ and where we assume without loss of 
generality that $\min A < \min B$.  We 
define $$\lambda_E(x \lessdot y) = (\min A, \min B)^u.$$
Here $\Lambda_n$ is defined as follows: for each $a \in [n]$, let $\Gamma_a:= \{(a,b)^u :   
a<b \le n+1, \,\, u \in \{0,1\} \}$. 
We partially order $\Gamma_a$  by letting $(a,b)^u \le  (a,c)^v$ if $b\le  c$ and $u \le v$.   
Note that $\Gamma_a$ is isomorphic to the direct product of the chain $a+1< a+2 <\dots < n+1 $ and
the chain $0 < 1$.  Now define $\Lambda_n$ to be the 
ordinal sum
$\Lambda_n := \Gamma_1 \oplus  \Gamma_2  \oplus \cdots \oplus \Gamma_{n}$. See Figure 
\ref{figure:weightedn3k2EL} for an example.

The labeling $\lambda_E$ has the property that when restricted to the intervals $[\hat 0, [n]^0]$ 
and $[\hat 0, [n]^{n-1}]$,  which are both isomorphic to $\Pi_n$, it reduces to the 
minimal labeling of $\Pi_n$ in \cite{Bjorner1982,Stanley1974}. 

\begin{figure}[h]

\begin{center} 
\begin{tikzpicture}
\node at (0,-1) {$\Pi_3^w$};
\node at (6,-1) {$\Lambda_3$};

\begin{scope}[xshift=0,scale=0.7]
\tikzstyle{every node}=[inner sep=0pt, scale=0.7, minimum width=4pt]
\node (n1232) at (3,4) {$123^{2}$};
  \node (n13020) at (-3,2) {$13^{ 0}/ 2^{ 0}$};
  \node (n102030) at (0,0)  {$1^{0}/ 2^{0}/ 3^{0}$};
  \node (n1231) at (0,4) {$123^{1}$};
  \node (n12030) at (-5,2) {$12^{0}/ 3^{0}$};
  \node (n13120) at (3,2)  {$13^{1}/ 2^{0}$};
  \node (n1230) at (-3,4){$123^ {0}$};
  \node (n10230) at (-1,2)  {$1^{0}/ 23^{0}$};
  \node (n12130) at (1,2)  {$12^{ 1}/ 3^{0}$};
  \node (n10231) at (5,2) {$23^ {1}/ 1^{0}$};

  \draw (n1231) -- (n10230) ;
  \draw [] (n13020) -- (n102030);
  \draw [] (n1232) -- (n13120);
  \draw [] (n1231)-- (n13020);
  \draw [] (n10230)--(n102030);
  \draw [] (n1230) -- (n10230);
  \draw [] (n1231) -- (n13120);
  \draw [] (n12030)-- (n102030);
  \draw [] (n1231) --(n12130);
  \draw [] (n1232) -- (n12130);
  \draw [] (n13120) --(n102030);
  \draw [] (n1231) --(n10231);
  \draw [] (n1230) -- (n13020);
  \draw [] (n1230)  -- (n12030);
  \draw [] (n12130) --  (n102030);
  \draw [] (n1232)  --  (n10231);
  \draw [] (n10231)  --  (n102030);
  \draw [] (n1231) -- (n12030);

\tikzstyle{every node}= [scale=0.6]

\node  at (-3.3,1) {\color{blue}$(1,2)^0$};
\node  at (-2,1) {\color{blue}$(1,3)^0$};
\node  at (-0.9,1) {\color{blue}$(2,3)^0$};
\node  at (0.9,1) {\color{blue}$(1,2)^1$};
\node  at (2,1) {\color{blue}$(1,3)^1$};
\node  at (3.3,1) {\color{blue}$(2,3)^1$};

\node  at (-4,3.5) {\color{blue}$(1,3)^0$};
\node  at (-3.3,3.3) {\color{blue}$(1,2)^0$};
\node  at (-2,3.5) {\color{blue}$(1,2)^0$};

\node  at (-3.9,2.6) {\color{blue}$(1,3)^1$};
\node  at (-2.4,2.6) {\color{blue}$(1,2)^1$};
\node  at (-1.1,2.6) {\color{blue}$(1,2)^1$};
\node  at (1.1,2.6) {\color{blue}$(1,3)^0$};
\node  at (2.4,2.6) {\color{blue}$(1,2)^0$};
\node  at (3.9,2.6) {\color{blue}$(1,2)^0$};

\node  at (4,3.5) {\color{blue}$(1,2)^1$};
\node  at (3.3,3.3) {\color{blue}$(1,2)^1$};
\node  at (2,3.5) {\color{blue}$(1,3)^1$};
\end{scope}
\begin{scope}[xshift=150,scale=0.4]
 \tikzstyle{every node}=[inner sep=1pt, minimum width=14pt,scale=0.7, font=\footnotesize]
\draw (0,0) node (n120) {\color{blue}$(1,2)^0$};
\draw (0,1) node (n130) {\color{blue}$(1,3)^0$};
\draw (0,2) node (n140) {\color{blue}$(1,4)^0$};
\draw (2,1) node (n121) {\color{blue}$(1,2)^1$};
\draw (2,2) node (n131) {\color{blue}$(1,3)^1$};
\draw (2,3) node (n141) {\color{blue}$(1,4)^1$};

\draw (n141) -- (n140) ;
\draw (n131) -- (n130) ;
\draw (n121) -- (n120) ;
\draw (n140)-- (n130) -- (n120) ;

\draw (2,4) node (n230) {\color{blue}$(2,3)^0$};
\draw (2,5) node (n240) {\color{blue}$(2,4)^0$};
\draw (4,5) node (n231) {\color{blue}$(2,3)^1$};
\draw (4,6) node (n241) {\color{blue}$(2,4)^1$};
\draw (4,7) node (n340) {\color{blue}$(3,4)^0$};
\draw (6,8) node (n341) {\color{blue}$(3,4)^1$};
\draw (n241) -- (n240) ;
\draw (n231) -- (n230) ;

\draw (n240)-- (n230)  -- (n141) -- (n131) -- (n121);

\draw (n341) -- (n340) ;

\draw (n340) -- (n241) -- (n231);

\end{scope}

\end{tikzpicture}
\end{center}
\caption[]{ER-labeling of  $\Pi_3^w$}\label{figure:weightedn3k2EL}
\end{figure}

\begin{theorem}[\cite{DleonWachs2016} Theorem 3.2] $\lambda_E$ is an ER-labeling.
 
\end{theorem}

\begin{theorem}\label{theorem:lambdaE_EW}
 $\lambda_E$ is an EW-labeling and hence $Q_{\lambda_E}(\Pi_n^w)$ is a Whitney dual of $\Pi_n^w$.
\end{theorem}
\begin{proof}Note that the information contained in the labels $(\min A, \min B)^u$ is enough to 
recover any saturated chain from $\hat{0}$. Hence the sequence of labels in each interval 
uniquely determines a chain. To show that $\lambda_E$ is an EW-labeling we are left to show that it 
satisfies the rank two switching property.
 The rank two intervals $[x,y]$ in $\Pi_n^w$ are of three possible  different types (see Figure 
\ref{figure:ranktwointervals_weighted}), we will show 
that in each of these types the rank two switching property is satisfied. 
  
\begin{enumerate}
\item[{\bf Type I}:] Two pairs of distinct blocks $\{A^{w_A},B^{w_B}\}$ and $\{C^{w_C},D^{w_D}\}$ 
of $x$ are merged to get $y$. Assume without loss of generality that $\min A < \min B$, $\min 
A < \min C$ and $\min C < \min D$.  The open interval
$(x,y)$ equals $\{z_1,z_2\}$ where $z_1$ is like $x$, but with a block $(A\cup B)^{w_A+w_B+u_1}$ 
instead of $\{A^{w_A},B^{w_B}\}$ and $z_2$ is like $x$, but with a block $(C\cup D)^{w_C+w_D+u_2}$ 
instead of $\{C^{w_C},D^{w_D}\}$, where $u_1,u_2 \in \{0,1\}$.  In this case this interval has two 
maximal chains with labels
$$\lambda_E(x\lessdot z_1)=(\min A, \min B)^{u_1}<(\min C,\min D)^{u_2}=\lambda_E(z_1\lessdot y), 
\text{ and}$$
$$\lambda_E(x\lessdot z_2)=(\min C, \min D)^{u_2}>(\min A,\min B)^{u_1}=\lambda_E(z_2\lessdot y)$$

\item[{\bf Type II}:] Three distinct blocks $\{A^{w_A},B^{w_B},C^{w_C}\}$ of $x$ are merged to get 
$y$, adding at each merging step the same weight $u \in \{0,1\}$.   Assume without loss of 
generality that $\min A < \min B<\min C$. The open interval
$(x,y)$ equals $\{z_1,z_2,z_3\}$, where each weighted partition  $z_i$ is obtained from $x$ by
merging two of the three blocks and adding $u$ to the total weight of the resulting block.   In 
this case the interval has a unique increasing chain with labels 
$$\lambda_E(x\lessdot z_1)=(\min A, \min B)^{u}<(\min A,\min C)^{u}=\lambda_E(z_1\lessdot y),$$
and there is a unique chain with labels
$$\lambda_E(x\lessdot z_2)=(\min A, \min C)^{u}>(\min A,\min B)^{u}=\lambda_E(z_2\lessdot y).$$

\item[{\bf Type III}:] Three distinct blocks $\{A^{w_A},B^{w_B},C^{w_C}\}$ of $x$ are merged to 
get $y$ in one step adding $0$ to the total weight and in the other step adding $1$. Assume without 
loss of generality that $\min A < \min B<\min C$. The open interval
$(x,y)$ equals $\{z_1,z_2,z_3,z_4,z_5,z_6\}$, where each weighted partition  $z_i$ is obtained from
$x$ by either merging two of the three blocks and adding either $0$ or $1$ to the total weight of 
the resulting block.    In 
this case the interval has a unique increasing chain with labels 
$$\lambda_E(x\lessdot z_1)=(\min A, \min B)^{0}<(\min A,\min C)^{1}=\lambda_E(z_1\lessdot y),$$
and there is a unique chain with labels
$$
    \lambda_E(x\lessdot z_2)=(\min A, \min C)^{1}>(\min A,\min B)^{0}=\lambda_E(z_2\lessdot y).
$$
\end{enumerate}
Thus, $\lambda_E$ has the rank two switching property.
\end{proof}

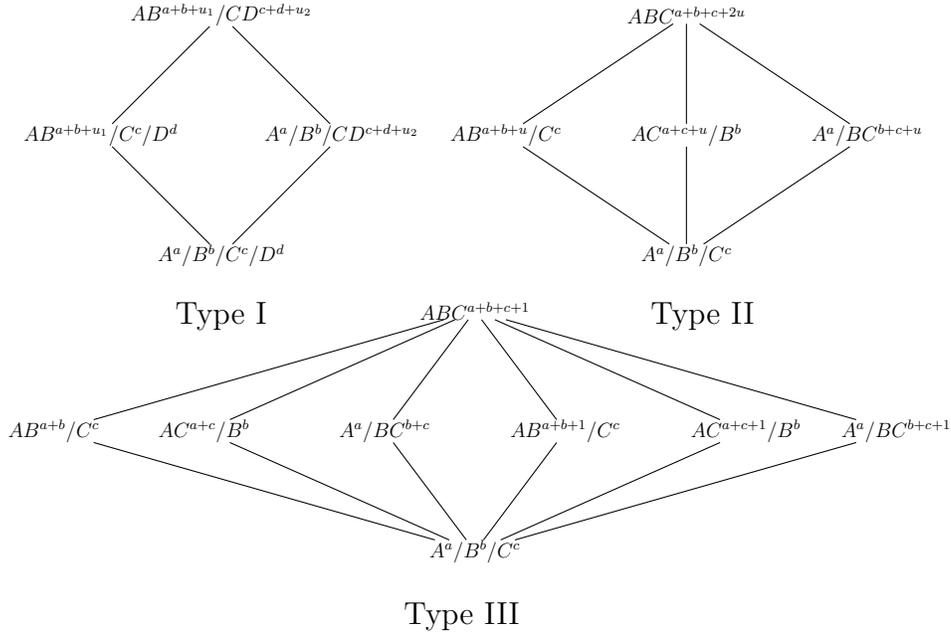
\begin{figure}
\centering
\begin{tikzpicture}[line join=bevel,scale=0.8]
  \node at (0,-1)  {Type I};
  \node at (8,-1)  {Type II};
  \node at (4,-6)  {Type III};

\begin{scope}
  \tikzstyle{every node}=[inner sep=0pt, scale=0.65, minimum width=4pt]
  \node (v1-2-3-4) at (0,0)  {$A^{a}/ B^{b}/ C^{c}/ D^{d}$};
  \node (v12-34) at (0,4)  {$AB^ {a+b+u_1}/CD^ {c+d+u_2}$};
  \node (v12-3-4) at (-2,2)  {$AB^{a+b+u_1}/ C^{c}/ D^{d}$};
  \node (v1-2-34) at (2,2)  {$A^{a}/ B^{b}/ CD^{c+d+u_2}$};
  \draw [] (v12-34) -- (v12-3-4);
  \draw [] (v12-34) -- (v1-2-34);
  \draw [] (v1-2-34) -- (v1-2-3-4); 
  \draw [] (v12-3-4) -- (v1-2-3-4);
 \end{scope}
\begin{scope}[xshift=220]
  \tikzstyle{every node}=[inner sep=0pt, scale=0.65, minimum width=4pt]
  \node (v1-2-3) at (0,0)  {$A^{a}/ B^{b}/ C^{c}$};
  \node (v123) at (0,4)  {$ABC^ {a+b+c+2u}$};
  \node (v12-3) at (-3,2)  {$AB^{a+b+u}/ C^{c}$};
  \node (v1-23) at (3,2)  {$A^{a}/ BC^{b+c+u} $};
 \node (v13-2) at (0,2)  {$AC^{a+c+u}/ B^{b}$};
  \draw [] (v123) -- (v12-3);
  \draw [] (v123) -- (v1-23);
 \draw [] (v123) -- (v13-2);
  \draw [] (v1-23) -- (v1-2-3); 
  \draw [] (v12-3) -- (v1-2-3);
  \draw [] (v13-2) -- (v1-2-3);
 \end{scope}

\begin{scope}[xshift=120,yshift=-140]
\tikzstyle{every node}=[ inner sep=0pt, scale=0.65, minimum width=4pt]
  \node (v1-2-3) at (0,0)  {$A^{a}/ B^{b}/ C^{c}$};
  \node (v123) at (0,4)  {$ABC^{a+b+c+1}$};
  \node (v12a-3) at (-7,2)  {$AB^{a+b}/ C^{c}$};
  \node (v13a-2) at (-4.5,2)  {$AC^{a+c}/ B^{b}$};
  \node (v1-23a) at (-1.5,2)  {$A^{a}/ BC^{b+c}$};   
  \node (v12b-3) at (1.5,2)  {$AB^{a+b+1}/ C^{c}$};
  \node (v13b-2) at (4.5,2)  {$AC^{a+c+1}/ B^{b}$};
  \node (v1-23b) at (7,2)  {$A^{a}/BC^ {b+c+1}$};

  \draw [] (v123)-- (v12a-3);
  \draw [] (v123) -- (v1-23a);
  \draw [] (v123)--(v12b-3);
  \draw [] (v123) --(v13a-2);
  \draw [] (v123) --(v13b-2);
  \draw [] (v123) --(v1-23b);
  \draw [] (v1-23a)-- (v1-2-3); 
  \draw [] (v12a-3) --(v1-2-3);
  \draw [] (v13a-2)--(v1-2-3);
  \draw [] (v1-23b) -- (v1-2-3);
  \draw [] (v12b-3) --(v1-2-3);
  \draw [] (v13b-2) --(v1-2-3);
\end{scope}

\end{tikzpicture}
\label{fig:type1}
\caption[]{Rank two intervals in $\Pi_n^w$}\label{figure:ranktwointervals_weighted}
\end{figure}

Now that we know, by Theorem \ref{theorem:lambdaE_EW}, that $\lambda_E$ is an EW-labeling of 
$\Pi_n^w$, we can  use Theorem \ref{theorem:secondcaracterizationQlambda} to describe 
$Q_{\lambda_E}(\Pi_n^w)$. We leave the general characterization of $Q_{\lambda_E}(\Pi_n^w)$ for a 
future article, but we explicitly compute the example of $Q_{\lambda_E}(\Pi_3^w)$ in Figure 
\ref{figure:second_whitney_dual_Pi3}. An interesting fact here is that $Q_{\lambda_E}(\Pi_3^w)$ 
and 
$Q_{\lambda_C}(\Pi_3^w)\simeq \SF_3$ are evidently not isomorphic (see Figures 
\ref{figure:spanning_forests_3} and \ref{figure:second_whitney_dual_Pi3}).

\begin{theorem}\label{theorem:Qlambdanonisomorphic}
 There exist a poset $P$ and two CW-labelings $\lambda_1$ and $\lambda_2$ of $P$ such that the 
posets $Q_{\lambda_1}(P)$ and $Q_{\lambda_2}(P)$ are not isomorphic. 
\end{theorem}

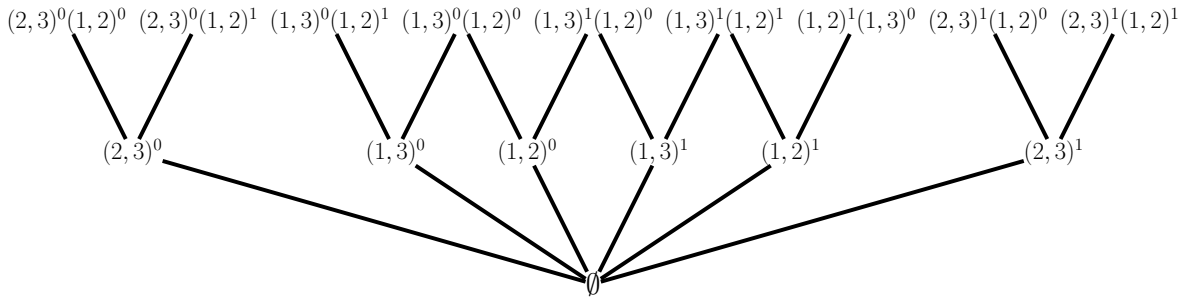
\begin{figure}[h]
\begin{center} 
\begin{tikzpicture}[scale=0.35]


\tikzstyle{every node}=[inner sep=0pt,minimum size=10,scale=0.45]
\node (p1) at (13,0.25) {\Huge$\emptyset$};
\node (q1) at (-4.5,5.25) {\LARGE$(2,3)^0$};
\node (q2) at (5.5,5.25) {\LARGE$(1,3)^0$};
\node (q3) at (10.5,5.25) {\LARGE$(1,2)^0$};
\node (q4) at (15.5,5.25) {\LARGE$(1,3)^1$};
\node (q5) at (20.5,5.25) {\LARGE$(1,2)^1$};
\node (q6) at (30.5,5.25) {\LARGE$(2,3)^1$};

\node (r1) at (-7,10.25) {\LARGE$(2,3)^0(1,2)^0$};
\node (r2) at (-2,10.25) {\LARGE$(2,3)^0(1,2)^1$};
\node (r3) at (3,10.25) {\LARGE$(1,3)^0(1,2)^1$};
\node (r4) at (8,10.25) {\LARGE$(1,3)^0(1,2)^0$};
\node (r5) at (13,10.25) {\LARGE$(1,3)^1(1,2)^0$};
\node (r6) at (18,10.25) {\LARGE$(1,3)^1(1,2)^1$};
\node (r7) at (23,10.25) {\LARGE$(1,2)^1(1,3)^0$};
\node (r8) at (28,10.25) {\LARGE$(2,3)^1(1,2)^0$};
\node (r9) at (33,10.25) {\LARGE$(2,3)^1(1,2)^1$};

\tikzstyle{every path}=[line width=1.5pt]
 \draw (p1) -- (q1) ;
 \draw (p1) -- (q2) ;
 \draw (p1) -- (q3) ;
 \draw (p1) -- (q4) ;
 \draw (p1) -- (q5) ;
 \draw (p1) -- (q6) ;

 \draw (q1) -- (r1) ;
 \draw (q1) -- (r2) ;
 \draw (q2) -- (r3) ;
 \draw (q2) -- (r4) ;
 \draw (q3) -- (r4) ;
 \draw (q3) -- (r5) ;
 \draw (q4) -- (r5) ;
 \draw (q4) -- (r6) ;
 \draw (q5) -- (r6) ;
 \draw (q5) -- (r7) ;
 \draw (q6) -- (r8) ;
 \draw (q6) -- (r9) ;
\end{tikzpicture}
\end{center}
\caption[]{$Q_{\lambda_E}(\Pi_3^w)$}\label{figure:second_whitney_dual_Pi3}
\end{figure}

\subsection{R$^*$S-labelable posets}  In~\cite{SimionStanley1999} Simion and Stanley introduced the notion 
of an R$^*$S-labeling as a tool to study local actions of the symmetric group on maximal chains of 
a poset. In this subsection, we show that the R$^*$S labelings that respect the consistency 
condition of the rank two switching property are CW-labelings.   Simion and Stanley~\cite{SimionStanley1999} showed the posets of shuffles $W_{M,N}$ introduced by 
Greene~\cite{Greene1988} 
have R$^*$S labelings and Hersh~\cite{Hersh1999} showed that the noncrossing partition lattices 
$NC_n^B$ and $NC_n^D$ of 
types B and D introduced by Reiner~\cite{Reiner1997} also have such labelings.   As noted 
in~\cite{SimionStanley1999},  Stanley's parking function labeling of the noncrossing partition 
lattice $NC_n$ of type $A$ described in a previous subsection  is an R$^*$S labeling. All of these 
are examples of  R$^*$S labelings that are CW-labelings, hence we have as a corollary that these 
posets have Whitney duals.

We start with the definition of an S-labeling of a poset given in~\cite{SimionStanley1999}.  We 
present it here in a slightly different language than how was originally stated in 
\cite{SimionStanley1999} to highlight the connection with 
CW-labelings.

\begin{definition} [\cite{SimionStanley1999}]\label{rStarDef}
Let $P$ be a graded poset  of rank $n$ with a $\hat{0}$ and  $\hat{1}$. Let $\lambda$ be a 
C-labeling such that the labels are totally ordered. We say $\lambda$ is an S-labeling if  
\begin{enumerate}
\item \label{condition:S1} For each  maximal chain $\c =  (\hat{0}=x_0\cover x_1 \cover \cdots 
\cover x_n =\hat{1})$  
and for each $1\leq i \leq n-1$ such that $\lambda(\c, x_{i-1} \cover x_i) \neq  \lambda(\c, x_{i} 
\cover x_{i+1})$, there exists a unique maximal chain $\c' =   (\hat{0}=x_0\cover x_1 \cover \cdots 
\cover x_{i-1}\cover x'_{i}\cover x_{i+1} \cover \cdots  \cover x_n =\hat{1})$ such that $\c$ and 
$\c'$ have the same sequence of labels except that  $\lambda(\c, x_{i-1} \cover x_i )  = 
\lambda(\c', x'_{i} \cover x_{i+1})$ and $\lambda(\c, x_{i} \cover x_{i+1} )  = \lambda(\c', 
x_{i-1} 
\cover x'_{i})$
\item \label{condition:S2} $\lambda$  is one-to-one on maximal chains.  That is,  two different 
maximal chains must have 
different sequences of labels from  bottom to top.
\end{enumerate}
\end{definition}

Considering the condition \eqref{condition:S1} of an S-labeling in Definition~\ref{rStarDef}, one 
can see that 
if the choice of the element $x'_{i}$ is consistent among all maximal chains that coincide in 
the first $i+2$ edges then the S-labeling also satisfies the rank two 
switching property of Definition \ref{definition:ranktwoC}. We will call S-labelings satisfying 
this additional condition \emph{consistent}. In that case, we note that condition 
\eqref{condition:S1} on S-labelings is in fact stronger than the rank two switching property.  For 
example, the labeling  of $\Pi_3$ given in 
Figure~\ref{figure:exampleER}  has the rank two switching property,  but it is not an S-labeling 
since
the chain with label sequence $(2,3), (1,3)$ cannot be switched with anything as it 
should be in the case of an S-labeling.  Also, note that condition \eqref{condition:S2} of 
Definition \ref{rStarDef} is also stronger than the simpler requirement of Definition 
\ref{definition:EW} that ascent-free chains are one-to-one. Hence any S-labeling satisfying the 
consistency condition and which is also a CR-labeling, is a CW-labeling.

In~\cite{SimionStanley1999} Simion and Stanley refer to C-labelings that are also  CR-labelings  as 
R$^*$-labelings. Moreover, any labeling which is both an R$^*$-labeling and an S-labeling is 
called an R$^*$S-labeling. 
We warn the reader of a possible source of confusion since in this paper 
we have used the term ER$^*$ to mean a different type of labeling.   Using 
Theorem~\ref{theorem:CW}, we have the following.

\begin{theorem}
A consistent R$^*$S-labeling is a CW-labeling.  Consequently, every poset with a consistent 
R$^*$S-labeling has a Whitney dual.
\end{theorem}

\begin{remark}
 The consistency condition is automatically satisfied for E-labelings and so it only needs to be 
checked when 
the underlying labeling is a C-labeling.
\end{remark}

The reader can check in \cite{SimionStanley1999} that the labeling for the poset of shuffles 
$W_{M,N}$ is an example of a consistent R$^*$S-labeling and hence a CW-labeling.
In \cite{Stanley1997} and \cite{Hersh1999} edge labelings of the noncrossing 
partition lattices of type A, B and D are given. These are all examples of R$^*$S-labelings and hence they are also 
EW-labelings. We then have the following corollary.

\begin{corollary}
The poset of shuffles $W_{M,N}$ and the noncrossing partition lattices of type A, B and D all have 
Whitney duals.
\end{corollary}

\section{$H_n(0)$-actions and Whitney labelings}\label{section:heckeaction}
In this section we describe an action of the $0$-Hecke algebra on the maximal chains of 
a poset $P$ with a generalized CW-labeling $\lambda$.  We will also see that the same action can be 
associated to the Whitney dual $Q_\lambda(P)$ constructed in Section 
\ref{section:whitneylabelings}.  The characteristic of this action is Ehrenborg's flag 
quasisymmetric function in the case of $P$ and is Ehrenborg's flag quasisymmetric function with 
$\omega$ applied in the case of $Q_\lambda(P)$.  The techniques we 
describe here   closely  follow the work of McNamara in \cite{McNamara2003} who studied 
actions of this kind on posets with EL-labelings in which the word of labels in every chain is 
a permutation of $\sym_n$, also known as $\sym_n$ EL-shellable  or \emph{snellable} posets.

\subsection{An action of the $0$-Hecke algebra}
Suppose that $P$ is a graded poset of rank $n$.  Moreover, suppose that  $\lambda$ is a generalized
CW-labeling of $P$.   Recall that  $\mathcal{M}_P$ is the set of maximal chains of $P$.  
Define maps $U_1,U_2,\dots,U_{n-1}: \mathcal{M}_P\rightarrow \mathcal{M}_P$ such that for  
$\c:(\hat{0}=x_0\cover x_1\cover\cdots \cover x_n)$

$$
U_i (\c)=
\begin{cases}
\c' & \mbox{if $\lambda(\c,x_{i-1}\cover x_i)<\lambda(\c, x_{i}\cover x_{i+1})$,}\\
\c & \mbox{otherwise,}
\end{cases}
$$
where $\c'$ the unique maximal chain of $P$ obtained by applying a quadratic exchange at rank $i$. As an example, consider the maximal chain $\c:(1/2/3/4\cover 13/2/4\cover 123/4\cover 1234)$ in $\NC_4$ with the parking function labeling (see Figure~\ref{fig:NC4}).  Since there is no ascent at rank $1$, $U_1(\c) =\c$.  However, there is an ascent at rank $2$, and $U_2(\c)= 1/2/3/4\cover 13/2/4\cover 134/2\cover 1234$.

We 
note that in \cite{McNamara2003}, where the labelings are snellings, 
the maps $U_i$ are similar except that instead of exchanging ascents by descents, they 
exchange descents by ascents.

\begin{proposition}\label{proposition:heckerelations}
The maps $U_1,U_2,\dots, U_{n-1}$ have the following properties.
\begin{enumerate}
\item For all $\c\in \mathcal{M}_P$, $U_i(\c)$ and $\c$ are the same 
except possibly at rank $i$. 
\item $U_i^2 = U_i$ for all $i$.
\item $U_iU_j = U_jU_i $ for all $i,j$ such that $|i-j|>1$.
\item  $U_iU_{i+1}U_i = U_{i+1}U_iU_{i+1}$ for all 
$i$.
\end{enumerate}
\end{proposition}
\begin{proof}
 The first three properties are immediate from the definition of $U_i$ and Definition 
\ref{definition:ranktwoC} of the rank two switching property for C-labelings. 
The 
last property is a consequence of the braid relation in a generalized CW-labeling when there is a 
critical condition at rank $i$ and is easily verified when there is no critical condition at rank 
$i$.
\end{proof}

The $0$-Hecke algebra of type $A$ is defined by abstract generators satisfying the same relations of those in Proposition \ref{proposition:heckerelations}. Thus the properties described in the proposition imply that there is an action of the generators of the $0$-Hecke algebra $H_n(0)$ on the set $\mathcal{M}_P$.  This action is said to be \emph{local} since the chains $U_i(\c)$  and $\c$ are the same except possibly at rank $i$.  Moreover, this action gives rise to a representation of the $0$-Hecke algebra on the space $\mathbb{C}\mathcal{M}_P$ linearly spanned by $\mathcal{M}_P$.

It turns out that the characteristic of this action is a well-known quasisymmetric function.  
Before 
we look at this characteristic, we need to review some material on quasisymmetric functions.

\subsection{Ehrenborg's flag quasisymmetric function}
In~\cite{Ehrenborg1996}, Ehrenborg introduced a formal power series now known as \emph{Ehrenborg's 
flag 
quasisymmetric function}.  Given a graded poset $P$ with a $\hat{0}$ and $\hat{1}$, it is defined by
$$
F_P(x_1,x_2,\dots ) = F_P(\xx) = \sum_{\hat{0}=t_1\leq t_2\leq \cdots \leq t_{k-1}< t_k = \hat{1}} 
x_{1}^{rk(t_0,t_1)}x_{2}^{rk(t_1,t_2)}\cdots x_{k}^{rk(t_{k-1},t_{k})}
$$
where the sum is over all multichains from $\hat{0}$ to $\hat{1}$ where $\hat{1}$ appears exactly 
once.  As the name suggests, $F_P(\xx)$ belongs to the ring of quasisymmetric functions.  That is, 
for each sequence $n_1,n_2,\dots, n_k$ the monomial $x_{i_1}^{n_1}x_{i_2}^{n_2}\cdots 
x_{i_k}^{n_k}$ has the same coefficient as 
$x_{j_1}^{n_1}x_{j_2}^{n_2}\cdots x_{j_k}^{n_k}$ whenever $i_1<i_2<\cdots< i_k$ and $j_1<j_2<\cdots 
<j_k$.   In addition to being a quasisymmetric function, $F_P(\xx)$ also keeps track of the flag 
$f$-vector and the flag $h$-vector of $P$ as we describe next.

Let $P$ be a graded poset with a $\hat{0}$ and $\hat{1}$.  For $S\subseteq [n-1]$ define
$$
\alpha_P(S) = |\{\hat{0}<x_1<x_2<\cdots <x_{|S|}<\hat{1}\mid \{\rho(x_1),\rho(x_2),\dots , 
\rho(x_{|S|})\} = S\}|.
$$
In other words, $\alpha_P(S)$ is the number of chains from $\hat{0}$ to $\hat{1}$ which 
use elements whose rank set is $S$. The function given by $\alpha_P: 2^{[n-1]} \rightarrow 
\mathbb{Z}$ 
is called the \emph{flag f-vector} of $P$.  We also define
$$
\beta_P(S) = \sum_{T\subseteq S} (-1)^{|S\setminus T|} \alpha_P(T).
$$
The function $\beta_P: 2^{[n-1]}\rightarrow \mathbb{Z}$ is called the \emph{flag h-vector} of $P$.  
The reason for the names flag $f$-vector and flag $h$-vector is that they refine the classical 
$f$-vector and $h$-vector of the order complex of $P$.  See~\cite{Stanley2012}[\S 3.13] for 
more details.

When $P$ has a $\hat{0}$ and $\hat{1}$, there is a nice relationship between $F_P(\xx)$ and 
$\beta_P(S)$.  Indeed, it is well-known that if $P$ has rank $n$, then
$$
F_P(\xx)  =  \sum_{S\subseteq [n-1]}  \beta_P(S) L_{S,n}(\xx)
$$
where $L_{S,n}$ is Gessel's fundamental quasisymmetric function defined by
$$
L_{S,n}(\xx) = \sum_{\substack{1\leq i_1\leq i_2\leq \cdots  \leq i_n\\ i_j<i_{j+1} \mbox{ if } j\in 
S} 
} x_{i_1}x_{i_2}\cdots x_{i_n}.
$$

The original definition of $F_P(\xx)$ requires that $P$ have a $\hat{1}$, however we would like to extend this 
to more general posets.   In order to do this, we consider a slight generalization of $F_P(\xx)$ to 
deal with posets with a single minimal element $\hat{0}$, but possibly with multiple maximal elements.  Let $P$ be 
a graded poset with a $\hat{0}$, then we 
define
$$
F_P(\xx) =  \sum_{m}F_{[\hat{0},m]}(\xx)
$$
where the sum is over all maximal elements $m$ of $P$. Note that in the case that $P$ has a 
$\hat{1}$, this is just Ehrenborg's classical 
definition.  
Since  intervals always have a $\hat{0}$ and 
a 
$\hat{1}$, we have that 
$$
F_P(\xx) = \sum_{m} \left( \sum_{S\subseteq [n-1]} \beta_{[\hat{0},m]} (S) L_{S,n} \right).
$$

Now suppose that $\lambda$ is a CR-labeling of $P$ and  that $P$ has a $\hat{0}$ and a $\hat{1}$.  
Recall that $\mathcal{M}_P$ denotes the set of maximal chains in $P$.  For $\c: (x_0\cover  
x_1\cover 
x_2\cover \cdots \cover x_n)\in\mathcal{M}_P$,  the \emph{descent set} of $\c$ is defined 
to be
$$
D(\c) = \{i \mid \lambda(\c, x_{i-1},x_i) \nless \lambda(\c, x_i,x_{i+1})\}.
$$
It was shown by Stanley~\cite{Stanley2012}[c.f. Theorem 3.14.2] for ER-labelings and of Bj\"orner 
and Wachs \cite{BjornerWachs1997} for CR-labelings that $\beta_S(P)$ is  the number of 
maximal chains with 
descent set  $S$. 
A simple modification of Stanley's proof of the combinatorial description of the numbers $\beta_P(S)$, shows that if $\lambda$ is an ER$^*$-labeling of a poset $P$ with $\hat{0}$ and $\hat{1}$, then $\beta_{P}(S) =\{\c \in 
\mathcal{M}_P\mid D(\c) =S^c\}$ for any $S\subseteq [n-1]$.

\begin{example}
We compute $F_P(\xx)$ for $P=\Pi_3$. As one  can see in 
Figure~\ref{figure:exampleER}, 
$\Pi_3$ 
has three maximal chains.  Under the ER-labeling of Example \ref{example:labelingposetofpartitions}, one of the maximal chains is increasing  and the other two are ascent-free.   It 
follows that $\beta_{\Pi_3}(\emptyset) = 1$ and $\beta_{\Pi_3}(\{1\}) = 2$.  Thus, 
$$
F_{\Pi_3}(\xx) = L_{\emptyset, 2}(\xx)+2L_{\{1\},2}(\xx).
$$
\end{example}

\begin{example}
We consider the example of $F_P(\xx)$ when $P = \ISF_3$.  As  one can see in 
Figure~\ref{figure:exampleER}, 
there are two maximal intervals with three maximal chains altogether.  Moreover, using the ER$^*$-labeling of Example \ref{example:labelingISF} we can compute  $\beta_{[\hat{0},m]}(S)$ with the number of maximal chains with strict ascent set given by $S$.

If $F_1$ 
is the increasing spanning forest with edge set $\{(1,2), (1,3)\}$ and $F_2$ is the one with  
$\{(1,2), (2,3)\}$, we see that 
$$
\beta_{[\hat{0},F_1]}(\emptyset) = 1,  \beta_{[\hat{0},F_1]}(\{1\}) 
=1,\beta_{[\hat{0},F_2]}(\emptyset) =1, \mbox{and } 
\beta_{[\hat{0},F_2]}(\{1\}) =0.
$$
Therefore
$$
F_{\ISF_3}(\xx)  = 2 L_{\emptyset, 2}(\xx)  +L_{\{1\},2}(\xx).
$$
\end{example}

\begin{example}
We compute $F_P(\xx)$ for $P=\NC_4$. Recall from Section \ref{section:noncrossingpartitiondual} that $\NC_4$ has an ER-labeling where the labels on the set $\M_{\NC_4}$ of maximal chains of $\NC_4$ correspond to parking functions of length $3$, see Figure~\ref{fig:NC4}. So $\NC_4$ 
has $16$ maximal chains with label words given by $(1,1,1)$, the three permutations of each $(1,1,2)$ $(1,1,3)$ and $(1,2,2)$; and the six permutations of $(1,2,3)$.  Considering the descent sets of each of these sequences we can compute that $\beta_{\NC_4}(\emptyset) = 1$,  $\beta_{\NC_4}(\{1\}) = 5$,  $\beta_{\NC_4}(\{2\}) = 5$ and $\beta_{\NC_4}(\{1,2\}) = 5$.  Thus, 
$$
F_{\NC_4}(\xx) = L_{\emptyset, 3}(\xx)+5L_{\{1\},3}(\xx)+5L_{\{2\},3}(\xx)+5L_{\{1,2\},3}(\xx).
$$
The quasisymmetric function $F_{\NC_n}(\xx)$ is in fact symmetric. Stanley~\cite{Stanley1997} showed that $\omega(F_{\NC_n}(\xx))$ is \emph{Haiman's Parking Function Symmetric Function} of $n$, where $\omega$ is the involution on the 
ring of quasisymmetric functions given by  $\omega(L_{S,n}) = L_{S^c, n}$ where $S^c$ is the 
complement 
of $S$ in $[n-1]$.

\end{example}

\begin{example}
Now consider $F_P(\xx)$ when $P=\NCDyck_4\cong Q_{\lambda}(\NC_4)$ together with its inherited  ER$^*$-labeling from $\NC_4$ in Section \ref{section:noncrossingpartitiondual}. Using Proposition  \ref{proposition:bijectionbetweensaturatedchains} we know that the maximal chains are in a label-preserving bijective correspondence with the ones of $\NC_4$, so they are labeled by parking functions as well.  One can show then that $F_{\NCDyck_4}(\xx) = 5L_{\emptyset, 3}(\xx) + 5L_{\{1\},3}(\xx)+5L_{\{2\},3}(\xx)+ L_{\{1,2\},3}(\xx)$.  
\end{example}

The reader may have noticed that the quasisymmetric functions above are very closely 
related.   Our examples show that 
$$
F_{\Pi_3}(\xx) = \omega (F_{\ISF_3}(\xx)) \text{ and } F_{\NC_4}(\xx) = \omega(F_{\NCDyck_4}(\xx)).
$$

This is no coincidence as we now see.    
\begin{theorem}\label{FandOmegaF}
Let $\lambda$ be a generalized CW-labeling of $P$. Then
$$ F_{Q_{\lambda}(P)}(\xx)=\omega(F_P(\xx)).$$
\end{theorem}

\begin{proof}
Recall that since $\lambda$ is a CR-labeling, $\beta_P(S)$ counts the number 
of maximal chains with descent set $S$.  Similarly since $\lambda^*$ is an ER$^*$-labeling of 
$Q_{\lambda}(P)$, $\beta_{Q_{\lambda}(P)}(S)$ is the number of maximal chains with ascent set 
$A(\c) 
:= \{i \mid \lambda(\c, x_{i-1},x_i) < \lambda(\c, x_i,x_{i+1})\}=S$.  The 
CW-analogue of Proposition~\ref{proposition:bijectionbetweensaturatedchains} implies that there is 
a bijection between maximal chains in $P$ and $Q_{\lambda}(P)$ which preserves labels.  It follows 
that for each $S\subseteq [n-1]$,
$$
\sum_{m\in P} \beta_{[\hat{0},m]}(S^c)= \sum_{m' \in Q_{\lambda}(P)}\beta_{[\hat{0},m']}(S)
$$
where each sum is over maximal elements of $P$ and $Q_{\lambda}(P)$ respectively.
Therefore
\begin{align*}
 \omega F_{Q_{\lambda}(P)}(\xx) &=  \omega\left( \sum_{m'\in Q_{\lambda}(P)} F_{[\hat{0},m']}(\xx) 
\right)\\
  &=  \omega\left(  \sum_{m'\in Q_{\lambda}(P)} \left(\sum_{S\subseteq [n-1]} 
\beta_{[\hat{0},m']}(S) L_{S,n}\right) \right)\\
  &=  \omega  \left(\sum_{S\subseteq [n-1]} \left(  \sum_{m'\in 
Q_{\lambda}(P)}\beta_{[\hat{0},m']}(S)\right) L_{S,n} \right)\\
&=  \omega\left( \sum_{S\subseteq [n-1]}\left(  \sum_{m\in P}\beta_{[\hat{0},m]}(S^c)\right) 
L_{S,n}\right)\\
&=  \sum_{S\subseteq  [n-1]} \left(  \sum_{m\in P}\beta_{[\hat{0},m]}(S^c) \right) \omega(L_{S,n})\\
&=\sum_{S\subseteq [n-1]} \left(  \sum_{m\in P}\beta_{[\hat{0},m]}(S^c) \right) L_{S^c,n}\\
&=\sum_{m\in P} \left(\sum_{S\subseteq [n-1]} \beta_{[\hat{0},m]}(S) L_{S,n}\right)\\
&= F_P(\xx)
\end{align*}
Thus we have proved the desired result. 
\end{proof}

\subsection{The characteristic of the action}

In~\cite{Norton1979}, Norton investigated  the representation theory of $H_n(0)$.  It is known that 
there are $2^{n-1}$ irreducible representations, all of them one-dimensional and hence they can be 
indexed by subsets of $[n-1]$.  With this indexing,  we have that if $U_i$ is one of the generators 
of $H_n(0)$ then the representation $\psi_S$ is given by
$$
\psi_S(U_i)=
\begin{cases}
1& \mbox{ if $i\in S,$}\\
0 & \mbox{otherwise.}
\end{cases}
$$
Hence, the character of the action is given by
$$
\chi_S(U_{i_1}U_{i_2}\cdots U_{i_k})=
\begin{cases}
1& \mbox{ if $i_1, i_2,\dots, i_k\in S,$}\\
0 & \mbox{otherwise.}
\end{cases}
$$
The \emph{(quasisymmetric) characteristic} of the character $\chi_S$ is defined by
$$
ch(\chi_S) = L_{S,n}
$$
where, as before, $L_{S,n}$ is Gessel's fundamental quasisymmetric function.  We will 
use $\chi_P$ to denote the character of the defining representation of a $H_n(0)$-action on $P$.

\begin{theorem}\label{CharacteristicOfAction}
Let $P$ be a graded poset of rank $n$ with a generalized CW-labeling.  The local $H_n(0)$-action 
previously described is such that 
$$
ch(\chi_P)=F_P(\xx).
$$ 
\end{theorem}
We note that the proof we present is almost identical to the one in~\cite[Proposition 4.1]{McNamara2003}.
\begin{proof}
Let $[L_{S,n}] f(\xx)$ denote the coefficient of $L_{S,n}$ in the expansion of the quasisymmetric 
function $f(\xx)$ in the fundamental basis.  We will show for any subset $S$ 
of $[n-1]$,
$$
 [L_{S,n}]ch(\chi_P)=[L_{S,n}]F_p(\xx). 
$$
 As we saw earlier, the coefficient in $F_P(\xx)$ is 
$$
\sum_{m} \beta_{[\hat{0}, m]}(S).
$$
    Thus, it suffices to show
$$
 [L_{S,n}]ch(\chi_P)=\sum_{m} \beta_{[\hat{0}, m]}(S).
$$

Now let $J\subseteq [n-1]$ and let $\{i_1,i_2,\dots, i_k\}$ be a multiset of $J$ where each element 
of $J$ appears at least once.  For $\c\in \mathcal{M}_P$, if $U_i(\c) \neq 
\c$, then $\c$ has an ascent at $i$.   It follows that $U_{i_1}U_{i_2}\cdots 
U_{i_k}(\c)=\c$ if and only if $\c$ has descent set containing $J$. 
Therefore,
\begin{align*}
\chi_P(U_{i_1}U_{i_2}\cdots U_{i_k})&= \#\{\c\in \mathcal{M}_P \mid D(\c)\supseteq J\}\\
&= \sum_{S\supseteq J} \#\{\c\in \mathcal{M}_P \mid D(\c)=S\}\\\
&= \sum_{m} \sum_{S\supseteq J} \#\{\c\in \mathcal{M}_{[\hat{0},m]} \mid D(\c)=S\}\\\
&= \sum_{m} \left(\sum_{S\supseteq J} \beta_{[\hat{0},m]}(S) \right)\\
&= \sum_{S\subseteq [n-1]} \left(\sum_{m} \beta_{[\hat{0},m]}(S) \right)\chi_S(U_{i_1}U_{i_2}\cdots 
U_{i_k})
\end{align*}

It follows that 
$$
[L_{S,n}] ch(\chi_P) = [L_{S,n}]  ch\left(  \sum_{S\subseteq [n-1]} \left(\sum_{m} 
\beta_{[\hat{0},m]}(S) \right)\chi_S\right) = \sum_{m} \beta_{[\hat{0},m]}(S)
$$
which completes the proof.
\end{proof}

The analogue of Proposition~\ref{proposition:bijectionbetweensaturatedchains} for CW-labelings 
implies that there is a bijection between maximal chains of $P$ and $Q_{\lambda(P)}$  which 
preserves labels.  It follows that the local $H_n(0)$-action on 
$\mathbb{C}\mathcal{M}(P)$ can be also transported to a $H_n(0)$-action on 
$\mathbb{C}\mathcal{M}(Q_{\lambda})$.  It turns out that this action on the maximal chains of $Q_\lambda(P)$ is local.

\begin{lemma}
Let $P$ be a graded poset with a generalized CW-labeling $\lambda$.  The $0$-Hecke algebra action 
on $Q_\lambda(P)$ is local.
\end{lemma}
\begin{proof}
We must show that if we apply $U_i$  to any maximal chain $\d$ of $Q_\lambda(P)$, the chain we get 
agrees with $\d$ everywhere except possible at rank $i$.   Suppose that $\d$ and  $\d'
$ are maximal chains in $Q_\lambda(P)$ such that $U_i(\d) = \d'$ and $\d\neq \d'$.  Let $\c$ and 
$\c'$ be respectively the preimages of these chains under the label preserving bijection between 
$\M_{Q_\lambda(P)}$ and $\M_P$ described in the generalized CW-labeling version of 
Proposition~\ref{proposition:bijectionbetweensaturatedchains}.  Then $U_i(\c) =\c'$ and $\c \neq 
\c'$  Since the action on $P$ is local, we can write $\c: (\hat{0}=x_0\cover x_1\cover \cdots 
x_{i-1}\cover x_i \cover x_{i+1} \cover \cdots \cover x_n)$ and $\c': (\hat{0}=x_0\cover x_1\cover 
\cdots x_{i-1}\cover x'_i \cover x_{i+1} \cover \cdots \cover x_n)$.  Denote $\c_k$ the 
subchain formed by the smallest $k+1$ elements of $\c$ and $\c'_k$ the one formed by the smallest 
$k+1$ elements of $\c'$.  We have that $\c_j = \c'_j$ for all $0\leq j\leq i-1$ and $\c_i \neq 
\c'_i$, but $\c_j $ and $\c'_j$ are equivalent for $j\geq i+1$ (since one chain is obtained from 
the other after applying a quadratic exchange at rank $i$).  Thus the chains 
$[\c_0]\cover [\c_1]\cover \cdots \cover [\c]$ and $[\c'_0]\cover [\c'_1]\cover \cdots \cover [\c']$ 
 in $Q_\lambda(P)$ agree everywhere except at rank $i$.  Moreover, by the (generalized 
CW-labeling versions of) Lemma \ref{lemma:main} and the proof of 
Proposition~\ref{proposition:bijectionbetweensaturatedchains} these chains are exactly  $\d$ and 
$\d'$.  We conclude that the action on $Q_\lambda(P)$ is local.
\end{proof}

\begin{proposition}\label{charOfQ}
Let $P$ be a poset with a generalized CW-labeling $\lambda$.  For any maximal interval $I$ in 
$Q_\lambda(P)$, 
$$
ch(\chi_I) = \omega(F_I(\xx))
$$
\end{proposition}
\begin{proof}
First note that since the action on $Q_\lambda(P)$ is local,  the action only permutes maximal 
chains within maximal intervals of $Q_\lambda(P)$.   Also, just as with $P$,  $U_i(\c) \neq \c'$ if 
and only if $\c$ has ascent at rank level $i$.    Finally,  note that since the labeling on 
$Q_\lambda(P)$ (and hence on $I$) is an ER$^*$-labeling, we have that $\beta_{I}(S) =\{\c \in 
\mathcal{M}_I\mid D(\c) =S^c\}$ for any $S\subseteq [n-1]$.   With this in mind, one can check that 
a slight modification of the proof of Theorem~\ref{CharacteristicOfAction} gives the result.
\end{proof}

\begin{remark}
Note that $F_{Q_{\lambda}(P)}(\xx)=\sum_{I}(F_I(\xx))$ and that $\chi_{P}=\chi_{Q_{\lambda}(P)}=\sum_{I} \chi_I$, where the sums are over maximal intervals $I$ of $Q_{\lambda}(P)$. Hence we obtain  Theorem 
\ref{theorem:goodactions} as a corollary of Theorems \ref{FandOmegaF}  and \ref{CharacteristicOfAction}; and Proposition \ref{charOfQ}.
\end{remark}

A poset $P$  is called \emph{bowtie-free} if  there does not exist distinct $a,b,c,d\in P$ with $c 
\lessdot a$, $d \lessdot a$, $c \lessdot b$ and $d \lessdot b$.
In~\cite{McNamara2003}, McNamara showed that a bowtie-free poset $P$ with a $\hat{0}$ and a 
$\hat{1}$ has a local 
$H_n(0)$-action with the property that the characteristic of this action is $\omega(F_p(\xx))$ if and 
only if $P$ is snellable.   Additionally, he showed that if $P$ is a lattice, then $P$ is 
supersolvable.  Proposition~\ref{charOfQ} then   implies the following corollary.

\begin{corollary}\label{corollary:mcnamara}
Let $P$ be a poset with a generalized CW-labeling $\lambda$.  If $I$ is a maximal interval of 
$Q_\lambda(P)$ 
and is bowtie free, then $I$ is snellable.  Moreover, if $I$ is a lattice, then $I$ is 
supersolvable.
\end{corollary}

\section{Open questions and further work}
In this Section we leave a few open questions that are motivated by the present work. 
In Theorem \ref{theorem:WimpliesWhitney} we showed that posets that have Whitney labelings  
also have Whitney duals. It is reasonably to expect that there are posets without Whitney labelings  
that have Whitney duals. Indeed, the poset $\ISF_3$ has $\Pi_3$ as a Whitney dual. 
However $\ISF_3$ cannot have a Whitney labeling since in one of the maximal intervals the rank two 
switching property cannot be satisfied, see Figure \ref{figure:examplemobius}. We would like to 
know if there is a general characterization of graded posets that have Whitney duals that 
completely answers Question \ref{question:whitneydual1}.  Additionally, we would like to know if there are other
different and insightful methods of constructing Whitney duals, we propose the following question.
\begin{question}
 Is there a systematic way to construct a Whitney dual of $P$ without the use of labelings?
\end{question}

In the context of Whitney labelings we have provided two definitions: Whitney labelings and 
generalized Whitney labelings. Although the conditions of a generalized Whitney labeling are the 
ones we use in the proofs, all our examples satisfy the, a priori stronger, requirements of Whitney 
labelings.

\begin{question}
Are the families of Whitney labelable graded posets and of generalized Whitney labelable
graded posets equal?
\end{question}

Of main interest is also to better understand the structure of the posets $Q_{\lambda}(P)$ that are 
constructed using Whitney labelings $\lambda$ of $P$. We know from Theorem 
\ref{theorem:Qlambdanonisomorphic} that these posets are strongly dependent on $\lambda$ and in 
Theorem \ref{theorem:secondcaracterizationQlambda} we provide a different description of its 
poset structure.
\begin{question}
 Is there a nice way of characterizing all the posets that are of the form $Q_{\lambda}(P)$ for some 
poset $P$ and some Whitney labeling $\lambda$?
\end{question}

In light of Corollary \ref{corollary:mcnamara} determining the structural properties of the posets $Q_{\lambda}(P)$ also becomes relevant.

\begin{question}
Are all $Q_{\lambda}(P)$ lattices? If this is not the case, are all of them bowtie-free?
\end{question}

\section*{Acknowledgments}
The authors are extremely grateful to Michelle Wachs for the various stimulating discussions that led to the concept of Whitney duality and this project. The authors are also very thankful to Peter McNamara for many useful conversations.

\bibliographystyle{plain}
\bibliography{whitneydualref}
\end{document}